\documentclass[10pt]{amsart}

\usepackage{color}
\usepackage{amstext}
\usepackage{amsfonts}
\usepackage{amsthm}
\usepackage{amsmath}
\usepackage{amssymb}
\usepackage{latexsym}
\usepackage{amsfonts}
\usepackage{graphicx}
\usepackage{texdraw}
\usepackage{graphpap}
\usepackage{subfigure}
\usepackage{enumerate}
\usepackage{here}
\usepackage{tabularx}
\usepackage{bm}
\usepackage{changepage}
\usepackage{multicol}
\usepackage{bbold}
\usepackage{tikz}

\usetikzlibrary{arrows,shapes,positioning,shadows,trees}

\tikzset{
  basic/.style  = {draw, text width=10cm, drop shadow, rectangle},
  root/.style   = {basic, align=center, fill=white},
  level 2/.style = {basic, thin,align=center, fill=white,
                   text width=10cm},
  level 3/.style = {basic, thin, align=left, fill=white, text width=3.5cm},
  level 4/.style = {basic, thin, align=left, fill=white, text width=5cm},
}

\usepackage{mathtools}
\mathtoolsset{showonlyrefs,showmanualtags} 

\usepackage[pagebackref,hypertexnames=false, colorlinks, citecolor=red, linkcolor=red]{hyperref} 
\usepackage[backrefs]{amsrefs}

\makeatletter
\newcommand*{\defeq}{\mathrel{\rlap{%
                     \raisebox{0.3ex}{$\m@th\cdot$}}%
                     \raisebox{-0.3ex}{$\m@th\cdot$}}%
                     =}
										
\newcommand*{\eqdef}{=
										 \mathrel{\rlap{%
                     \raisebox{0.3ex}{$\m@th\cdot$}}%
                     \raisebox{-0.3ex}{$\m@th\cdot$}}%
										}
\makeatother

\usepackage{bbm}   
\newcommand{\unit}{\mathbbm 1} 

\newcommand{\mbr}{\mathbb{R}}
\newcommand{\mbz}{\mathbb{Z}}

\newcommand{\mcd}{\mathcal{D}}
\newcommand{\mcr}{\mathcal{R}}
\newcommand{\mch}{\mathcal{H}}
\newcommand{\mcx}{\mathcal{X}}

\newcommand{\bR}{\mathbf{R}}
\newcommand{\bP}{\mathbf{P}}
\newcommand{\bQ}{\mathbf{Q}}

\newcommand{\ep}{\epsilon}
\newcommand{\lb}{\lambda}

\newcommand{\dr}{\bm{\mathcal{D}}}

\newcommand{\vn}{\vec{n}}

\newcommand{\Pp}{\mathsf{P}_{\mathsf{b}}}
\newcommand{\pp}{\mathsf{p}_{\mathsf{b}}}

\newcommand{\La}{\left\langle }
\newcommand{\Ra}{\right\rangle }

\newcommand{\f}[2]{\ensuremath{\frac{#1}{#2}}}

\DeclareMathOperator*{\esssup}{ess\,sup}

\DeclareSymbolFont{bbold}{U}{bbold}{m}{n}
\DeclareSymbolFontAlphabet{\mathbbold}{bbold}
\newcommand{\mbs}{\mathbbm{S}}

{\end{list}}

\numberwithin{equation}{section}

\newtheorem{thm}{Theorem}[section]

\newtheorem{cor}[thm]{Corollary}

\newtheorem{prop}[thm]{Proposition}
\newtheorem*{prop*}{Proposition}

\theoremstyle{remark}
\newtheorem{rem}[thm]{Remark}
\newtheorem*{rem*}{Remark}

\begin{document}

\title[Weighted little bmo and two-weight inequalities for Journ\'{e} commutators]{Weighted little bmo and two-weight inequalities for Journ\'{e} commutators}

\author[Irina Holmes]{Irina Holmes$^{\dagger}$}
\address{Irina Holmes, Department of Mathematics\\ Michigan State University\\ 619 Red Cedar Road, Wells Hall\\ East Lansing, MI 48824 USA}
\email{holmesir@math.msu.edu}
\thanks{$\dagger$ This material is based upon work supported by the National Science Foundation under Award No. 1606270.}

\author[Stefanie Petermichl]{Stefanie Petermichl$^{\S}$}
\address{Stefanie Petermichl, Institut de Math\'{e}matiques de Toulouse\\ Universit\'{e} Paul Sabatier\\ 118 route de Narbonne\\ F-31062 Toulouse Cedex 9}
\thanks{$\S$ Research supported in part by ERC project CHRiSHarMa no. DLV-682402.}
\email{stefanie.petermichl@math.univ-toulouse.fr}

\author[Brett D. Wick]{Brett D. Wick$^{\ddagger}$}
\address{Brett D. Wick, Department of Mathematics\\ Washington University - St. Louis\\ One Brookings Drive\\
St. Louis, MO 63130-4899 USA}
\email{wick@math.wustl.edu}
\thanks{$\ddagger$  Research supported in part by a National Science Foundation DMS grant  \#1560955.}


\begin{abstract}
We characterize the boundedness of the commutators $[b, T]$ with biparameter Journ\'{e} operators $T$ in the two-weight, Bloom-type setting, and express the norms of these commutators in terms of a weighted little $bmo$ norm of the symbol $b$.
Specifically, if $\mu$ and $\lambda$ are biparameter $A_p$ weights, $\nu := \mu^{1/p}\lambda^{-1/p}$ is the Bloom weight, and 
$b$ is in $bmo(\nu)$, then we prove a lower bound and testing condition 
$\|b\|_{bmo(\nu)} \lesssim \sup \| [b, R_k^1 R_l^2]: L^p(\mu) \rightarrow L^p(\lambda)\|$, where
$R_k^1$ and $R_l^2$ are Riesz transforms acting in each variable.
Further, we prove that for such symbols $b$ and any biparameter Journ\'{e} operators $T$ the commutator $[b, T]:L^p(\mu) \rightarrow L^p(\lambda)$ is bounded.
Previous results in the Bloom setting do not include the biparameter case and are restricted to Calder\'{o}n-Zygmund operators. 
Even in the unweighted, $p=2$ case, the upper bound fills a gap that remained open in the multiparameter literature for iterated commutators with Journ\'e operators. 
As a by-product we also obtain a much simplified proof for a one-weight bound for Journ\'{e} operators originally due to R. Fefferman.
\end{abstract}

\maketitle
\setcounter{tocdepth}{1}
\tableofcontents

\section{Introduction and Statement of Main Results}

In 1985, Bloom \cite{Bloom} proved a two-weight version of the celebrated commutator theorem of Coifman, Rochberg and Weiss \cite{CRW}.
Specifically, \cite{Bloom} characterized the two-weight norm of the commutator $[b, H]$ with the Hilbert transform in terms of the norm of $b$
in a certain weighted BMO space:
	$$ \| [b, H] : L^p(\mu) \rightarrow L^p(\lambda)\| \simeq \|b\|_{BMO(\nu)}, $$
where $\mu, \lambda$ are $A_p$ weights, $1<p<\infty$, and $\nu := \mu^{1/p}\lambda^{-1/p}$. 
Recently, this was extended to the $n$-dimensional case of Calder\'{o}n-Zygmund operators in \cite{HLW2}, using the modern dyadic methods
started by \cite{P} and continued in \cite{HytRep}. The main idea in these methods is to represent continuous operators like the Hilbert transform
in terms of dyadic shift operators. 
This theory was recently extended to biparameter singular integrals in \cite{MRep}. 

In this paper we extend the Bloom theory to commutators with
biparameter Calder\'{o}n-Zygmund operators, also known as Journ\'{e} operators, and characterize their norms
in terms of a weighted version of the little bmo space of Cotlar and Sadosky \cite{CotlarSadosky}. 
The main results are:

\begin{thm}[Upper Bound] \label{T:UB}
Let $T$ be a biparameter Journ\'{e} operator on $\mbr^{\vn} = \mbr^{n_1}\otimes\mbr^{n_2}$, as defined in Section \ref{Ss:JourneDef}.
Let $\mu$ and $\lambda$ be $A_p(\mbr^{\vn})$ weights, $1<p<\infty$, and define $\nu := \mu^{1/p} \lambda^{-1/p}$. Then
	\begin{equation}
	\| [b, T] : L^p(\mu) \rightarrow L^p(\lambda) \| \lesssim \|b\|_{bmo(\nu)},
	\end{equation}
where $\|b\|_{bmo(\nu)}$ denotes the norm of $b$ in the weighted little $bmo(\nu)$ space on $\mbr^{\vn}$.
\end{thm}

We make a few remarks about the proof of this result. At its core, the strategy is the same as in \cite{HLW2},
and may be roughly stated as:
	\begin{enumerate}
	\item Use a representation theorem to reduce the problem from bounding the norm of $[b, T]$ to bounding the norm of 
		$[b, \text{Dyadic Shift}]$.
	\item Prove the two-weight bound for $[b, \text{Dyadic Shift}]$ by decomposing into paraproducts.
	\end{enumerate}
However, the biparameter case presents some significant new obstacles. In \cite{HLW2}, $T$ was a Calder\'{o}n-Zygmund operator on
$\mbr^n$, and the representation theorem was that of Hyt\"{o}nen \cite{HytRep}. In the present paper, $T$ is a biparameter Journ\'{e} operator on $\mbr^{\vn} = \mbr^{n_1} \otimes \mbr^{n_2}$ (see Section \ref{Ss:JourneDef}) and we use Martikainen's representation theorem \cite{MRep} to reduce the problem to commutators $[b, \mbs_{\dr}]$, where $\mbs_{\dr}$ is now a \textit{biparameter} dyadic shift. These can be cancellative, i.e. all Haar functions have mean zero, (defined in Section \ref{Ss:CShifts}), or non-cancellative (defined in Section \ref{Ss:NCShifts}). The strategy is summarized in Figure \ref{fig:M1}.

The main difficulty arises from the structure of the biparameter dyadic shifts.	At first glance, the cancellative shifts are ``almost'' compositions of two one-parameter shifts $\mbs_{\mcd_1}$ and $\mbs_{\mcd_2}$ applied in each variable -- if this were so, many of the results would follow trivially by iteration of the one-parameter results. Unfortunately, there is no reason for the coefficients $a_{P_1Q_1R_1P_2Q_2R_2}$ in the biparameter shifts to ``separate'' into a product $a_{P_1Q_1R_1} \cdot a_{P_2Q_2R_2}$,  as would be required in a composition of two one-parameter shifts. Therefore, many of the inequalities needed for biparameter shifts must be proved from scratch. 

Even more difficult is the case of non-cancellative shifts. As outlined in Section \ref{Ss:NCShifts}, these are really paraproducts, and there are three possible types that arise from the representation theorem: 
	\begin{enumerate}
	\item Full standard paraproducts;
	\item Full mixed paraproducts;
	\item Partial paraproducts.
	\end{enumerate}
These methods were considered previously in \cite{OPS} and \cite{OP} for the unweighted, $p = 2$ case. 
In \cite{OPS} it was shown that
	\begin{equation} \label{E:OPS} 
	\| [b, T] : L^2(\mbr^{\vn}) \rightarrow L^2(\mbr^{\vn}) \| \lesssim \|b\|_{bmo(\mbr^{\vn})},
	\end{equation}
where $T$ is a \textit{paraproduct-free} Journ\'{e} operator. This restriction essentially means that all the dyadic shifts in the representation of $T$ are \textit{cancellative}, so the case of non-cancellative shifts remained open. This gap was partially filled in \cite{OP}, which treats the case of non-cancellative shifts of standard paraproduct type. So the case of general Journ\'{e} operators, which includes non-cancellative shifts of mixed and partial type in the representation, remained open even in the unweighted, $p = 2$ case. These types of paraproducts are notoriously difficult -- see also \cite{MOrponen} for a wonderful discussion of this issue.
We fill this gap in Section \ref{Ss:NCShifts}, where we prove two-weight bounds of the type 
	$$ \|[b, \mbs_{\dr}] : L^p(\mu) \rightarrow L^p(\lambda) \| \lesssim \|b\|_{bmo(\nu)},$$
where $\mbs_{\dr}$ is a non-cancellative shift. The same is proved for cancellative shifts in Section \ref{Ss:CShifts}. 

\begin{figure}
\centering

\begin{tikzpicture}[
  level 1/.style={sibling distance=100mm},
  edge from parent/.style={<-,draw},
  >=latex]
\node[root] {$\| [b, T] : L^p(\mu) \rightarrow L^p(\lambda) \| \lesssim \|b\|_{bmo(\nu)}$}
  child [below=.9cm]{node[level 2] (c1)  
  {$\| [b, \mbs_{\dr}^{\vec{i},\vec{j}}] : L^p(\mu) \rightarrow L^p(\lambda) \| \lesssim \|b\|_{bmo(\nu)}$\\
  with at most polynomial bounds in $i, j$.}
  edge from parent node[left,draw=none] {Martikainen representation theorem}};
\begin{scope}
\node[level 3] [below = .5cm of c1] (c11) {Cancellative Shifts: \\ Theorem \ref{T:CShifts}};
\node[level 3] [right = 2cm of c11] (cr) {Two-weight bounds for paraproducts: \\ Section \ref{S:Para}};
\node[level 3] [below of = c11] (c12) {Non-Cancellative Shifts};
\node[level 4] [below right = of c12] (c121) {Full standard paraproduct:\\ Theorem \ref{T:NCS-1}};
\node[level 4] [below of = c121] (c122) {Full mixed paraproduct:\\ Theorem \ref{T:NCS-2}};
\node[level 4] [below of = c122] (c123) {Partial paraproduct: \\ Theorem \ref{T:NCS-3}};
\end{scope}

\draw[<-] (c1.195) |- (c11.west);
\draw[<-] (c1.195) |- (c12.west);
\draw[<-] (c12) |- (c121);
\draw[<-] (c12) |- (c122);
\draw[<-] (c12) |- (c123);
\draw[<-] (c11) -- (cr);
\draw[<-] (c121.east) -- + (1,0) |- (cr);
\draw[<-] (c122.east) -- + (1,0) |- (cr);
\draw[<-] (c123.east) -- + (1,0) |- (cr);

\end{tikzpicture}
\caption{Strategy for Theorem \ref{T:UB}} \label{fig:M1}
\end{figure}
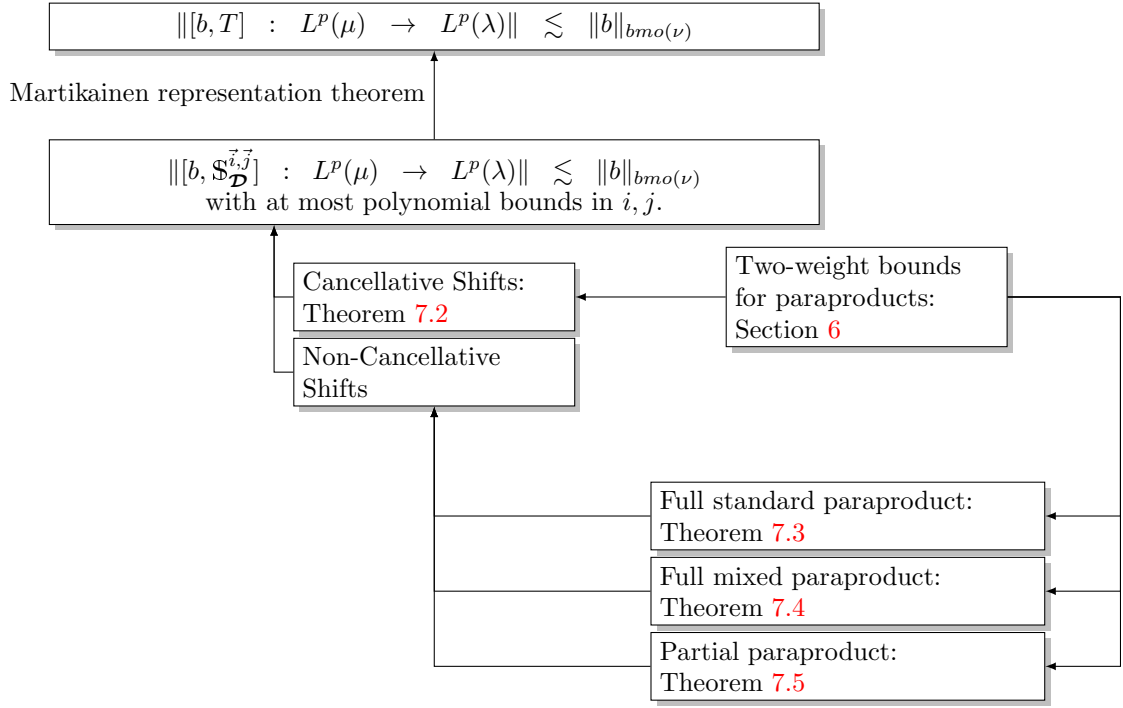

At the backbone of all these proofs will be the biparameter paraproducts, developed in Section \ref{S:Para}, and a variety of biparameter square functions, developed in Section \ref{S:BDSF}. 
For instance, in the case of the cancellative shifts, one can decompose the commutator as
	$$ [b, \mbs_{\dr}^{\vec{i}, \vec{j}}] f = \sum \: [\Pp, \mbs_{\dr}^{\vec{i}, \vec{j}}] f + \sum \: [\pp, \mbs_{\dr}^{\vec{i}, \vec{j}}] f + \mathcal{R}_{\vec{i},\vec{j}}f.$$
Here $\Pp$ runs through nine paraproducts associated with \textit{product BMO}, and $\pp$ runs through six paraproducts associated with \textit{little bmo}, so we are dealing with fifteen paraproducts in total in the biparameter case. 
Some of these are straightforward generalizations of the one-parameter paraproducts, while some are more complicated ``mixed'' paraproducts. Two-weight bounds are proved for all these paraproducts in Section \ref{S:Para}, building on two essential blocks: the biparameter square functions in Section \ref{S:BDSF}, and the weighted $H^1 - BMO$ duality in the product setting, developed in Section \ref{S:BWBMO}. In fact, Section \ref{S:BWBMO} is a self-contained presentation of large parts of the weighted biparameter BMO theory. 

Once the paraproducts are bounded, all that is left is to bound the so-called ``remainder term'' $\mathcal{R}_{\vec{i},\vec{j}}f$, of the form $\Pi_{\mbs f}b - \mbs \Pi_f b$, where one can no longer appeal directly to the paraproducts. At this point however, things become very technical, so bounding the remainder terms is no easy task. To help guide the reader, we outline below the general strategy we will employ. This applies to Theorem \ref{T:CShifts}, and in large part to Theorems \ref{T:NCS-1}, \ref{T:NCS-2}, and \ref{T:NCS-3}:
	\begin{enumerate}[1.]
	\item We break up the remainder term into more convenient sums of operators of the type $\mathcal{O}(b, f)$, involving both $b \in bmo(\nu)$ and $f \in L^p(\mu)$. We want to show $\|\mathcal{O}(b, f) : L^p(\mu) \rightarrow L^p(\lambda)\| \lesssim \|b\|_{bmo(\nu)}$. 
Using duality this amounts to showing that
	$$ |\La \mathcal{O}(b, f), g\Ra| \lesssim \|b\|_{BMO(\nu)} \|f\|_{L^p(\mu)} \|g\|_{L^{p'}(\lambda')}.$$
	\item Some of these operators $\mathcal{O}(b, f)$ involve full Haar coefficients $\widehat{b}(Q_1\times Q_2)$ of $b$, while others involve a Haar coefficient in one variable and averaging in the other variable, such as  $\La b, h_{Q_1} \times \unit_{Q_2}/|Q_2| \Ra$.
	Since, ultimately, we wish to use some type of $H^1-BMO$ duality, the goal will be to ``separate out'' $b$ from the inner product $\La \mathcal{O}(b, f), g\Ra$. 
	 If $\mathcal{O}(b, f)$ involves full Haar coefficients of $b$, we use duality with \textit{product BMO} and obtain 
		$$|\La \mathcal{O}(b, f), g\Ra| \lesssim \|b\|_{BMO(\nu)} \|S_{\dr}\phi(f, g)\|_{L^1(\nu)},$$
	where $\phi(f, g)$ is the operator we are left with after separating out $b$, and $S_{\dr}$ is the full biparameter dyadic square function. 
		 If $\mathcal{O}(b,f)$ involves terms of the form $\La b, h_{Q_1} \times \unit_{Q_2}/|Q_2| \Ra$, we use duality with \textit{little bmo}, and obtain something of the form
		$$|\La \mathcal{O}(b, f), g\Ra| \lesssim \|b\|_{bmo(\nu)} \|S_{\mcd_1}\phi(f, g)\|_{L^1(\nu)},$$
	where $S_{\mcd_1}$ is the dyadic square function in the first variable. Obviously this is replaced with $S_{\mcd_2}$ if the Haar coefficient on $b$ is in the second variable.
	\item Then the next goal is to show that 
		$$S_{\dr}\phi(f, g) \lesssim (\mathcal{O}_1 f) (\mathcal{O}_2 g),$$
	where $\mathcal{O}_{1,2}$ will be operators satisfying a \textit{one-weight} bound of the type $L^p(w) \rightarrow L^p(w)$. These operators will usually be a combination of the biparameter square functions in Section \ref{S:BDSF}. Once we have this, we are done.
	\end{enumerate}
	
In Theorem \ref{T:CShifts}, dealing with cancellative shifts, the crucial part is really step 1. At first glance, the remainder term
$\mathcal{R}_{\vec{i},\vec{j}}f$ seems intractable using this method, since it involves average terms $\La b\Ra_{Q_1\times Q_2}$ instead of Haar coefficients of $b$. So they key here is to decompose these terms in some convenient form.

In Section \ref{Ss:NCShifts}, dealing with non-cancellative shifts, the proofs follow this strategy in spirit, but deviate as we advance through the more and more difficult operators.
The main issue here is that we are are really dealing with terms of the form $|\La \mathcal{O}(a, b, f), g \Ra|$, where now the operator $\mathcal{O}$ involves a function $b$ in the \textit{weighted little} $bmo(\nu)$, and a function $a$ in \textit{unweighted product} BMO. In the most difficult case of partial paraproducts, $a$ is even more complicated, because it is essentially a \textit{sequence} of \textit{one-parameter} unweighted BMO functions. In all these cases, the creature $\phi$ in the last step is really $\phi(a, f, g)$. While in the previous case involving $\phi(f,g)$ it was straightforward to see the correct operators $\mathcal{O}_{1,2}$ to achieve step 5, in this case nothing straightforward seems to work.

 There are two key new ideas in these cases: one is to combine the cumbersome remainder term with a cleverly chosen third term, which will make the decompositions easier to handle. The other is to temporarily employ martingale transforms -- which works for us because this does not increase the BMO norms. We briefly describe the three situations below. As above, we will be rather non-rigorous about the notations in this expository section. There is plenty of notation later, and the purpose here is just to explain the main ideas and guide the reader through the technical proofs in Section \ref{Ss:NCShifts}.
 	\begin{enumerate}[1.]
	\item \textit{The full standard paraproduct} -- Theorem \ref{T:NCS-1}. This case only requires simple martingale transforms ($a_{\tau}$ and $g_{\tau}$, which have all non-negative Haar coefficients), and otherwise follows the strategy outlined above. However, we already start to see the operators $\mathcal{O}_{1,2}$ becoming strange compositions of ``standard'' operators and unweighted paraproducts, such as
		$$ S_{\dr}\phi \leq (M_S \Pi^*_{a_{\tau}} g_{\tau}) (S_{\dr}f).$$
		
	\item \textit{The full mixed paraproduct} -- Theorem \ref{T:NCS-2}. Here we introduce the idea of combining the remainder term 
	$\Pi_{\mbs f}b - \mbs \Pi_f b$ with a third term $T$, and we analyze $(\Pi_{\mbs f}b -T)$ and $(T - \mbs \Pi_f b)$ separately. 		This allows us to express the remainder as 
		$$\sum [\mathsf{P}_{\mathsf{a}}, \mathsf{p}_{\mathsf{b}}]f + T_{a,b}^{(1,0)}f - T_{a,b}^{(0,1)}f,$$
	a sum of \textit{commutators of paraproduct operators}, and a new remainder term.
	The new remainder has no cancellation properties, so we prove separately that the $T_{a,b}$ operators satisfy
		$$|\La T_{a,b}f, g\Ra| \lesssim \|b\|_{bmo(\nu)}\|f\|_{L^p(\mu)}\|g\|_{L^{p'}(\lambda')}.$$
	Here is where we employ the strategy outlined earlier, combined with a martingale transform $a_{\tau}$ applied to $a$. Interestingly, this transform depends on the particular argument $f$ of $[b, \mbs_{\dr}]f$. This will be absorbed in the end by the $BMO$ norm of the symbol for $\mbs_{\dr}$, so ultimately the choice of $f$ will not matter.
	
	\item \textit{The partial paraproducts} -- Theorem \ref{T:NCS-3}. Here we again combine the remainder terms with a third term $T$, and this time end up with terms of the form $\mathsf{p}_{\mathsf{b}} F$, where $F$ is a term depending on $a$ and $f$. So we are done if we can show that $\|F\|_{L^p(\mu)} \leq \|f\|_{L^p(\mu)}$. Without getting too technical about the notations, we reiterate that here $a$ is not \textit{one function} but rather a \textit{sequence} $a_{PQR}$ of one-parameter unweighted BMO functions. So the difficulty here is that the inner products look something like
		$$\La F, g\Ra = \sum \La \Pi^*_{a_{PQR}}\widetilde{f}, \widetilde{g}\Ra,$$
	where each summand has its own BMO function! The trick is then to write this as $\sum \La a_{PQR}, \phi_{PQR}(f, g)\Ra$. The happy ending is that these functions $a_{PQR}$ have uniformly bounded BMO norms, so at this point we apply unweighted one-parameter $H^1-BMO$ duality and we are left to work with $\|S_{\mcd}\phi(f, g)\|_{L^1(\mbr^n)}$; this is manageable. In one case, we do have to work with $F_{\tau}$ instead, which is again obtained by applying martingale transforms chosen in terms of $f$ -- only this time to each function $a_{PQR}$.
	\end{enumerate}

Finally, we see no reason why this result cannot be generalized to $k$-parameter Journ\'{e} operators. The main trouble in such a generalization should be strictly computational, as the number of paraproducts will blow up. 

In section \ref{S:mix} we recall the definition of the mixed ${\text{BMO}}_{\mathcal{I}}$ classes in between Chang-Fefferman's product BMO and Cotlar-Sadosky's little BMO. In the same way as in \cite{OPS} we deduce a corollary from Theorem \ref{T:UB}:

\begin{thm}[Upper bound, iterated, unweighted case]\label{upperbd_all_journe}
Let us consider $\mathbb{R}^{\vec{d}}$, $\vec{d}=(d_1,\ldots ,d_t)$ with a partition $\mathcal{I}=(I_s)_{1\le s \le l}$ of $\{1,\ldots ,t\}$. Let $b\in {\text{BMO}}_{\mathcal{I}}(\mathbb{R}^{\vec{d}})$ and let $T_s$ denote a multi-parameter Journ\'e operator acting on function defined on $\bigotimes_{k\in I_s}\mathbb{R}^{d_k}$. Then we have the estimate 
\[
\|[T_1,\ldots[T_l,b]\ldots]\|_{L^p(\mathbb{R}^{\vec{d}})\to L^p(\mathbb{R}^{\vec{d}})}\lesssim \|b\|_{{\text{BMO}}_{\mathcal{I}}(\mathbb{R}^{\vec{d}})}.
\]
\end{thm}

Coming back to the Bloom setting, we prove the lower estimate below, via a modification of the unweighted one-parameter argument of Coifman-Rochberg-Weiss. 

\begin{thm}[Lower Bound] \label{T:LB}
Let $\mu, \lb$ be $A_p(\mbr^n\times\mbr^n)$ weights, and set $\nu = \mu^{1/p}\lb^{-1/p}$. Then
\begin{equation}
\|b\|_{bmo(\nu)}
   \lesssim \sup_{1 \leqslant k, l \leqslant n} \| [b, R^1_k R^2_l] \|_{L^p
   ({\mu}) \rightarrow L^p (\lambda) .},
\end{equation}
where $R_{k}^1$ and $R_l^2$ are the Riesz transforms acting in
the first and second variable, respectively.
\end{thm}

This lower estimate allows us to see the tensor products of Riesz transforms as a representative testing class for all Journ\'e operators. 

We point out that in our quest to prove Theorem \ref{T:UB}, we also obtain a much simplified proof of the following one-weight result for Journ\'{e} operators, originally due to R. Fefferman:

\begin{thm}[Weighted Inequality for Journ\'{e} Operators] \label{T:Journe}
Let $T$ be a biparameter Journ\'{e} operator on $\mbr^{\vn} = \mbr^{n_1}\otimes\mbr^{n_2}$. Then $T$ is bounded
$L^p(w) \rightarrow L^p(w)$ for all $w\in A_p(\mbr^{\vn})$, $1<p<\infty$.
\end{thm}

A version of Theorem \ref{T:Journe} first appeared in R. Fefferman and E. M. Stein \cite{RStein}, with restrictive assumptions on the kernel.
Subsequently the kernel assumptions were weakened significantly by R. Fefferman in \cite{RF}, at the cost of assuming
the weight belongs to the more restrictive class $A_{p/2}$. This was due to the use of his sharp function 
$T^{\#}f = M_S (f^2)^{1/2}$, where $M_S$ is strong maximal function. Finally, R. Fefferman improved his own result in \cite{RF2}, where he showed that the $A_p$ class sufficed and obtained the full statement of Theorem \ref{T:Journe}. This was achieved by an involved bootstrapping argument based on his previous result \cite{RF}.

Our proof in Section \ref{Ss:FeffProof} of Theorem \ref{T:Journe} is significantly simpler. This may seem like a ``rough sell'' in light of the many pages of highly technical calculations that precede it. However, our proof of Theorem \ref{Ss:FeffProof} is only based on one-weight bounds for the biparameter dyadic shifts, of the form
	\begin{equation}\label{E:1wtDS}
	\| \mbs_{\dr}^{\vec{i}, \vec{j}} : L^p(w) \rightarrow L^p(w) \| \lesssim 1.
	\end{equation}
These had to be proved along the way, as part of our proof of the  two-weight upper bound for commutators, Theorem \ref{T:UB}.
These one-weight bounds are useful in themselves, and their proofs are not that long:
the proof for cancellative shifts, given in \eqref{E:2pDShift1wt}, is easy, and the proof for the non-cancellative shifts of partial paraproduct type is given in Proposition \ref{P:NCS-3-1wt}. Once we have \eqref{E:1wtDS}, the proof of Theorem \ref{T:Journe} follows immediately from Martikainen's representation theorem -- just as in the one-parameter case, a weighted bound for Calder\'{o}n-Zygmund operators follows trivially from Hyt\"{o}nen's representation theorem, once one has the one-weight bounds for the one-parameter dyadic shifts.

The paper is organized as follows. In Section \ref{S:BN} we review the necessary background, both one- and bi-parameter, and set up the notation.
In Section \ref{S:BDSF} we set up the types of dyadic square functions we will need throughout the rest of the paper. In Section \ref{S:BWBMO}, we
discuss the weighted and Bloom BMO spaces in the biparameter setting, and use some of these results in Section \ref{S:LB} to prove the lower bound result.
Section \ref{S:Para} is dedicated to biparameter paraproducts, which will be crucial in the final Section \ref{S:UB}, which proves the upper bound by 
an appeal to Martikainen's \cite{MRep} representation theorem. Finally, we prove Theorem \ref{T:Journe}.

\section{Background and Notation}
\label{S:BN}

In this section we review some of the basic building blocks of one-parameter dyadic harmonic analysis on $\mbr^n$,
 followed by their biparameter versions for $\mbr^{\vn} := \mbr^{n_1} \otimes \mbr^{n_2}$.

\subsection{Dyadic Grids on $\mbr^n$} 
Let $\mcd_0 := \{ 2^{-k}([0,1)^n + m) : k \in \mbz, m \in \mbz^n\}$ denote the standard dyadic grid on $\mbr^n$.
For every $\omega = (\omega_j)_{j \in \mbz} \in (\{0, 1\}^n)^{\mbz}$ define the shifted dyadic grid $\mcd_{\omega}$:
	$$ \mcd_{\omega} := \{ Q \stackrel{\cdot}{+} \omega : Q \in \mcd_0\} \text{, where }
	Q \stackrel{\cdot}{+} \omega := Q + \sum_{j: 2^{-j}<l(Q)} 2^{-j}\omega_j, $$
and $l(Q)$ denotes the side length of a cube $Q$. 
The indexing parameter $\omega$ is rarely relevant in what follows: it only appears when we are dealing with 
$\mathbb{E}_{\omega}$ -- expectation with respect to the standard probability measure on the space
of parameters $\omega$. 
In fact, an important feature of the (by now standard) methods we employ in this paper is obtaining upper bounds
for dyadic operators that are independent of the choice of dyadic grid. The focus therefore is on the geometrical
properties shared by any dyadic grid $\mcd$ on $\mbr^n$:
	\begin{itemize}
	\item $P \cap Q \in \{ P, Q, \emptyset \}$ for every $P, Q \in \mcd$.
	\item The cubes $Q \in \mcd$ with $l(Q) = 2^{-k}$, for some fixed integer $k$, partition $\mbr^n$.
	\end{itemize}
For every $Q \in \mcd$ and every non-negative integer $k$ we define:
	\begin{itemize}
	\item $Q^{(k)}$ -- the $k^{th}$ generation ancestor of $Q$ in $\mcd$, i.e. 
		the unique element of $\mcd$ which contains $Q$ and has side length $ 2^k l(Q)$.
	\item $(Q)_k$ -- the collection of $k^{th}$ generation descendants of $Q$ in $\mcd$, i.e.
		the $2^{kn}$ disjoint subcubes of $Q$ with side length $2^{-k} l(Q)$.
	\end{itemize}
	
\subsection{The Haar system on $\mbr^n$} 
Recall that every dyadic interval $I$ in $\mbr$ is associated with two Haar functions:
	$$ h_I^{0} := \frac{1}{\sqrt{|I|}} (\unit_{I-} - \unit_{I_+}) \text{ and } h_I^1 := \frac{1}{\sqrt{|I|}}\unit_I, $$
the first one being cancellative (it has mean $0$).
Given a dyadic grid $\mcd$ on $\mbr^{n}$, every dyadic cube $Q = I_1 \times\ldots\times I_n$, where all $I_i$ are dyadic intervals in $\mbr$ with common length $l(Q)$,
is associated with $2^n-1$ cancellative Haar functions:
	$$ h^\ep_{Q}(x) := h_{I_1\times\ldots\times I_n}^{(\ep_1, \ldots, \ep_n)}(x_1, \ldots, x_n) := \prod_{i=1}^n h_{I_i}^{\ep_i}(x_i),$$
where $\ep \in \{0,1\}^n\setminus \{(1, \ldots, 1)\}$ is the signature of $h_Q^{\ep}$. To simplify notation, we assume that signatures are never the identically $1$ signature, in which case
the corresponding Haar function would be non-cancellative. 
The cancellative Haar functions form an orthonormal basis for $L^2(\mbr^n)$. We write
	$$ f = \sum_{Q\in\mcd} \widehat{f}(Q^\ep) h_Q^\ep, $$
where $\widehat{f}(Q^\ep) := \La f, h_Q^\ep\Ra$, $\La f, g\Ra := \int_{\mbr^n} fg\,dx$, and summation over $\ep$ is assumed. 
We list here some other useful facts which will come in handy later:
	\begin{itemize}
	\item $h_P^{\ep}(x)$ is constant on any subcube $Q \in \mcd$, $Q \subsetneq P$. We denote this value by $h_P^{\ep}(Q)$.
	\item The average of $f$ over a cube $Q \in \mcd$ may be expressed as:
		\begin{equation} \label{E:1pavg} 
		\La f\Ra_Q = \sum_{P \in \mcd, P \supsetneq Q} \widehat{f}(P^{\ep}) h_P^{\ep}(Q). 
		\end{equation}
	\item Then, if $Q \subsetneq R \in \mcd$:
		\begin{equation} \label{E:1pavgdiff} 
		\La f\Ra_Q - \La f\Ra_R = \sum_{P \in \mcd: Q\subsetneq P \subset R} \widehat{f}(P^{\ep}) h_P^{\ep}(Q). 
		\end{equation}
	\item For $Q\in\mcd$:
		\begin{equation} \label{E:1punitmo}
		\unit_Q (f - \La f\Ra_Q) = \sum_{P \in\mcd: P \subset Q} \widehat{f}(P^{\ep})h_P^{\ep}.
		\end{equation}
	\item For two \textit{distinct} signatures $\ep\neq\delta$, define the signature $\ep+\delta$ by letting $(\ep+\delta)_i$ be $1$ if $\ep_i = \delta_i$ and $0$ otherwise. 
	Note that $\ep+\delta$ is distinct from both $\ep$ and $\delta$, and is not the identically $\vec{1}$ signature. Then
	$$ h_Q^\ep h_Q^\delta = \frac{1}{\sqrt{Q}}h_Q^{\ep+\delta} \text{, if } \ep\neq\delta \text{, and } h_Q^\ep h_Q^\ep = \frac{\unit_Q}{|Q|}. $$
	Again to simplify notation, we assume throughout this paper that we only write $h_Q^{\ep+\delta}$ for \textit{distinct} signatures $\ep$ and $\delta$.
	\end{itemize}
Given a dyadic grid $\mcd$, we define the dyadic square function on $\mbr^n$ by:
	$$ S_{\mcd}f(x) := \bigg( \sum_{Q \in \mcd} |\widehat{f}(Q^{\ep})|^2 \frac{\unit_Q(x)}{|Q|} \bigg)^{1/2}. $$
Then $\|f\|_p \simeq \|S_{\mcd}f\|_p$ for all $1 < p < \infty$. We also define the dyadic version of the maximal function:
	$$ M_{\mcd}f(x) = \sup_{Q \in \mcd} \La |f|\Ra_Q \unit_Q(x). $$

\subsection{$A_p(\mbr^n)$ Weights}
Let $w$ be a weight on $\mbr^n$, i.e. $w$ is an almost everywhere positive, locally integrable function. 
For $1 < p < \infty$, let $L^p(w) \defeq L^p(\mbr^n; w(x)\,dx)$. For a cube $Q$ in $\mbr^n$, we let 
	$$w(Q) := \int_Q w(x)\,dx \text{ and } \left<w\right>_Q := \frac{w(Q)}{|Q|}.$$
We say that $w$ belongs to the Muckenhoupt $A_p(\mbr^n)$ class provided that:
	$$[w]_{A_p} := \sup_{Q} \left<w\right>_Q \left<w^{1-p'}\right>_Q^{p-1} < \infty,$$
where $p'$ denotes the H\"older conjugate of $p$ and the supremum above is over all cubes $Q$ in $\mbr^n$ with sides parallel to the axes. 
The weight $w' := w^{1-p'}$ is sometimes called the weight ``conjugate'' to $w$, because 
$w \in A_p$ if and only if $w' \in A_{p'}$.	

We recall the classical inequalities for the maximal  and square functions:
	$$\|M f\|_{L^p(w)} \lesssim \|f\|_{L^p(w)} \text{ and } \|f\|_{L^p(w)} \simeq \|S_{\mcd}f\|_{L^p(w)},$$
for all $w \in A_p(\mbr^n)$, $1 < p < \infty$, where throughout this paper ``$A \lesssim B$'' denotes $A \leq cB$ for some
constant $c$ which may depend on the dimensions and the weight $w$. 
In dealing with dyadic shifts, we will also need to consider the following shifted dyadic square function: given non-negative integers $i$ and $j$, define
	$$ S_{\mcd}^{i, j}f(x) := \bigg[ \sum_{R\in\mcd} \Big( \sum_{P \in (R)_i} |\widehat{f}(P^\ep)| \Big)^2 \Big( \sum_{Q\in (R)_j} \frac{\unit_Q(x)}{|Q|}\Big) \bigg]^{1/2}. $$
It was shown in \cite{HLW2} that
	\begin{equation} \label{E:ShiftedDSF1p}
	\| S_{\mcd}^{i,j} : L^p(w) \rightarrow L^p(w) \| \lesssim 2^{\frac{n}{2}(i+j)},
	\end{equation}	
for all $w \in A_p(\mbr^n)$, $1 < p < \infty$.

A \textit{martingale transform} on $\mbr^n$ is an operator of the form
	$$ f \mapsto f_{\tau} := \sum_{P\in\mcd} \tau_P^{\ep} \widehat{f}(P^{\ep}) h_P^{\ep}, $$
where each $\tau_P^{\ep}$ is either $+1$ or $-1$. Obviously $S_{\mcd}f = S_{\mcd}f_{\tau}$, 
so one can work with $f_{\tau}$ instead when convenient, without increasing the $L^p(w)$-norm of $f$.

\subsection{The Haar system on $\mbr^{\vn}$}
In $\mbr^{\vn} := \mbr^{n_1} \otimes \mbr^{n_2}$, we work with dyadic rectangles 
	$$\dr := \mcd_1 \times \mcd_2 = \{ R = Q_1 \times Q_2 : Q_i \in \mcd_i \},$$
where each $\mcd_i$ is a dyadic grid on $\mbr^{n_i}$. 
While we unfortunately lose the nice nestedness and partitioning properties of
one-parameter dyadic grids, we do have the tensor product Haar wavelet orthonormal basis
for $L^2(\mbr^{\vn})$, defined by
	$$ h_R^{\vec{\ep}} (x_1, x_2) := h_{Q_1}^{\ep_1}(x_1) \otimes h_{Q_2}^{\ep_2}(x_2), $$
for all $R = Q_1 \times Q_2 \in \dr$ and $\vec{\ep} = (\ep_1, \ep_2)$. 
We often write
	$$ f = \sum_{Q_1\times Q_2} \widehat{f}(Q_1^{\ep_1}\times Q_2^{\ep_2}) h_{Q_1}^{\ep_1} \otimes h_{Q_2}^{\ep_2},$$
short for summing over $Q_1 \in \mcd_1$ and $Q_2 \in \mcd_2$, and of course over all signatures, where
	$$\widehat{f}(Q_1^{\ep_1}\times Q_2^{\ep_2}) := \left< f, h_{Q_1}^{\ep_1} \otimes h_{Q_2}^{\ep_2} \right> = 
		\int_{\mbr^{\vn}} f(x_1, x_2) h_{Q_1}^{\ep_1}(x_1) h_{Q_2}^{\ep_2}(x_2)\,dx_1\,dx_2 .$$

While the averaging formula \eqref{E:1pavg} has a straightforward biparameter analogue:
	\begin{equation} \label{E:2pavg}
	\La f\Ra_{Q_1\times Q_2} = \sum_{P_1 \supsetneq Q_1; \: P_2 \supsetneq Q_2} 
		\widehat{f} (P_1^{\ep_1}\times P_2^{\ep_2}) h_{P_1}^{\ep_1}(Q_1)  h_{P_2}^{\ep_2}(Q_2),
	\end{equation}
the expression  in \eqref{E:1punitmo} takes a slightly messier form in two parameters: for any $R = Q_1 \times Q_2$
	\begin{align} 
	\unit_{R} (f - \La f\Ra_{R}) &= 
		\sum_{\substack{P_1\subset Q_1 \\ P_2 \subset Q_2}} \widehat{f} (P_1^{\ep_1}\times P_2^{\ep_2}) h_{P_1}^{\ep_1} \otimes h_{P_2}^{\ep_2} &  \nonumber \\
		& 	+  \sum_{P_2\subset Q_2} \La f, \frac{\unit_{Q_1}}{|Q_1|} \otimes h_{P_2}^{\ep_2} \Ra \unit_{Q_1} \otimes h_{P_2}^{\ep_2} 
		+ \sum_{P_1\subset Q_1} \La f, h_{P_1}^{\ep_1} \otimes \frac{\unit_{Q_2}}{|Q_2|} \Ra h_{P_1}^{\ep_1}\otimes \unit_{Q_2} & \label{E:2punitmo} \\
		&= \sum_{\substack{P_1\subset Q_1 \\ P_2 \subset Q_2}} \widehat{f} (P_1^{\ep_1}\times P_2^{\ep_2}) h_{P_1}^{\ep_1} \otimes h_{P_2}^{\ep_2}
		+ \unit_R [m_{Q_1}f(x_2) - \La f\Ra_R] + \unit_R [m_{Q_2}f(x_1) - \La f\Ra_R], &\nonumber
	\end{align}	
where for any cubes $Q_i \in \mcd_i$:
	\begin{equation} \label{E:mQnotation} 
	m_{Q_1}f(x_2) := \frac{1}{|Q_1|}\int_{Q_1} f(x_1, x_2)\,dx_1 \text{, and } 
		m_{Q_2}f(x_1) := \frac{1}{|Q_2|} \int_{Q_2} f(x_1, x_2)\,dx_2.
	\end{equation}
As we shall see later, this particular expression will be quite relevant for biparameter BMO spaces.

\subsection{$A_p(\mbr^{\vn})$ Weights}
A weight $w(x_1, x_2)$ on $\mbr^{\vn}$ belongs to the class $A_p(\mbr^{\vn})$, for some $1 < p < \infty$, provided that
	$$ [w]_{A_p} := \sup_R \La w\Ra_R \La w^{1-p'}\Ra_R^{p-1} < \infty, $$
where the supremum is over all \textit{rectangles} $R$. 
These are the weights which characterize $L^p(w)$ boundedness of the \textit{strong} maximal function:
	$$ M_Sf(x_1, x_2) := \sup_R \La |f|\Ra_R \unit_R(x_1, x_2), $$
where the supremum is again over all rectangles. As is well-known, the usual weak $(1, 1)$ inequality fails for the
strong maximal function, where it is replaced by an Orlicz norm expression. In the weighted case, we have  \cite{BagbyKurtz}
for all $w \in A_p(\mbr^{\vn})$:
	\begin{equation} \label{E:LlogLMS}
	w\{ x \in \mbr^{\vn}: M_Sf(x) > \lambda \} \lesssim \int_{\mbr^{\vn}} \left( \frac{|f(x)|}{\lambda} \right)^p \left( 1 + \log^{+} \frac{|f(x)|}{\lambda} \right)^{k-1}\,dw(x).
	\end{equation}

Moreover, $w$ belongs to $A_p(\mbr^{\vn})$ if and only if $w$ belongs to the \textit{one-parameter}
classes $A_p(\mbr^{n_i})$ in each variable separately and uniformly:
	$$ [w]_{A_p(\mbr^{\vn})} \simeq \max \bigg\{
		\esssup_{x_1 \in \mbr^{n_1}} [w(x_1, \cdot)]_{A_p(\mbr^{n_2})}, \:\:\:
				\esssup_{x_2 \in \mbr^{n_2}} [w(\cdot, x_2)]_{A_p(\mbr^{n_1})} \bigg\}.$$  
It also follows as in the one-parameter case that $w \in A_p(\mbr^{\vn})$
if and only if $w' := w^{1-p'} \in A_{p'}(\mbr^{\vn})$, and
$L^p(w)^* \simeq L^{p'}(w')$, in the sense that:
	\begin{equation} \label{E:ApDuality}
	\|f\|_{L^p(w)} = \sup\{ |\La f, g\Ra|: g \in L^{p'}(w'), \|g\|_{L^{p'}(w')} \leq 1 \}.
	\end{equation}
	
We may also define weights $m_{Q_1}w$ and $m_{Q_2}w$ on $\mbr^{n_2}$ and $\mbr^{n_1}$, respectively, as in \eqref{E:mQnotation}. As shown below, these are then also uniformly in their respective one-parameter $A_p$ classes:
	
\begin{prop}\label{P:avg2ParAp}
If $w \in A_p(\mbr^{\vn})$, $1<p<\infty$, then $m_{Q_1}w \in A_p(\mbr^{n_2})$ and $m_{Q_2}w \in A_p(\mbr^{n_1})$ for any cubes $Q_i \subset \mbr^{n_i}$, with uniformly bounded $A_p$ constants:
	$$ [m_{Q_i}w]_{A_p(\mbr^{n_j})} \leq [w]_{A_p(\mbr^{\vn})}, $$
for all $Q_i \subset \mbr^{n_i}$, $i \in \{1,2\}$, $i \neq j$.
\end{prop}

\begin{proof}
Fix a cube $Q_1 \subset \mbr^{n_1}$. Then for every $x_2 \in \mbr^{n_2}$,
	$$ |Q_1| = \int_{Q_1} 1\,dx_1 \leq \bigg( \int_{Q_1} w(x_1, x_2)\,dx_1\bigg)^{1/p} 
		\bigg( \int_{Q_1} w'(x_1, x_2)\,dx_1\bigg)^{1/p'}, $$
and so
	$$ ( m_{Q_1}w)'(x_2) := ( m_{Q_1}w)^{1-p'}(x_2) \leq m_{Q_1}w'(x_2).$$
Then for all cubes $Q_2 \subset \mbr^{n_2}$,
	$$ \La m_{Q_1}w\Ra_{Q_2} \La (m_{Q_1}w)'\Ra_{Q_2}^{p-1} \leq \La w\Ra_{Q_1\times Q_2} \La w'\Ra_{Q_1\times Q_2}^{p-1}
		\leq [w]_{A^p(\mbr^{\vn})}, $$
proving the result for $m_{Q_1}w$. The other case follows symmetrically. 
\end{proof}

Finally, we will later use a reverse H\"{o}lder property of biparameter $A_p$ weights. This is well-known to experts, but we include a proof here for completeness.

\begin{prop}\label{P:2ParRH}
If $w \in A_p(\mbr^{\vn})$, then there exist positive constants $C, \epsilon, \delta > 0$ (depending only on $\vn$, $p$,
and $[w]_{A_p(\mbr^{\vn})}$), such that 
	\begin{enumerate}[i).]
	\item For all rectangles $R \subset \mbr^{\vn}$,
		$$ \bigg( \frac{1}{|R|} \int_R w(x)^{1+\epsilon}\,dx\bigg)^{\frac{1}{1+\epsilon}} \leq
			\frac{C}{|R|}\int_R w(x)\,dx. $$
	\item For all rectangles $R \subset \mbr^{\vn}$ and all measurable subsets $E\subset R$,
		$$ \frac{w(E)}{w(R)} \leq C \bigg( \frac{|E|}{|R|}\bigg)^{\delta}. $$
	\end{enumerate}
\end{prop}

\begin{proof}
Note first that \textit{ii).} follows easily from \textit{i).} by applying the H\"{o}lder inequality with exponents $1+\ep$
and $\frac{1+\ep}{\ep}$ in $w(E) = \int_E w(x)\,dx$. This gives \textit{ii).} with $\delta = \frac{\ep}{1+\ep}$. 

In order to prove \textit{i).} we first recall a more general statement of the one-parameter reverse H\"{o}lder property of $A_p$ weights (see Remark 9.2.3 in \cite{Grafakos}):

	\setlength{\leftskip}{1cm}
	\noindent\textit{For any $1<p<\infty$ and $B > 1$, there exist positive constants}
		\begin{equation}\label{E:tempStar}
		D = D(n, p, B) \text{ and } \beta = \beta(n, p, B)
		\end{equation}
	\textit{such that for all $v \in A_p(\mbr^n)$ with $[v]_{A_p(\mbr^{\vn})} \leq B$, the reverse H\"{o}lder condition}
		\begin{equation}\label{E:tempStarStar} 
		\bigg( \frac{1}{|Q|} \int_Q v(t)^{1+\beta}\,dt\bigg)^{\frac{1}{1+\beta}} \leq
			\frac{D}{|Q|}\int_Q v(t)\,dt. 
		\end{equation}
	\textit{holds for all cubes $Q\subset \mbr^n$.}
	
	\setlength{\leftskip}{0pt}

\noindent It is easy to see that if a weight $v$ satisfies the reverse H\"{o}lder condition \eqref{E:tempStarStar} with constants 
$D, \beta$, then it also satisfies it with any constants $C, \epsilon$ with $C \geq D$ and $\ep \leq \beta$.

Now let $w\in A_p(\mbr^{\vn})$, set $B := [w]_{A_p(\mbr^{\vn})}$, and for $i \in \{1, 2\}$ let
	$ D_i := D(n_i, p, B)$  and $\beta_i := \beta(n_i, p, B) $
be as in \eqref{E:tempStar}. Fix a rectangle $R = Q_1\times Q_2$, a measurable subset $E\subset R$, and set
	$$ C^2 := \max(D_1, D_2) \text{ and } \ep:= \min(\beta_1, \beta_2). $$
For almost all $x_1 \in \mbr^{n_1}$, the weight $w(x_1, \cdot) \in A_p(\mbr^{n_2})$ with $[w(x_1, \cdot)]_{A_p(\mbr^{n_2})} \leq B$,
so $w(x_1, \cdot)$ satisfies reverse H\"{o}lder with constants $D_2, \beta_2$ -- and therefore also with constants $\sqrt{C}, \ep$. So
	\begin{align}
	\frac{1}{|R|} \int_R w(x)^{1+\ep}\,dx &= \frac{1}{|Q_1|} \int_{Q_1} 
		\bigg(\frac{1}{|Q_2|} w(x_1, x_2)^{1+\ep}\,dx_2\bigg)\,dx_1\\
	& \leq \frac{1}{|Q_1|} \int_{Q_1} \bigg( \frac{\sqrt{C}}{|Q_2|} \int_{Q_2} w(x_1, x_2)\,dx_2\bigg)^{1+\ep} \,dx_1\\
	& = \frac{C^{(1+\ep)/2}}{|Q_1|} \int_{Q_1} (m_{Q_2}w(x_1))^{1+\ep}\,dx_1.
	\end{align}
By Proposition \ref{P:avg2ParAp}, the weight $m_{Q_2}w\in A_p(\mbr^{n_1})$ with $[m_{Q_2}w]_{A_p(\mbr^{n_1})} \leq B$, so this weight satisfies reverse H\"{o}lder with constants $D_1, \beta_1$ -- and therefore also with constants $\sqrt{C}, \ep$. Then the last inequality above gives that
	$$\bigg( \frac{1}{|R|} \int_R w(x)^{1+\epsilon}\,dx\bigg)^{\frac{1}{1+\epsilon}} \leq \frac{C}{|Q_1|}\int_{Q_1} 
	m_{Q_2}w(x_1)\,dx_1 = \frac{C}{|R|}\int_R w(x)\,dx.$$
\end{proof}

\section{Biparameter Dyadic Square Functions}
\label{S:BDSF}

Throughout this section, fix dyadic rectangles $\dr := \mcd_1 \times \mcd_2$ on $\mbr^{\vn}$.  
The dyadic square function associated with $\dr$ is then defined in the obvious way:
	$$ S_{\dr}f(x_1, x_2) := \bigg( \sum_{R\in\dr} |\widehat{f}(R^{\vec{\ep}})|^2 \frac{\unit_R(x_1, x_2)}{|R|} \bigg)^{1/2}. $$
We also want to look at the dyadic square functions \textit{in each variable}, namely
	$$ S_{\mcd_1}f(x_1, x_2) := \bigg( \sum_{Q_1\in\mcd_1} |H_{Q_1}^{\ep_1}f(x_2)|^2 \frac{\unit_{Q_1}(x_1)}{|Q_1|} \bigg)^{1/2} 
	; \:\:\:
	S_{\mcd_2}f(x_1, x_2) := \bigg( \sum_{Q_2\in\mcd_2} |H_{Q_2}^{\ep_2}(x_1)|^2 \frac{\unit_{Q_2}(x_2)}{|Q_2|} \bigg)^2,$$
where for every $Q_i \in \mcd_i$ and signatures $\ep_i$, we denote
	$$ H_{Q_1}^{\ep_1} f(x_2) := \int_{\mbr^{n_1}} f(x_1, x_2) h_{Q_1}^{\ep_1}(x_1)\,dx_1 ; \:\:\:
	H_{Q_2}^{\ep_2} f(x_1) := \int_{\mbr^{n_2}} f(x_1, x_2) h_{Q_2}^{\ep_2}(x_2)\,dx_2. $$
Then for any $w \in A_p(\mbr^{\vn})$:
	$$ \|f\|_{L^p(w)} \simeq \|S_{\dr}f\|_{L^p(w)} \simeq \|S_{\mcd_1}f\|_{L^p(w)} \simeq \|S_{\mcd_2}f\|_{L^p(w)}. $$
	
More generally, define the shifted biparameter square function, for pairs $\vec{i} = (i_1, i_2)$ and $\vec{j} = (j_1, j_2)$ of non-negative integers, by:
	\begin{equation}\label{E:ShiftedDSF2pDef} 
	S_{\dr}^{\vec{i}, \vec{j}}f := \bigg[ \sum_{\substack{R_1\in\mcd_1 \\ R_2\in\mcd_2}}  
		\bigg( \sum_{\substack{P_1\in (R_1)_{i_1} \\ P_2\in (R_2)_{i_2}}} |\widehat{f}(P_1^{\ep_1}\times P_2^{\ep_2})| \bigg)^2
		\bigg( \sum_{\substack{Q_1\in (R_1)_{j_1} \\ Q_2\in (R_2)_{j_2}}} \frac{\unit_{Q_1}}{|Q_1|} \otimes \frac{\unit_{Q_2}}{|Q_2|}\bigg)  \bigg]^{1/2}. 
	\end{equation}
We claim that:
	\begin{equation} \label{E:ShiftedDSF2p}
	\| S_{\dr}^{\vec{i}, \vec{j}} : L^p(w) \rightarrow L^p(w) \| \lesssim 2^{\frac{n_1}{2}(i_1+j_1)} 2^{\frac{n_2}{2}(i_2+j_2)},
	\end{equation}
for all $w \in A_p(\mbr^{\vn})$, $1 < p < \infty$.
This follows by iteration of the one-parameter result in \eqref{E:ShiftedDSF1p}, through the following vector-valued version of the extrapolation theorem (see Corollary 9.5.7 in \cite{Grafakos}):

\begin{prop} \label{P:VvalExt}
Suppose that an operator $T$ satisfies $\|T: L^2(w) \rightarrow L^2(w)\| \leq AC_n[w]_{A_2}$ for all $w \in A_2(\mbr^n)$, for some constants $A$ and $C_n$, where the latter only depends on the dimension. Then:
	$$ \left\| \bigg( \sum_{j} |T f_j |^2 \bigg)^{1/2} \right\|_{L^p(w)} 
	\leq A C_n'[w]_{A_p}^{\max(1, \frac{1}{p-1})}  \left\| \bigg( \sum_j |f_j |^2 \bigg)^{1/2} \right\|_{L^p(w)},$$
for all $w\in A_p(\mbr^n)$, $1 < p < \infty$ and all sequences $\{f_j\}\subset L^p(w)$, where $C_n'$ is a dimensional constant.
\end{prop}

\begin{proof}[Proof of \eqref{E:ShiftedDSF2p}]
Note that $(S_{\dr}^{\vec{i}, \vec{j}} f)^2 = \sum_{R_1\in\mcd_1} (S_{\mcd_2}^{i_2, j_2} F_{R_1})^2$, where
	$$ F_{R_1}(x_1, x_2) := \sum_{P_2\in\mcd_2} \bigg( \sum_{P_1\in (R_1)_{i_1}} |\widehat{f}(P_1^{\ep_1}\times P_2^{\ep_2})| \bigg) 
		\bigg( \sum_{Q_1\in (R_1)_{j_1}} \frac{\unit_{Q_1}(x_1)}{|Q_1|}  \bigg)^{1/2}  h_{P_2}^{\ep_2}(x_2).$$
Then
	$$ \| S_{\dr}^{\vec{i}, \vec{j}} f\|^p_{L^p(w)} = \int_{\mbr^{n_1}} \int_{\mbr^{n_2}} 
		\bigg( \sum_{R_1\in\mcd_1} (S_{\mcd_2}^{i_2, j_2} F_{R_1}(x_1, x_2))^2 \bigg)^{p/2} w(x_1, x_2)\,dx_2\,dx_1.$$
For almost all fixed $x_1\in\mbr^{n_1}$, $w(x_1, \cdot)$ is in $A_p(\mbr^{n_2})$ uniformly, so we may apply Proposition \ref{P:VvalExt} and \eqref{E:ShiftedDSF1p} to the inner integral and obtain:
	$$  \| S_{\dr}^{\vec{i}, \vec{j}} f\|^p_{L^p(w)} \lesssim 2^{\frac{pn_2}{2}(i_2+j_2)}
		 \int_{\mbr^{n_1}} \int_{\mbr^{n_2}} \bigg( \sum_{R_1\in\mcd_1} |F_{R_1}(x_1, x_2)|^2 \bigg)^{p/2} w(x_1, x_2) \,dx_2\,dx_1. $$
Now, we can express the integral above as
	$$  \int_{\mbr^{n_2}} \int_{\mbr^{n_1}} \bigg( S_{\mcd_1}^{i_1, j_1}f_{\tau}(x_1, x_2) \bigg)^p w(x_1, x_2) \,dx_1\,dx_2
	 	\lesssim 2^{\frac{pn_1}{2}(i_1+j_1)} \|f_{\tau}\|^p,$$
where $f_{\tau} = \sum_{P_1\times P_2} |\widehat{f}(P_1^{\ep_1} \times P_2^{\ep_2})| h_{P_1}^{\ep_1} \otimes h_{P_2}^{\ep_2}$
 is just a biparameter martingale transform applied to $f$, and therefore $\|f\|_{L^p(w)}\simeq \|f_{\tau}\|_{L^p(w)}$ by passing to the square function.
\end{proof}

\subsection{Mixed Square and Maximal Functions} \label{Ss:MixedSquare}
We will later encounter mixed operators such as:
	$$ [SM]f(x_1, x_2) := \left( \sum_{Q_1\in\mcd_1} \left( M_{\mcd_2}(H_{Q_1}^{\ep_1}f)(x_2) \right)^2 \frac{\unit_{Q_1}(x_1)}{|Q_1|} \right)^{1/2}, $$
	$$ [MS]f(x_1, x_2) := \left( \sum_{Q_2\in\mcd_2} \left( M_{\mcd_1}(H_{Q_2}^{\ep_2}f)(x_1) \right)^2 \frac{\unit_{Q_2}(x_2)}{|Q_2|} \right)^{1/2}. $$
	
Next we show that these operators are bounded $L^p(w) \rightarrow L^p(w)$ for all $w\in A_p(\mbr^{\vn})$. The proof only relies on the fact that the one-parameter maximal function satisfies a weighted bound. So we state the result in a slightly more general form below, replacing $M_{\mcd_2}$ and $M_{\mcd_1}$ by any one-parameter operator that satisfies a weighted bound.

\begin{prop} \label{P:Mixed2pSF}
Let $T$ denote a (one-parameter) operator acting on functions on $\mbr^n$ that satisfies $\|T: L^2(v) \rightarrow L^2(v)\| \leq C$ for all $v \in A_2(\mbr^n)$.
Define the following operators on $\mbr^{\vn}$:
	$$ [ST]f(x_1, x_2) := \left( \sum_{Q_1\in\mcd_1} \left( T(H_{Q_1}^{\ep_1}f)(x_2) \right)^2 \frac{\unit_{Q_1}(x_1)}{|Q_1|} \right)^{1/2}, $$
	$$ [TS]f(x_1, x_2) := \left( \sum_{Q_2\in\mcd_2} \left( T(H_{Q_2}^{\ep_2}f)(x_1) \right)^2 \frac{\unit_{Q_2}(x_2)}{|Q_2|} \right)^{1/2}, $$
where $T$ acts on $\mbr^{n_2}$ in the first operator, and on $\mbr^{n_1}$ in the second.
Then $[ST]$ and $[TS]$ are bounded $L^p(w) \rightarrow L^p(w)$ for all $w\in A_p(\mbr^{\vn})$.
\end{prop}

\begin{proof}
	\begin{align*}
	\| [ST]f \|^p_{L^p(w)} &=  \int_{\mbr^{n_1}} \int_{\mbr^{n_2}} 
		\bigg( \sum_{Q_1\in\mcd_1} \bigg(T(H_{Q_1}^{\ep_1})(x_2) \frac{\unit_{Q_1}(x_1)}{\sqrt{|Q_1|}}\bigg)^2 \bigg)^{p/2} w(x_1, x_2)\,dx_2\,dx_1 \\
		& \lesssim  \int_{\mbr^{n_1}} \int_{\mbr^{n_2}} 
			\bigg( \sum_{Q_1\in\mcd_1} (H_{Q_1}^{\ep_1})^2(x_2) \frac{\unit_{Q_1}(x_1)}{|Q_1|} \bigg)^{p/2} w(x_1, x_2)\,dx_2\,dx_1\\
		&= \|S_{\mcd_1}f\|^p_{L^p(w)} \lesssim \|f\|^p_{L^p(w)},
	\end{align*}
where the first inequality follows as before from Proposition \ref{P:VvalExt}. The proof for $[TS]$ is symmetrical.
\end{proof}

More generally, define \textit{shifted} versions of these mixed operators:
	$$ [ST]^{i_1, j_1} f(x_1, x_2) := \left( \sum_{R_1\in\mcd_1}
		\bigg( \sum_{P_1\in (R_1)_{i_1}} T(H_{P_1}^{\ep_1}f)(x_2) \bigg)^2 \sum_{Q_1\in(R_1)_{j_1}} \frac{\unit_{Q_1}(x_1)}{|Q_1|}   \right)^{1/2}, $$
	$$ [TS]^{i_2, j_2} f(x_1, x_2) := \left( \sum_{R_2\in\mcd_2}
		\bigg( \sum_{P_2\in (R_2)_{i_2}} T(H_{P_2}^{\ep_2}f)(x_1) \bigg)^2 \sum_{Q_2\in(R_2)_{j_2}} \frac{\unit_{Q_2}(x_2)}{|Q_2|}   \right)^{1/2}. $$
Under the same assumptions on $T$, it is easy to see that
	\begin{equation} \label{E:Mixed2pSFShift}
	\| [ST]^{i_1, j_1} : L^p(w) \rightarrow L^p(w) \| \lesssim 2^{\frac{n_1}{2}(i_1+j_1)}
	\text{ and }
	\| [TS]^{i_2, j_2} : L^p(w) \rightarrow L^p(w) \| \lesssim 2^{\frac{n_2}{2}(i_2+j_2)},
	\end{equation}
for all $w \in A_p(\mbr^{\vn})$. Specifically,
	\begin{align*}
	\| [ST]^{i_1, j_1} f\|^p_{L^p(w)} &= \int |S_{\mcd_1}^{i_1, j_1}F(x_1, x_2)|^p \,dw
		\text{, where } F(x_1, x_2) := \sum_{P_1\in\mcd_1} T(H_{P_1}^{\ep_1}f)(x_2) h_{P_1}^{\ep_1}(x_1), 
	\end{align*}
so $ \| [ST]^{i_1, j_1} f\|_{L^p(w)} \lesssim 2^{\frac{n_1}{2}(i_1+j_1)} \|F\|_{L^p(w)}$.
Now, $  \|F\|_{L^p(w)} \simeq \|S_{\mcd_1}F\|_{L^p(w)} = \| [ST] f\|_{L^p(w)} \lesssim \|f\|_{L^p(w)}$.

\section{Biparameter Weighted BMO Spaces}
\label{S:BWBMO}

Given a weight $w$ on $\mbr^n$, a locally integrable function $b$ is said to be in the weighted $BMO(w)$ space if
	$$\|b\|_{BMO(w)} := \sup_Q \frac{1}{w(Q)}\int_Q |b(x) - \La b\Ra_Q|\,dx < \infty,$$
where the supremum is over all cubes $Q$ in $\mbr^n$. If $w = 1$, we obtain the unweighted $BMO(\mbr^n)$ space. 
The dyadic version $BMO_{\mcd}(w)$ is obtained by only taking supremum over $Q\in\mcd$ for some given dyadic grid $\mcd$
on $\mbr^n$. If the weight $w\in A_p(\mbr^n)$ for some $1<p<\infty$, Muckenhoupt and Wheeden show in \cite{MuckWheeden} that
	\begin{equation}\label{E:MuckWheeden-1p} 
	\|b\|_{BMO(w)} \simeq \|b\|_{BMO(w';p')} := \sup_Q \bigg( \frac{1}{w(Q)} \int_Q |b - \La b\Ra_Q|^{p'}\,dw' \bigg)^{1/p'}, 
	\end{equation}
where $w'$ is the conjugate weight to $w$.
Moreover, if $w\in A_2(\mbr^n)$, Wu's argument in \cite{Wu} shows that $BMO_{\mcd}(w) \simeq H^1_{\mcd}(w)^*$,
where the dyadic Hardy space $H_{\mcd}^1(w)$ is defined by the norm
	$$ \|\phi\|_{H^1_{\mcd}(w)} := \| S_{\mcd}\phi\|_{L^1(w)}.$$
Then
	\begin{equation} \label{E:H1BMO-1p}
	|\La b, \phi\Ra| \lesssim \|b\|_{BMO_{\mcd}(w)} \|S_{\mcd}\phi\|_{L^1(w)} \text{, for all } w\in A_2(\mbr^n).
	\end{equation}

Now suppose $\mu$ and $\lb$ are $A_p(\mbr^n)$ weights for some $1<p<\infty$, and define the Bloom weight $\nu: = \mu^{1/p} \lb^{-1/p}$.
As shown in \cite{HLW2}, the weight $\nu \in A_2(\mbr^n)$, which means we may use \eqref{E:H1BMO-1p} with $\nu$.
A two-weight John-Nirenberg theorem for the Bloom BMO space $BMO(\nu)$ is also proved in \cite{HLW2}, namely
	\begin{equation} \label{E:Bloom-JN-1p}
	\|b\|_{BMO(\nu)} \simeq \|b\|_{BMO(\mu,\lb,p)} \simeq \|b\|_{BMO(\lb',\mu',p')},
	\end{equation}
where
	\begin{align}
	& \|b\|_{BMO(\mu,\lb,p)} := \sup_Q \bigg( \frac{1}{\mu(Q)} \int_Q |b - \La b\Ra_Q|^p\,d\lb \bigg)^{1/p},\\
	& \|b\|_{BMO(\lb',\mu',p')} := \sup_Q \bigg( \frac{1}{\lb'(Q)} \int_Q |b - \La b\Ra_Q|^{p'}\,d\mu' \bigg)^{1/p'}.
	\end{align}

We now look at weighted BMO spaces in the product setting $\mbr^{\vn} = \mbr^{n_1}\otimes\mbr^{n_2}$. 
Suppose $w(x_1, x_2)$ is a weight on $\mbr^{\vn}$. Then we have three BMO spaces:
\begin{itemize}
\item \underline{Weighted Little $bmo(w)$:} is the space of all locally integrable functions $b$ on $\mbr^{\vn}$ such that
	$$ \|b\|_{bmo(w)} := \sup_R \frac{1}{w(R)} \int_R |b - \La b\Ra_R| \,dx < \infty,$$
where the supremum is over all \textit{rectangles} $R = Q_1\times Q_2$ in $\mbr^{\vn}$. 
Given a choice of dyadic rectangles $\dr = \mcd_1\times\mcd_2$, we define the dyadic weighted little
$bmo_{\dr}(w)$ by taking supremum over $R\in\dr$.

\item \underline{Weighted Product $BMO_{\dr}(w)$:} is the space of all locally integrable functions $b$ on $\mbr^{\vn}$ such that:
	\begin{equation}
	\|b\|_{BMO_{\dr}(w)} := \sup_{\Omega} \left( \frac{1}{w(\Omega)} \sum_{R \subset \Omega; R \in \dr} | \widehat{b}(R)|^2 \frac{1}{\La w\Ra_R} \right)^{1/2} < \infty,
	\end{equation}
where the supremum is over all \textit{open sets} $\Omega \subset \mbr^{\vn}$ with $w(\Omega) < \infty$. 

\item \underline{Weighted Rectangular $BMO_{\dr, Rec}(w)$:} is defined  in a similar fashion to the unweighted case --
just like product BMO, but taking supremum over rectangles instead of over open sets:
	$$ \|b\|_{BMO_{\dr, Rec}(w)} := \sup_R \left( \frac{1}{w(R)} \sum_{T \subset R} |\widehat{b}(T^\ep)|^2 \frac{1}{\La w\Ra_T} \right)^{1/2}, $$
where the supremum is over all rectangles $R$, and the summation is over all subrectangles $T \in \dr$, $T \subset R$.
\end{itemize}

We have the inclusions
	$$ bmo_{\mcd}(w) \subsetneq BMO_{\dr}(w) \subsetneq BMO_{\dr, Rec}(w). $$
Let us look more closely at some of these spaces.

\subsection{Weighted Product $BMO_{\dr}(w)$}

As in the one parameter case, we define the dyadic weighted Hardy space $\mch^1_{\dr}(w)$ to be the space of all $\phi \in L^1(w)$ such that $S_{\dr}\phi \in L^1(w)$, 
a Banach space under the norm $ \|\phi\|_{\mch^1_{\dr}(w)} := \| S_{\dr} \phi\|_{L^1(w)}$.
The following result exists in the literature under various forms, but we include a proof here for completeness.

\begin{prop} \label{P:H1BMO}
With the notation above, $\mch^1_{\dr}(w)^* \equiv BMO_{\dr}(w)$. Specifically, 
every $b \in BMO_{\dr}(w)$ determines a continuous linear functional on $\mch^1_{\dr}(w)$ 
by $\phi \mapsto \La b, \phi \Ra$:
		\begin{equation} \label{E:H1BMOprod}
		\left| \La b, \phi \Ra \right| \lesssim \|b\|_{BMO_{\dr}(w)} \|S_{\dr} \phi\|_{L^1(w)},
		\end{equation}
and, conversely, every $L \in \mch^1_{\dr}(w)^*$ may be realized as $L \phi = \La b, \phi\Ra$ for some $b \in BMO_{\dr}(w)$.
\end{prop}

\begin{proof}
To prove the first statement, let $b \in BMO_{\dr}(w)$ and $\phi \in \mch^1_{\dr}(w)$. For every $j \in \mbz$, define the set
	$ U_j := \{ x \in \mbr^{\vn}: S_{\dr}\phi(x) > 2^j \}, $
and the collection of rectangles 
	$ \mcr_j := \{ R\in\dr: w(R \cap U_j) > \f{1}{2} w(R)\}. $
Clearly $U_{j+1} \subset U_j$ and $\mcr_{j+1} \subset \mcr_j$. Moreover, 
	\begin{equation}\label{E:DualP-b}
	\sum_{j \in \mbz} 2^j w(U_j) \simeq \|S_{\dr}\phi\|_{L^1(w)},
	\end{equation}
which comes from the measure theoretical fact that for any integrable function $f$ on a measure space $(\mcx, \mu)$:
	$ \|f\|_{L^1(\mu)} \simeq \sum_{j \in \mbz} 2^j \mu\{x\in\mcx: |f(x)|>2^j\}. $

As shown in Proposition \ref{P:2ParRH}, there exist $C, \delta > 0$ such that 
	$\frac{w(E)}{w(R)} \leq C \left( \frac{|E|}{|R|} \right)^{\delta},$
for all rectangles $R$ and measurable subsets $E \subset R$. Define then for every $j \in \mbz$ the (open) set:
	$$ V_j := \{ x \in \mbr^{\vn}: M_S\unit_{U_j}(x) > \theta \} \text{, where } \theta := \left(\frac{1}{2C}\right)^{1/\delta}.$$
First note that if $R \in \mcr_j$, then
	$$ \frac{1}{2} < \frac{w(R\cap U_j)}{w(R)} \leq C\left( \frac{|R\cap U_j|}{|R|} \right)^{\delta} \text{, so } \theta < \La \unit_{U_j}\Ra_{R} \leq M_S \unit_{U_j}(x)
	\text{, for all } x \in R.$$ 
Therefore
	\begin{equation} \label{E:DualP-c}
	\bigcup_{R \in \mcr_j} R \subset V_j.
	\end{equation}	
Using \eqref{E:LlogLMS}, we have that
	\begin{equation}\label{E:DualP-d} 
	w(V_j) \lesssim \int_{U_j} \frac{1}{\theta^p} \left( 1 + \log^{+} \frac{1}{\theta}\right)^{k-1}\,dw \simeq w(U_j). 
	\end{equation}

Now suppose $R\in\dr$ but $R \notin \bigcup_{j\in\mbz} \mcr_j$. Then $w(R \cap \{S_{\dr}\phi \leq 2^j\}) \geq \f{1}{2}w(R)$
for all $j \in \mbz$, and so
	$$ w(R \cap \{S_{\dr}\phi = 0\}) = w\left( \bigcap_{j=1}^{\infty} R \cap \{S_{\dr}\phi \leq 2^{-j}\} \right) \geq \frac{1}{2} w(R).$$
Then $|\{S_{\dr}\phi = 0\}| \geq |R \cap \{S_{\dr}\phi = 0\}| \geq \theta |R| > 0$, and we may write
	$$ |\widehat{\phi}(R)|^2 = \int_{\{S_{\dr}\phi = 0\}} |\widehat{\phi}(R)|^2 \frac{\unit_R}{|R\cap \{S_{\dr}\phi = 0\}|} \,dx
		\leq \frac{1}{\theta} \int_{\{S_{\dr}\phi = 0\}} (S_{\dr}\phi)^2\, dx = 0.$$
So
	\begin{equation} \label{E:DualP-e}
	\widehat{\phi}(R) = 0 \text{, for all } R \in \dr,\: R \notin \bigcup_{j \in \mbz}\mcr_j.
	\end{equation}
Finally, if $R \in \bigcap_{j \in \mbz}\mcr_j$, then
	$$0 = w(R \cap \{S_{\dr}\phi = \infty\}) = \lim_{j \rightarrow \infty} w(R \cap \{S_{\dr}\phi > 2^j\}) \geq \f{1}{2} w(R),$$
a contradiction. In light of this and \eqref{E:DualP-e}, 
	\begin{eqnarray*}
	\sum_{R \in \dr} |\widehat{b}(R)| |\widehat{\phi}(R)| &=& \sum_{j \in \mbz} \sum_{R \in \mcr_j \setminus \mcr_{j+1}} |\widehat{b}(R)| |\widehat{\phi}(R)|\\
		&\leq& \sum_{j \in \mbz} \left( \sum_{R \in \mcr_j \setminus \mcr_{j+1}} |\widehat{b}(R)|^2 \frac{1}{\La w\Ra_R} \right)^{1/2} 
			\left( \sum_{R \in \mcr_j \setminus \mcr_{j+1}} |\widehat{\phi}(R)|^2 \La w\Ra_R \right)^{1/2}
	\end{eqnarray*}
To estimate the first term, we simply note that
	$$ \sum_{R \in \mcr_j \setminus \mcr_{j+1}} |\widehat{b}(R)|^2 \frac{1}{\La w\Ra_R} \leq \sum_{R \in \mcr_j} |\widehat{b}(R)|^2 \frac{1}{\La w\Ra_R}
		\leq \sum_{R\subset V_j; R \in \dr} |\widehat{b}(R)|^2 \frac{1}{\La w\Ra_R} \leq \|b\|^2_{BMO_{\dr}(w)} w(V_j),  $$ 
where the second inequality follows from \eqref{E:DualP-c}.
For the second term, remark that any $R \in \mcr_{j} \setminus \mcr_{j+1}$ satisfies $R \subset V_j$ and $w(R \setminus U_{j+1}) \geq \frac{1}{2}w(R)$.
Then 
	\begin{eqnarray*}
	\sum_{R \in \mcr_j \setminus \mcr_{j+1}} |\widehat{\phi}(R)|^2 \La w\Ra_R &\leq& 
		2\sum_{R\in\mcr_j\setminus\mcr_{j+1}} |\widehat{\phi}(R)|^2 \frac{w(R\setminus U_{j+1})}{|R|}\\
	&=& 2 \int_{V_j \setminus U_{j+1}} \sum_{R\in\mcr_j\setminus\mcr_{j+1}} |\widehat{\phi}(R)|^2 \frac{\unit_R}{|R|}\,dw\\
	&\leq& 2 \int_{V_j\setminus U_{j+1}} (S_{\dr}\phi)^2\,dw \lesssim 2^{2j} w(V_j),
	\end{eqnarray*}
since $S_{\dr}\phi \leq 2^{j+1}$ off $U_{j+1}$. Finally, we have by \eqref{E:DualP-d}:
	$$ \sum_{R \in \dr} |\widehat{b}(R)| |\widehat{\phi}(R)|  \lesssim \|b\|_{BMO_{\dr}(w)} 
		\sum_{j \in \mbz} 2^j w(V_j) \simeq \|b\|_{BMO_{\dr}(w)} \sum_{j\in\mbz} 2^j w(U_j).$$
Combining this with \eqref{E:DualP-b}, we obtain \eqref{E:H1BMOprod}.

To see the converse, let $L \in \mch^1_{\dr}(w)$. Then $L$ is given by $L \phi = \La b, \phi\Ra$  for some function $b$. Fix an open set $\Omega$ with $w(\Omega) < \infty$. Then
	$$ \left( \sum_{R\subset \Omega; R\in\dr} |\widehat{b}(R)|^2 \frac{1}{\La w\Ra_R} \right)^{1/2} 
		\leq \sup_{\|\phi\|_{l^2(\Omega, w)} \leq 1} \left| \sum_{R\subset\Omega, R\in\dr} \widehat{b}(R) \widehat{\phi}(R)\right|, $$
where $ \|\phi\|_{l^2(\Omega, w)}^2 := \sum_{R\subset\Omega, R\in\dr} |\widehat{\phi}(R)|^2 \La w\Ra_R$. By a simple application of H\"{o}lder's inequality,
	$$ \left| \sum_{R\subset\Omega, R\in\dr} \widehat{b}(R) \widehat{\phi}(R)\right| \lesssim \|L\|_{\star} \|\phi\|_{\mch^1_{\dr}(w)}
		\leq \|L\|_{\star} (w(\Omega))^{1/2} \|\phi\|_{l^2(\Omega, w)},  $$
so $\|b\|_{BMO_{\dr}(w)} \lesssim \|L\|_{\star}$.

\end{proof}

\subsection{Weighted little $bmo_{\dr}(w)$}

In this case, we also want to look at  each variable separately. Specifically, we look at the space
$BMO(w_1, x_2)$: for each $x_2 \in \mbr^{n_2}$, this is the weighted BMO space over $\mbr^{n_1}$, with respect to the weight $w(\cdot, x_2)$.
	$$ BMO(w_1, x_2) := BMO(w(\cdot, x_2); \:\: \mbr^{n_1}) \text{, for each } x_2 \in \mbr^{n_2}. $$
The norm in this space is given by
	$$ \|b(\cdot, x_2)\|_{BMO(w_1, x_2)} := \sup_{Q_1} \frac{1}{w(Q_1, x_2)} \int_{Q_1} | b(x_1, x_2) - m_{Q_1}b(x_2) |\,dx_1, $$
where
	$$ w(Q_1, x_2) := \int_{Q_1} w(x_1, x_2)\,dx_1 \:\:\text{ and }\:\: m_{Q_1}b(x_2) := \frac{1}{|Q_1|} \int_{Q_1} b(x_1, x_2)\,dx_1. $$
The space $BMO(w_2, x_1)$ and the quantities $w(Q_2, x_1)$ and $m_{Q_2}b(x_1)$ are defined symmetrically. 

\begin{prop} \label{P:bmo-eachVar}
Let $w(x_1, x_2)$ be a weight on $\mbr^{\vn} = \mbr^{n_1} \otimes \mbr^{n_2}$. Then $b \in L^1_{loc}(\mbr^{\vn})$ is in $bmo(w)$ if and only if $b$
is in the one-parameter weighted BMO spaces $BMO(w_i, x_j)$ separately in each variable, uniformly:
	\begin{equation} \label{E:bmoSep}
	\|b\|_{bmo(w)} \simeq \max\left\{
	  \esssup_{x_1\in\mbr^{n_1}} \|b(x_1, \cdot)\|_{BMO(w_2, x_1)};\:\: 	  \esssup_{x_2\in\mbr^{n_2}} \|b(\cdot, x_2)\|_{BMO(w_1, x_2)} 
	\right\}.
	\end{equation}
\end{prop}

\begin{rem} \label{R1}
In the \textit{unweighted} case $bmo(\mbr^{\vn})$, if we fixed $x_2\in\mbr^{n_2}$, we would look at $b(\cdot, x_2)$ in the space $BMO(\mbr^{n_1})$ -- the \textit{same}
one-parameter BMO space for all $x_2$. In the weighted case however, the one-parameter space for $b(\cdot, x_2)$ \textit{changes with} $x_2$, because the weight
$w(\cdot, x_2)$ changes with $x_2$.
\end{rem}

\begin{proof}
Suppose first that $b \in bmo(w)$. Then for all cubes $Q_1$, $Q_2$:
	\begin{eqnarray*}
	\|b\|_{bmo(w)} &\geq& \frac{1}{w(Q_1\times Q_2)} \int_{Q_1} \int_{Q_2} |b(x_1, x_2) - \La b\Ra_{Q_1 \times Q_2}|\,dx_2\,dx_1 \\
	&\geq& \frac{1}{w(Q_1 \times Q_2)} \int_{Q_1} \left| \int_{Q_2} b(x_1, x_2) - \La b\Ra_{Q_1 \times Q_2} \,dx_2\right|\,dx_1, \\
	\end{eqnarray*}
so
	\begin{equation} \label{E:bmoSep1}
	\int_{Q_1} | m_{Q_2}b(x_1) - \La b\Ra_{Q_1 \times Q_2} |\,dx_1 \leq \frac{w(Q_1\times Q_2)}{|Q_2|} \|b\|_{bmo(w)}.
	\end{equation}	
Now fix a cube $Q_2$ in $\mbr^{n_2}$ and let $f_{Q_2}(x_1) := \int_{Q_2} |b(x_1, x_2) - m_{Q_2}b(x_1)|\,dx_2$. Then for any $Q_1$:
	\begin{eqnarray*}
	\La f_{Q_2}\Ra_{Q_1} &\leq& \frac{1}{|Q_1|} \int_{Q_1} \int_{Q_2} |b(x_1, x_2) - \La b\Ra_{Q_1\times Q_2}|\,dx
		+ \frac{1}{|Q_1|} \int_{Q_1} \int_{Q_2} |m_{Q_2}b(x_1) - \La b\Ra_{Q_1\times Q_2}|\,dx\\
	&\leq& \frac{w(Q_1\times Q_2)}{|Q_1|} \|b\|_{bmo(w)} + \frac{|Q_2|}{|Q_1|} \int_{Q_1} |m_{Q_2}b(x_1) - \La b\Ra_{Q_1\times Q_2}|\,dx_1\\
	&\leq& 2 \frac{w(Q_1\times Q_2)}{|Q_1|} \|b\|_{bmo(w)} = 2 \La w(Q_2, \cdot)\Ra_{Q_1} \|b\|_{bmo(w)},
	\end{eqnarray*}
where the last inequality follows from \eqref{E:bmoSep1}. By the Lebesgue differentiation theorem:
	$$ f_{Q_2}(x_1) = \lim_{Q_1 \rightarrow x_1} \La f_{Q_2}\Ra_{Q_1} \leq 2\|b\|_{bmo(w)} \lim_{Q_1 \rightarrow x_1} \La w(Q_2, \cdot)\Ra_{Q_1} 
	= 2\|b\|_{bmo(w)} w(Q_2, x_1),$$
for almost all $x_1 \in \mbr^{n_1}$, where $Q_1 \rightarrow x_1$ denotes a sequence of cubes containing $x_1$ with side length tending to $0$. 

We would like to say at this point that $\|b(x_1, \cdot)\|_{BMO(w_2, x_1)} = \sup_{Q_2} \frac{1}{w(Q_2, x_1)} f_{Q_2}(x_1)$ is uniformly (a.a. $x_1$)
bounded. However, we must be a little careful and note that at this point we really have that for every cube $Q_2$ in $\mbr^{n_2}$, there is a 
\textit{null set} $N(Q_2) \subset \mbr^{n_1}$ such that
	$$ f_{Q_2}(x_1) \leq 2\|b\|_{bmo(w)} w(Q_2, x_1) \text{ for all } x_1 \in \mbr^{n_1} \setminus N(Q_2). $$
In order to obtain the inequality we want, holding for a.a. $x_1$, let $N:= \cup N(\widetilde{Q_2})$ where $\widetilde{Q_2}$ are the cubes in $\mbr^{n_2}$ with rational side length
and centers with rational coordinates. Then $N$ is a null set and $ f_{\widetilde{Q_2}}(x_1) \leq 2\|b\|_{bmo(w)} w(\widetilde{Q_2}, x_1)$ for all $x_1 \in \mbr^{n_1} \setminus N$.
By density, this statement then holds for \textit{all} cubes $Q_2$ and $x_1 \notin N$, so
	$$ \esssup_{x_1\in\mbr^{n_1}} \|b(x_1, \cdot)\|_{BMO(w_2, x_1)} \leq 2 \|b\|_{bmo(w)}. $$
The result for the other variable follows symmetrically.

Conversely, suppose 
	$$ \|b(x_1, \cdot)\|_{BMO(w_2, x_1)} \leq C_1 \text{ for a.a. } x_1 \text{, and } \|b(\cdot, x_2)\|_{BMO(w_1, x_2)} \leq C_2 \text{ for a.a. } x_2. $$
Then for any $R = Q_1 \times Q_2$:
	\begin{eqnarray*}
	\int_R |b - \La b\Ra_R|\,dx &\leq& \int_{Q_1}\int_{Q_2} |b(x_1, x_2) - m_{Q_2}(x_1)|\,dx + \int_{Q_1} |Q_2| |m_{Q_2}b(x_1) - \La b\Ra_{Q_1\times Q_2}|\,dx_1 \\
	&\leq& \int_{Q_1} C_2 w(Q_2, x_1)\,dx_1 + \int_{Q_1}\int_{Q_2} |b(x_1, x_2) - m_{Q_1}b(x_2)|\,dx_2\,dx_1\\
	&\leq& C_2 w(R) + \int_{Q_2} C_1 w(Q_1, x_2)\,dx_2\\
	&=& (C_1 + C_2) w(R),
	\end{eqnarray*}
so
	$$\|b\|_{bmo(w)} \leq 2\max\left\{
	  \esssup_{x_1\in\mbr^{n_1}} \|b(x_1, \cdot)\|_{BMO(w_2, x_1)};\:\: 	  \esssup_{x_2\in\mbr^{n_2}} \|b(\cdot, x_2)\|_{BMO(w_1, x_2)} 
	\right\}.$$
\end{proof}

\begin{cor} \label{C:A2bmo2}
Let $w \in A_2(\mbr^{\vn})$ and $b \in bmo_{\dr}(w)$. Then
	$$ |\La b, \phi\Ra | \lesssim  \|b\|_{bmo_{\dr}(w)} \|S_{\mcd_i}\phi\|_{L^1(w)},$$
for all $i \in \{1, 2\}$.
\end{cor}

\begin{proof}
This follows immediately from the one-parameter result in \eqref{E:H1BMO-1p} and the proposition above:
	\begin{align}
	|\La b, \phi\Ra | & \leq \int_{\mbr^{n_1}} | \La b(x_1, \cdot), \phi(x_1, \cdot)\Ra_{\mbr^{n_2}} |\,dx_1\\
	& \lesssim \int_{\mbr^{n_1}} \|b(x_1, \cdot)\|_{BMO_{\mcd_2}(w(x_1, \cdot))} \|S_{\mcd_2}\phi(x_1, \cdot)\|_{L^1(w(x_1, \cdot))} \,dx_1\\
	& \lesssim \|b\|_{bmo(w)} \|S_{\mcd_2}\phi\|_{L^1(w)},
	\end{align}
and similarly for $S_{\mcd_1}$.
\end{proof}

We now look at the little bmo version of \eqref{E:MuckWheeden-1p}.

\begin{prop} \label{P:MuckWheeden-2p}
If $w\in A_p(\mbr^{\vn})$ for some $1<p<\infty$, then
	$$ \|b\|_{bmo(w)} \simeq \|b\|_{bmo(w;p')} := \sup_R \bigg( \frac{1}{w(R)} \int_R |b - \La b\Ra_R|^{p'}\,dw' \bigg)^{1/p'}. $$
\end{prop}

\begin{proof}
By Proposition \ref{P:bmo-eachVar} and \eqref{E:MuckWheeden-1p}:
	$$\|b\|_{bmo(w)} \simeq \max\left\{
	  \esssup_{x_1\in\mbr^{n_1}} \|b(x_1, \cdot)\|_{BMO(w(x_1, \cdot); p')};\:\: 	  \esssup_{x_2\in\mbr^{n_2}} \|b(\cdot, x_2)\|_{BMO(w(\cdot, x_2);p')} 
	\right\}.  $$ 
Suppose first that $b \in bmo(w;p')$. Note that for some function $g$ on $\mbr^{\vn}$ and a cube $Q_2$ in $\mbr^{n_2}$, we have
	$$ \int_{Q_2} |g(x_1, x_2)|^{p'} w'(x_1, x_2)\,dx_2 \geq \frac{1}{w(Q_2,x_1)^{p'-1}} \left| \int_{Q_2} g(x_1, x_2)\,dx_2 \right|^{p'}.$$	
Then
	\begin{align}
	\|b\|^{p'}_{bmo(w;p')} &\geq \frac{1}{w(R)} \int_{Q_1} \frac{1}{w(Q_2, x_1)^{p'-1}} 
		\left| \int_{Q_2} b(x_1, x_2) - \La b\Ra_{Q_1\times Q_2}\,dx_2 \right|^{p'}\,dx_1\\
	&= \frac{1}{w(R)} \int_{Q_1} \left| m_{Q_2}b(x_1) - \La b\Ra_{Q_1\times Q_2} \right|^{p'} \frac{|Q_2|^{p'}}{w(Q_2,x_1)^{p'-1}}\,dx_1\\
	&\geq \frac{1}{w(R)} \int_{Q_1} \left| m_{Q_2}b(x_1) - \La b\Ra_{Q_1\times Q_2} \right|^{p'} w'(Q_2, x_1)\,dx_1,
	\end{align}	
where the last inequality follows from
	$$ \frac{|Q_2|^{p'}}{w(Q_2,x_1)^{p'-1}} = |Q_2| \frac{1}{\La w(x_1, \cdot)\Ra_{Q_2}^{p'-1}} \geq |Q_2| \frac{\La w'(x_1, \cdot)\Ra_{Q_2}}{[w(x_1,\cdot)]_{A_p}^{p'-1}}
	\simeq w'(Q_2, x_1). $$	
Now fix $Q_2$ and consider $f_{Q_2}(x_1) := \int_{Q_2} |b(x_1, x_2) - m_{Q_2}b(x_1)|^{p'} w'(x_1, x_2)\,dx_2$. Then
	\begin{align}
	\La f_{Q_2}\Ra_{Q_1} &\lesssim \frac{1}{|Q_1|} \int_{Q_1} \int_{Q_2} \bigg( 
		|b(x_1, x_2) - \La b\Ra_{Q_1\times Q_2}|^{p'} + |m_{Q_2}b(x_1) - \La b\Ra_{Q_1\times Q_2}|^{p'} \bigg) w'(x_1, x_2)\,dx_2\,dx_1\\
		&\lesssim \frac{w(Q_1\times Q_2)}{|Q_1|} \|b\|^{p'}_{bmo(w;p')} + \frac{1}{|Q_1|} \int_{Q_1} |m_{Q_2} b(x_1) - \La b\Ra_{Q_1\times Q_2}|^{p'}
			w'(Q_2, x_1)\,dx_1\\
		& \lesssim \frac{w(Q_1\times Q_2)}{|Q_1|} \|b\|^{p'}_{bmo(w;p')}.
	\end{align}
Then for almost all $x_1$:
	$$ f_{Q_2}(x_1) = \lim_{Q_1\rightarrow x_1} \La f_{Q_2}\Ra_{Q_1} \lesssim \lim_{Q_1\rightarrow x_1} \frac{w(Q_1\times Q_2)}{|Q_1|} \|b\|^{p'}_{bmo(w;p')} 
		= w(Q_2, x_1) \|b\|^{p'}_{bmo(w;p')}.$$
Taking again rational cubes, we obtain
	$$ \|b(x_1, \cdot)\|_{BMO(w(x_1, \cdot);p')} = \sup_{Q_2} \bigg( \frac{1}{w(Q_2, x_1)} f_{Q_2}(x_1) \bigg)^{1/p'} \lesssim \|b\|_{bmo(w;p')},$$
for almost all $x_1$.

Conversely, if $b\in bmo(w)$, then there exist $C_1$ and $C_2$ such that
	$$ \|b(x_1, \cdot)\|_{BMO(w(x_1, \cdot);p')}\leq C_1 \text{ a.a. } x_1\text{, and } 
	\|b(\cdot, x_2)\|_{BMO(w(\cdot, x_2);p')} \leq C_2 \text{ a.a. } x_2.$$
Then
	\begin{align}
	 \int_R |b - \La b\Ra_R|^{p'}\,dw' \lesssim & \int_{Q_1}\int_{Q_2} |b(x_1, x_2) - m_{Q_2}b(x_1)|^{p'} w'(x_1, x_2)\,dx_2\,dx_1\\
	&+ \int_{Q_1} \int_{Q_2} |m_{Q_2}b(x_1) - \La b\Ra_{Q_1\times Q_2}|^{p'} w'(x_1, x_2)\,dx_2\,dx_1.
	\end{align}	
The first integral is easily seen to be bounded by
	$$ \int_{Q_1} \|b(x_1, \cdot)\|^{p'}_{BMO(w(x_1,\cdot))} w(Q_2, x_1)\,dx_1 \leq C_1^{p'} w(Q_1\times Q_2).$$
The second integral is equal to:
	\begin{align}
	& \int_{Q_1} |m_{Q_2}b(x_1) - \La b\Ra_{Q_1\times Q_2}|^{p'} w'(Q_2, x_1)\,dx_1\\
	&\leq \int_{Q_1} \frac{w'(Q_2, x_1)}{|Q_2|^{p'}} \big( \int_{Q_2} |b(x_1,x_2) - m_{Q_1}b(x_2)|\,dx_2 \big)^{p'}\,dx_1\\
	&\leq \int_{Q_1} \frac{w'(Q_2,x_1) w(Q_2, x_1)^{p'-1}}{|Q_2|^{p'}} \int_{Q_2} |b(x_1,x_2) - m_{Q_1}b(x_2)|^{p'} w'(x_1, x_2)\,dx_2\,dx_1.
	\end{align}	
We may express the first term as $\La w'(x_1, \cdot)\Ra_{Q_2} \La w(x_1, \cdot)\Ra_{Q_2}^{p'-1} \lesssim [w]_{A_p}^{p'-1}$ for almost all $x_1$. Then, the integral is
further bounded by
	$$\int_{Q_2} w(Q_1, x_2) \|b(\cdot, x_2)\|_{BMO(w(\cdot, x_2);p')}\,dx_2
	\lesssim C_2^{p'} w(Q_1\times Q_2).$$
Finally, this gives
	$$ \|b\|_{bmo(w;p')} \lesssim (C_1^{p'} + C_2^{p'})^{1/p'} \lesssim \max(C_1, C_2) \simeq \|b\|_{bmo(w)}. $$
\end{proof}

We also have a two-weight John-Nirenberg for Bloom little bmo, which follows very similarly to the proof above. 

\begin{prop}\label{P:Bloom-JN-2p}
Let $\mu, \lb \in A_p(\mbr^{\vn})$ for $1<p<\infty$, and $\nu := \mu^{1/p}\lb^{-1/p}$. Then
	\begin{equation} \label{E:Bloom-JN-2p}
	\|b\|_{bmo(\nu)} \simeq \|b\|_{bmo(\mu,\lb,p)} \simeq \|b\|_{bmo(\lb',\mu',p')},
	\end{equation}
where
	\begin{align}
	& \|b\|_{bmo(\mu,\lb,p)} := \sup_R \bigg( \frac{1}{\mu(R)} \int_R |b - \La b\Ra_R|^p\,d\lb \bigg)^{1/p},\\
	& \|b\|_{bmo(\lb',\mu',p')} := \sup_R \bigg( \frac{1}{\lb'(R)} \int_R |b - \La b\Ra_R|^{p'}\,d\mu' \bigg)^{1/p'}.
	\end{align}
\end{prop}

\noindent Remark that it also easily follows that $\nu \in A_2(\mbr^{\vn})$.

\section{Proof of the Lower Bound}
\label{S:LB}

\begin{proof}[Proof of Theorem \ref{T:LB}]
To see the lower bound, we adapt the argument of Coifman, Rochberg and Weiss \cite{CRW}. Let $\{ X_k (x) \}$ and
$\{ Y_l (y) \}$ both be orthonormal bases for the space of spherical harmonics of degree $n$
in $\mathbb{R}^n$ respectively. Then $\sum_k | X_k (x) |^2 = c_n | x
|^{2 n}$ and thus
\[ 1 = \frac{1}{c_n} \sum_k \frac{X_k (x - x')}{| x - x' |^{2 n}} X_k (x - x')
\]
and similarly for $Y_l .$

Furthermore $X_k (x - x') = \sum_{| \alpha | + | \beta | = n}
\mathbf{x}^{(k)}_{\alpha \beta} x^{\alpha} x'^{\beta}$ and equally for
$Y_l$. Remember that
\[ b (x, y) \in {bmo(\nu)} \Longleftrightarrow \| b \|_{{bmo(\nu)}} = \sup_Q
   \frac{1}{\nu(Q)} \int_Q | b (x, y) - \langle b \rangle_Q | d x d y < \infty
   . \]
Here, $Q = I \times J$ and $I$ and $J$ are cubes in $\mathbb{R}^n$. Let us
define the function
\[ \Gamma_Q (x, y) = \text{sign} (b (x, y) - \langle b \rangle_Q)
   \mathbf{1}_Q (x, y) . \]
So
\begin{eqnarray*}
  &  & | b (x, y) - \langle b \rangle_Q |  | Q | \mathbf{1}_Q (x, y)\\
  & = & (b (x, y) - \langle b \rangle_Q) | Q | \Gamma_Q (x, y)\\
  & = & \int_Q (b (x, y) - b (x', y')) \Gamma_Q (x, y) d x' d y'\\
  & \sim & \sum_{k, l} \int_Q (b (x, y) - b (x', y')) \frac{X_k (x - x')}{| x
  - x' |^{2 n}} X_k (x - x') \frac{Y_l (y - y')}{| y - y' |^{2 n}} Y_l (y
  - y') \Gamma_Q (x, y) d x' d y'\\
  & = & \sum_{k, l} \int_{\mathbb{R}^{2 n}} \frac{b (x, y) - b (x', y')}{| x
  - x' |^{2 n} | y - y' |^{2 n}}_{} X_k (x - x') Y_l (y - y') \cdot\\
  &  & \cdot \sum_{| \alpha | + | \beta | = n} \mathbf{x}^{(k)}_{\alpha
  \beta} x^{\alpha} x'^{\beta} \sum_{| \gamma | + | \delta | = n}
  \mathbf{y}^{(l)}_{\gamma \delta} y^{\gamma} y'^{\delta} \Gamma_Q (x, y)
  \mathbf{1}_Q (x', y') d x' d y' .
\end{eqnarray*}

\

Note that
\begin{align}
& \int_{\mathbb{R}^{2 n}} \frac{b (x, y) - b (x', y')}{| x - x' |^{2
  n} | y - y' |^{2 n}}_{} X_k (x - x') Y_l (y - y') \cdot x'^{\beta} y'^{\delta} \mathbf{1}_Q (x', y') d x' d y' \\
  =& [b, T_k T_l] (x'^{\beta} y'^{\delta} \mathbf{1}_Q (x', y')) .
\end{align}
Here $T_k $ and $T_l$ are the Calder\'on-Zygmund operators that correspond to the
kernels
\[ \frac{X_k (x)}{| x |^{2 n}} \text{ and } \frac{Y_l (y)}{| y |^{2 n}} . \]
Observe that these have the correct homogeneity due to the homogeneity of the
$X_k$ and $Y_l$. With this notation, the above becomes
\begin{eqnarray*}
  &  & | b (x, y) - \langle b \rangle_Q |  | Q | \mathbf{1}_Q (x, y)\\
  & = & \sum_{k, l} \sum_{| \alpha | + | \beta | = n} \sum_{| \gamma | + |
  \delta | = n} \mathbf{x}^{(k)}_{\alpha \beta} x^{\alpha}
  \mathbf{y}^{(l)}_{\gamma \delta} y^{\gamma} \Gamma_Q (x, y) [b, T_k T_l]
  (x'^{\beta} y'^{\delta} \mathbf{1}_Q (x', y')) (x, y) .
\end{eqnarray*}

Now, we integrate with respect to $(x, y)$ and the measure $\lambda .$ Now let
us assume for a moment that both $I$ and $J$ are centered at 0 and thus $Q$
centered at 0. In this case, since $\Gamma_Q$ and $\mathbf{1}_Q$ are
supported in $Q$, there is only contribution for $x, x', y, y'$ in $Q$.
\begin{eqnarray*}
  &  & | Q | \left( \int_Q | b (x, y) - \langle b \rangle_Q |^p d \lambda (x,
  y) \right)^{1 / p}\\
  & \leqslant & \sum_{k, l} \sum_{| \alpha | + | \beta | = n} \sum_{| \gamma
  | + | \delta | = n} \| \mathbf{x}^{(k)}_{\alpha \beta} x^{\alpha}
  \mathbf{y}^{(l)}_{\gamma \delta} y^{\gamma} \Gamma_Q (x, y) [b, T_k T_l]
  (x'^{\beta} y'^{\delta} \mathbf{1}_Q (x', y')) (x, y) \|_{L^p (\lambda)}\\
  & \lesssim & \sum_{k, l} \sum_{| \alpha | + | \beta | = n} \sum_{| \gamma |
  + | \delta | = n} \mathfrak{l} (I)^{| \alpha |} \mathfrak{l} (J)^{| \gamma
  |} \| [b, T_k T_l] (x'^{\beta} y'^{\delta} \mathbf{1}_Q (x', y')) \|_{L^p
  (\lambda)}\\
  & \lesssim & \sum_{k, l} \sum_{| \alpha | + | \beta | = n} \sum_{| \gamma |
  + | \delta | = n} \mathfrak{l} (I)^{| \alpha |} \mathfrak{l} (J)^{| \gamma
  |} \| [b, T_k T_l] \|_{L^p ({\mu}) \rightarrow L^p (\lambda)} \|
  x'^{\beta} y'^{\delta} \mathbf{1}_Q (x', y') \|_{L^p ({\mu})}\\
  & \lesssim & \sum_{k, l} \sum_{| \alpha | + | \beta | = n} \sum_{| \gamma |
  + | \delta | = n} \mathfrak{l} (I)^{| \alpha |} \mathfrak{l} (J)^{| \gamma
  |} \mathfrak{l} (I)^{| \beta |} \mathfrak{l} (J)^{| \delta |} \| [b, T_k
  T_l] \|_{L^p ({\mu}) \rightarrow L^p (\lambda)} {\mu} (Q)^{1 / p}
\end{eqnarray*}
We disregarded the coefficients of the $X$ and $Y$ at the cost of a constant.

Notice that the $T_k$ and $T_l$ are homogeneous polynomials in Riesz
transforms. Therefore the commutator $[b, T_k T_l]$ writes as a linear
combination of terms such as $M [b, R_i^1 R_j^2] N$ where $M$ and $N$ are
compositions of Riesz transforms: in a first step write $[b, T_k T_l]$ as
linear combination of terms of the form $[b, R_{(n)}^k R_{(n)}^l]$ where
\[ R_{(n)}^k = \prod_s R_{i_s^{(k)}}^1 \]
is a composition of $n$ Riesz transforms acting in the variabe $1$ with a
choice $i^{(k)} = (i^{(k)}_s)^n_{s = 1} \in \{ 1, \ldots, n \}^n$ for each $k$
and similar for $R_{(n)}^l$ acting in variable 2. Then, for each term, apply
$[A B, b] = A [B, b] + [A, b] B$ successively as follows. Use $A = R_{i_1}^1
R_{j_1}^2$ and $B$ of the form $R_{(n - 1)}^k R_{(n - 1)}^l$ and repeat. It
then follows that for each $k, l$ the commutator $[b, T_k T_l]$ writes as a
linear combination of terms such as $M [b, R_i^1 R_j^2] N$ where $M$ and $N$
are compositions of Riesz transforms. It is decisive that $T_k$ and $T_l$ are
homogeneous polynomials in Riesz transforms of the same degree. We required
that all commutators of the form $[b, R^1_i R^2_j]$ are bounded, we have shown
the ${bmo}$ estimate for $b$ for rectangles $Q$ whose sides are centered
at 0. We now translate $b$ in the two directions separately and obtain what we
need, by Proposition \ref{P:Bloom-JN-2p}:
$$ \|b\|_{bmo(\nu)}  \simeq \|b\|_{bmo(\mu,\lb,p)} := \sup_R \bigg( \frac{1}{\mu(R)} \int_R |b - \La b\Ra_R|^p\,d\lb \bigg)^{1/p} 
\lesssim \sup_{1 \leqslant k, l \leqslant n} \| [b, R^1_k R^2_l] \|_{L^p
   ({\mu}) \rightarrow L^p (\lambda) .} $$

\end{proof}

\section{Biparameter Paraproducts}
\label{S:Para}

Decomposing two functions $b$ and $f$ on $\mbr^n$ into their Haar series adapted to some dyadic grid $\mcd$ and analyzing the different inclusion properties of the dyadic cubes, one may express their product as
	$$ bf = \Pi_b f + \Pi_b^*f + \Gamma_b f + \Pi_f b,$$
where
	$$ \Pi_b f := \sum_{Q\in\mcd} \widehat{b}(Q^{\ep}) \La f\Ra_Q h_Q^{\ep},
	\:\:\: \Pi_b^*f := \sum_{Q\in\mcd} \widehat{b}(Q^{\ep}) \widehat{f}(Q^{\ep}) \frac{\unit_Q}{|Q|}, \text{ and }
	\Gamma_b f := \sum_{Q \in \mcd} \widehat{b}(Q^{\ep}) \widehat{f}(Q^{\delta}) \frac{1}{\sqrt{|Q|}} h_Q^{\ep+\delta}. $$
In \cite{HLW2}, it was shown that, when $b \in BMO(\nu)$, the operators $\Pi_b$, $\Pi_b^*$, and $\Gamma_b$ are bounded $L^p(\mu) \rightarrow L^p(\lambda)$.

\subsection{Product BMO Paraproducts}\label{Ss:ProdPara}

In the biparameter setting $\dr = \mcd_1 \times \mcd_2$, we have \textit{fifteen} paraproducts. We treat them beginning with the nine paraproducts associated with product BMO.
First, we have the three ``pure'' paraproducts, direct adaptations of the one-parameter paraproducts:
\begin{align}
\Pi_b f &:= \sum_{Q_1 \times Q_2} \widehat{b}(Q_1^{\ep_1} \times Q_2^{\ep_2}) \La f\Ra_{Q_1 \times Q_2} h_{Q_1}^{\ep_1} \otimes h_{Q_2}^{\ep_2},\\
\Pi_b^*f &:= \sum_{Q_1 \times Q_2} \widehat{b}(Q_1^{\ep_1}\times Q_2^{\ep_2}) \widehat{f}(Q_1^{\ep_1} \times Q_2^{\ep_2}) \frac{\unit_{Q_1}}{|Q_1|}\otimes\frac{\unit_{Q_2}}{|Q_2|},\\
\Gamma_{b}f  &:= \sum_{Q_1 \times Q_2} \widehat{b}(Q_1^{\ep_1} \times Q_2^{\ep_2}) \widehat{f}(Q_1^{\delta_1}\times Q_2^{\delta_2})
	\frac{1}{\sqrt{|Q_1|}} \frac{1}{\sqrt{|Q_2|}} h_{Q_1}^{\ep_1+\delta_1} \otimes h_{Q_2}^{\ep_2+\delta_2} = \Gamma_{b}^*f. 
\end{align}
Next, we have the ``mixed'' paraproducts. We index these based on the types of Haar functions acting on $f$, since the action on $b$ is the same for all of them,
namely $\widehat{b}(Q_1\times Q_2)$ -- this is the property which associates these paraproducts with product $BMO_{\dr}$: in a proof using duality,
one would separate out the $b$ function and be left with the biparameter square function $S_{\dr}$. They are:
\begin{align}
\Pi_{b;(0,1)}f &:= \sum_{Q_1\times Q_2} \widehat{b}(Q_1^{\ep_1}\times Q_2^{\ep_2}) \La f, h_{Q_1}^{\ep_1} \otimes \frac{\unit_{Q_2}}{|Q_2|} \Ra 
	\frac{\unit_{Q_1}}{|Q_1|} \otimes h_{Q_2}^{\ep_2}\\
\Pi_{b;(1,0)}f &:= \sum_{Q_1\times Q_2} \widehat{b}(Q_1^{\ep_1}\times Q_2^{\ep_2}) \La f, \frac{\unit_{Q_1}}{|Q_1|} \otimes h_{Q_2}^{\ep_2} \Ra 
	h_{Q_1}^{\ep_1} \otimes \frac{\unit_{Q_2}}{|Q_2|} = \Pi_{b; (0,1)}^*\\
\Gamma_{b; (0,1)}f  &:= \sum_{Q_1 \times Q_2} \widehat{b}(Q_1^{\ep_1} \times Q_2^{\ep_2}) \La f, h_{Q_1}^{\delta_1} \otimes 
	\frac{\unit_{Q_2}}{|Q_2|} \Ra \frac{1}{\sqrt{|Q_1|}} h_{Q_1}^{\ep_1+\delta_1}\otimes h_{Q_2}^{\ep_2} \\
\Gamma_{b; (0,1)}^*f &:=  \sum_{Q_1 \times Q_2} \widehat{b}(Q_1^{\ep_1} \times Q_2^{\ep_2}) \widehat{f}(Q_1^{\delta_1}\times Q_2^{\ep_2})
	\frac{1}{\sqrt{|Q_1|}} h_{Q_1}^{\ep_1+\delta_1} \otimes \frac{\unit_{Q_2}}{|Q_2|}\\
\Gamma_{b; (1,0)}f  &:= \sum_{Q_1 \times Q_2} \widehat{b}(Q_1^{\ep_1} \times Q_2^{\ep_2}) \La f, \frac{\unit_{Q_1}}{|Q_1|} \otimes 
	h_{Q_2}^{\delta_2} \Ra \frac{1}{\sqrt{|Q_2|}} h_{Q_1}^{\ep_1}\otimes h_{Q_2}^{\ep_2+\delta_2} \\
\Gamma_{b; (1,0)}^*f &:=   \sum_{Q_1 \times Q_2} \widehat{b}(Q_1^{\ep_1} \times Q_2^{\ep_2}) \widehat{f}(Q_1^{\ep_1}\times Q_2^{\delta_2})
	\frac{1}{\sqrt{|Q_2|}} \frac{\unit_{Q_1}}{|Q_1|} \otimes h_{Q_2}^{\ep_2+\delta_2}.
\end{align}

\begin{prop} \label{P:BMOparaprod}
If $\nu := \mu^{1/p}\lambda^{-1/p}$ for $A_p(\mbr^{\vn})$ weights $\mu$ and $\lambda$, and $\mathsf{P}_{\mathsf{b}}$ denotes any one of the
nine paraproducts defined above, then 
	\begin{equation} \label{E:BMOparaprod} 
	\left\| \Pp : L^p(\mu) \rightarrow L^p(\lb) \right\| \lesssim \|b\|_{BMO_{\dr}(\nu)},
	\end{equation}
where $\|b\|_{BMO_{\dr}(\nu)}$ denotes the norm of $b$ in the dyadic weighted product $BMO_{\dr}(\nu)$ space on $\mbr^{\vn}$.
\end{prop}

\begin{proof}
We first outline the general strategy we use to prove \eqref{E:BMOparaprod}.
From \eqref{E:ApDuality}, it suffices to take $f \in L^p(\mu)$ and $g \in L^{p'}(\lambda')$ and show that:
	$$ |\La \Pp f, g\Ra| \lesssim \|b\|_{BMO_{\dr}(\nu)} \|f\|_{L^p(\mu)} \|g\|_{L^{p'}(\lambda')}. $$
\begin{enumerate}[1.]
\item Write $\La \Pp f, g\Ra = \La b, \phi\Ra$, where $\phi$ depends on $f$ and $g$. By \eqref{E:H1BMOprod}, $|\La \Pp f, g\Ra| \lesssim \|b\|_{BMO_{\dr}(\nu)} \|S_{\dr}\phi\|_{L^1(\nu)}$.
\item Show that $S_{\dr}\phi \lesssim (\mathcal{O}_1 f) (\mathcal{O}_2 g)$, where $\mathcal{O}_1$ and $\mathcal{O}_2$ are operators satisfying a \textit{one-weight bound}
	$L^p(w) \rightarrow L^p(w)$, for all $w \in A_p(\mbr^{\vn})$ -- these operators will usually be a combination of maximal and square functions.
\item Then the $L^1(\nu)$-norm of $S_{\dr}\phi$ can be separated into the $L^p(\mu)$ and $L^{p'}(\lambda')$ norms of these operators $\mathcal{O}_i$, by a simple application of 
	H\"{o}lder's inequality:
		$$ \|S_{\dr}\phi\|_{L^1(\nu)} \lesssim \|\mathcal{O}_1f\|_{L^p(\mu)} \|\mathcal{O}_2g\|_{L^{p'}(\lambda')}
		\lesssim \|f\|_{L^p(\mu)} \|g\|_{L^{p'}(\lb')},$$
	and the result follows.
\end{enumerate}
Remark also that we will not have to treat the adjoints $\Pp^*$ separately: interchanging the roles of $f$ and $g$ in the proof strategy above will show that $\Pp$ is also bounded
$L^{p'}(\lambda') \rightarrow L^{p'}(\mu')$, which means that $\Pp^*$ is bounded $L^p(\mu) \rightarrow L^p(\lambda)$.

Let us begin with $\Pi_b f$. We write
	$$ \La \Pi_b f, g\Ra = \La b, \phi\Ra \text{, where } \phi := 
		\sum_{Q_1\times Q_2} \La f\Ra_{Q_1\times Q_2} \widehat{g}(Q_1^{\ep_1}\times Q_2^{\ep_2}) h_{Q_1}^{\ep_1}\otimes h_{Q_2}^{\ep_2}. $$
Then
	$$ (S_{\dr}\phi)^2 \leq \sum_{Q_1\times Q_2} \La |f|\Ra^2_{Q_1\times Q_2} |\widehat{g}(Q_1^{\ep_1}\times Q_2^{\ep_2})|^2 
		\frac{\unit_{Q_1}}{|Q_1|}\otimes \frac{\unit_{Q_2}}{|Q_2|} \leq (M_{S}f)^2 \cdot (S_{\dr}g)^2,$$
so
	$$ |\La \Pi_b f, g\Ra| \lesssim \|b\|_{BMO_{\dr}(\nu)} \|M_S f\|_{L^p(\mu)} \|S_{\dr}g\|_{L^{p'}(\lambda')} \lesssim 
		\|b\|_{BMO_{\dr}(\nu)} \|f\|_{L^p(\mu)} \|g\|_{L^{p'}(\lambda')}. $$
Note that if we take instead $f \in L^{p'}(\lambda')$ and $g \in L^p(\mu)$, we have 
	$$ |\La \Pi_bf, g\Ra| \lesssim \|b\|_{BMO_{\dr}(\nu)} \|M_S f\|_{L^{p'}(\lambda')} \|S_{\dr}g\|_{L^p(\mu)} \lesssim
		\|b\|_{BMO_{\dr}(\nu)} \|f\|_{L^{p'}(\lambda')} \|g\|_{L^p(\mu)}, $$
proving that $\|\Pi_b : L^{p'}(\lambda') \rightarrow L^{p'}(\mu')\| = \| \Pi_b^* : L^p(\mu) \rightarrow L^p(\lambda)\| \lesssim \|b\|_{BMO_{\dr}(\nu)}$.
For $\Gamma_b$:
	$$ \La \Gamma_b f, g\Ra = \La b, \phi\Ra \text{, where } \phi :=
		\sum_{Q_1\times Q_2} \widehat{f}(Q_1^{\ep_1}\times Q_2^{\ep_2}) \widehat{g}(Q_1^{\delta_1}\times Q_2^{\delta_2}) 
		\frac{1}{\sqrt{|Q_1|}} \frac{1}{\sqrt{|Q_2|}} h_{Q_1}^{\ep_1+\delta_1} \otimes h_{Q_2}^{\ep_2+\delta_2},$$
from which it easily follows that $S_{\dr}\phi \lesssim S_{\dr}f \cdot S_{\dr}g$.

Let us now look at $\Pi_{b; (0,1)}$. In this case:
	$$ \phi := \sum_{Q_1\times Q_2} \La f, h_{Q_1}^{\ep_1} \otimes \frac{\unit_{Q_2}}{|Q_2|}\Ra 
		\La g, \frac{\unit_{Q_1}}{|Q_1|}\otimes h_{Q_2}^{\ep_2}\Ra h_{Q_1}^{\ep_1}\otimes h_{Q_2}^{\ep_2}.$$
Then
	\begin{align}
	(S_{\dr}\phi)^2 &=  \sum_{Q_1\times Q_2} \La f, h_{Q_1}^{\ep_1} \otimes \frac{\unit_{Q_2}}{|Q_2|}\Ra^2 
		\La g, \frac{\unit_{Q_1}}{|Q_1|}\otimes h_{Q_2}^{\ep_2}\Ra^2 \frac{\unit_{Q_1}}{|Q_1|} \otimes \frac{\unit_{Q_2}}{|Q_2|}\\
			&= \sum_{Q_1\times Q_2} \La H_{Q_1}^{\ep_1}f\Ra_{Q_2}^2 \La H_{Q_2}^{\ep_2}g\Ra_{Q_1}^2 \frac{\unit_{Q_1}}{|Q_1|} \otimes \frac{\unit_{Q_2}}{|Q_2|}\\
			&\leq \Bigg( \sum_{Q_1} \big( M_{\mcd_2} H_{Q_1}^{\ep_1}f \big)^2(x_2) \frac{\unit_{Q_1}(x_1)}{|Q_1|}\Bigg)
				\Bigg( \sum_{Q_2} \big( M_{\mcd_1} H_{Q_2}^{\ep_2}g \big)^2(x_1) \frac{\unit_{Q_2}(x_2)}{|Q_2|}\Bigg)\\
			& = [SM]^2f \cdot [MS]^2g,
	\end{align}
where $[SM]$ and $[MS]$ are the mixed square-maximal operators in Section \ref{Ss:MixedSquare}.
Boundedness of $\Pi_{b;(0,1)}$ then follows from Proposition \ref{P:Mixed2pSF}. 
By the usual duality trick, the same holds for $\Pi_{b; (1,0)}$.
Finally, for $\Gamma_{b; (0,1)}$:
	$$ \phi = \sum_{Q_1\times Q_2} \La H_{Q_1}^{\delta_1}f\Ra_{Q_2} \frac{1}{\sqrt{|Q_1|}} \widehat{g}(Q_1^{\ep_1+\delta_1}\times Q_2^{\ep_2}) h_{Q_1}^{\ep_1}\otimes 
	h_{Q_2}^{\ep_2}, $$
so $S_{\dr}\phi \lesssim [SM]f \cdot S_{\dr}g$. Note that $\Gamma_{b; (1,0)}$ works the same way, except we bound $S_{\dr}\phi$ by $[MS]f \cdot S_{\dr}g$, and the remaining two paraproducts follow by duality.
\end{proof}

\subsection{Little bmo Paraproducts} \label{Ss:bmoPara}

Next, we have the six paraproducts associated with little bmo. We denote these by the small greek letters corresponding to the previous paraproducts, and
index them based on the Haar functions acting on $b$ -- in this case, separating out the $b$ function will yield
one of the square functions $S_{\mcd_i}$ in one of the variables:
	\begin{align*}
	\pi_{b; (0,1)}f &:= \sum_{Q_1\times Q_2} \La b, h_{Q_1}^{\ep_1}\otimes \frac{\unit_{Q_2}}{|Q_2|} \Ra
		\La f, \frac{\unit_{Q_1}}{|Q_1|} \otimes h_{Q_2}^{\ep_2} \Ra h_{Q_1}^{\ep_1} \otimes h_{Q_2}^{\ep_2} \\
	\pi_{b; (0,1)}^*f &:= \sum_{Q_1\times Q_2} \La b, h_{Q_1}^{\ep_1}\otimes \frac{\unit_{Q_2}}{|Q_2|} \Ra
		\widehat{f}(Q_1^{\ep_1}\times Q_2^{\ep_2}) \frac{\unit_{Q_1}}{|Q_1|} \otimes h_{Q_2}^{\ep_2}\\
	\pi_{b; (1,0)}f &:= \sum_{Q_1\times Q_2} \La b, \frac{\unit_{Q_1}}{|Q_1|}\otimes h_{Q_2}^{\ep_2} \Ra
		\La f, h_{Q_1}^{\ep_1}\otimes \frac{\unit_{Q_2}}{|Q_2|} \Ra h_{Q_1}^{\ep_1}\otimes h_{Q_2}^{\ep_2}\\
	\pi_{b; (1,0)}^*f &:= \sum_{Q_1\times Q_2} \La b, \frac{\unit_{Q_1}}{|Q_1|}\otimes h_{Q_2}^{\ep_2} \Ra
		\widehat{f}(Q_1^{\ep_1}\times Q_2^{\ep_2}) h_{Q_1}^{\ep_1} \otimes \frac{\unit_{Q_2}}{|Q_2|}\\
	\gamma_{b; (0,1)} f &:= \sum_{Q_1\times Q_2} \La b, h_{Q_1}^{\delta_1}\otimes \frac{\unit_{Q_2}}{|Q_2|} \Ra
		\widehat{f}(Q_1^{\ep_1}\times Q_2^{\ep_2}) \frac{1}{\sqrt{|Q_1|}} h_{Q_1}^{\ep_1+\delta_1} \otimes h_{Q_2}^{\ep_2} = \gamma_{b; (0,1)}^*f \\ 
	\gamma_{b; (1,0)} f &:= \sum_{Q_1\times Q_2} \La b, \frac{\unit_{Q_1}}{|Q_1|}\otimes h_{Q_2}^{\delta_2} \Ra
		\widehat{f}(Q_1^{\ep_1}\times Q_2^{\ep_2}) \frac{1}{\sqrt{|Q_2|}} h_{Q_1}^{\ep_1}\otimes h_{Q_2}^{\ep_2+\delta_2} = \gamma_{b; (1,0)}^*f.
	\end{align*}
	
\begin{prop} \label{P:bmoparaprod}
If $\nu := \mu^{1/p}\lambda^{-1/p}$ for $A_p(\mbr^{\vn})$ weights $\mu$ and $\lambda$, and $\pp$ denotes any one of the
six paraproducts defined above, then 
	\begin{equation} \label{E:bmoparaprod} 
	\left\| \pp : L^p(\mu) \rightarrow L^p(\lb) \right\| \lesssim \|b\|_{bmo_{\dr}(\nu)},
	\end{equation}
where $\|b\|_{bmo_{\dr}(\nu)}$ denotes the norm of $b$ in the dyadic weighted little $bmo_{\dr}(\nu)$ space on $\mbr^{\vn}$.
\end{prop}

\begin{proof}
The proof strategy is the same as that of the product BMO paraproducts, with the modification that we use one of the $S_{\mcd_i}$ square functions
and Corollary \ref{C:A2bmo2}. For instance, in the case of $\pi_{b;(0,1)}$ we write
	$$ \La \pi_{b; (0,1)}f, g\Ra = \La b, \phi\Ra \text{, where } 
		\phi:= \sum_{Q_1\times Q_2} \La H_{Q_2}^{\ep_2}f \Ra_{Q_1} \widehat{g}(Q_1^{\ep_1}\times Q_2^{\ep_2}) h_{Q_1}^{\ep_1}\otimes \frac{\unit_{Q_2}}{|Q_2|}. $$
Then
	\begin{align}
	(S_{\mcd_1}\phi)^2 &\leq \sum_{Q_1} \bigg( \sum_{Q_2} \La |H_{Q_2}^{\ep_2}f|\Ra_{Q_1}^2 \unit_{Q_1}(x_1) \frac{\unit_{Q_2}(x_2)}{|Q_2|} \bigg)
		\bigg( \sum_{Q_2} |\widehat{g}(Q_1^{\ep_1}\times Q_2^{\ep_2})|^2 \frac{\unit_{Q_2}(x_2)}{|Q_2|} \bigg)\frac{\unit_{Q_1}(x_1)}{|Q_1|}\\
	&\leq \bigg( \sum_{Q_2} M_{\mcd_1}^2(H_{Q_2}^{\ep_2}f)(x_1) \frac{\unit_{Q_2}(x_2)}{|Q_2|} \bigg)
		\bigg( \sum_{Q_1} \sum_{Q_2} |\widehat{g}(Q_1^{\ep_1}\times Q_2^{\ep_2})|^2 \frac{\unit_{Q_1}(x_1)}{|Q_1|} \otimes \frac{\unit_{Q_2}(x_2)}{|Q_2|}\bigg)\\
	&= [MS]^2f \cdot S_{\dr}^2g,
	\end{align}
and so
	$$ |\La\pi_{b;(0,1)}f, g\Ra| \lesssim \|b\|_{bmo_{\dr}(\nu)} \|S_{\mcd_1}\phi\|_{L^1(\nu)} 
		\lesssim \|b\|_{bmo_{\dr}(\nu)} \|f\|_{L^p(\mu)} \|g\|_{L^{p'}(\lambda')}.$$
The proof for $\pi_{b;(1,0)}$ is symmetrical -- we take $S_{\mcd_2}\phi$, which will be bounded by $[SM]f \cdot S_{\dr}g$. The adjoint paraproducts 
$\pi_{b; (0,1)}^*$ and $\pi_{b; (1,0)}^*$ follow again by duality. Finally,
for $\gamma_{b; (0,1)}$:
	$$ \phi := \sum_{Q_1\times Q_2} \widehat{f}(Q_1^{\ep_1}\times Q_2^{\ep_2}) \frac{1}{\sqrt{|Q_1|}} \widehat{g}(Q_1^{\ep_1+\delta_1}\times Q_2^{\ep_2}) 
	h_{Q_1}^{\ep_1}\otimes \frac{\unit_{Q_2}}{|Q_2|},$$
from which it easily follows that $S_{\mcd_1}\phi \leq S_{\dr}f \cdot S_{\dr}g$. The proof for $\gamma_{b; (1,0)}$ is symmetrical.
\end{proof}

\section{Commutators with Journ\'{e} Operators}
\label{S:UB}

\subsection{Definition of Journ\'{e} Operators} \label{Ss:JourneDef}

We begin with the definition of biparameter Calder\'{o}n-Zygmund operators, or Journ\'{e} operators, on $\mbr^{\vn} := \mbr^{n_1} \otimes \mbr^{n_2}$,
 as outlined in \cite{MRep}. As shown later in \cite{Herran}, these conditions are
equivalent to the original definition of Journ\'{e} \cite{Journe}.

\vspace{0.05in}
\noindent I. \underline{Structural Assumptions:} Given $f = f_1\otimes f_2$ and $g = g_1\otimes g_2$, where $f_i, g_i : \mbr^{n_i} \rightarrow \mathbb{C}$ satisfy $spt(f_i) \cap spt(g_i) = \emptyset$ for $i = 1, 2$, we assume the kernel representation
	$$ \La Tf, g\Ra = \int_{\mbr^{\vn}} \int_{\mbr^{\vn}} K(x, y) f(y) g(x) \,dx\,dy. $$
The kernel $K : \mbr^{\vn} \times \mbr^{\vn} \setminus \{ (x, y) \in \mbr^{\vn}\times\mbr^{\vn}: x_1=y_1 \text{ or } x_2=y_2 \} \rightarrow \mathbb{C}$ is assumed to satisfy:
	\begin{enumerate}[1.]
	\item Size condition:
		$$ |K(x, y)| \leq C \frac{1}{|x_1-y_1|^{n_1}} \frac{1}{|x_2-y_2|^{n_2}}. $$
		
	\item H\"{o}lder conditions:
		\begin{enumerate}[2a.]
		\item If $|y_1 - y_1'| \leq \frac{1}{2}|x_1 - y_1|$ and $|y_2 - y_2'| \leq \frac{1}{2}|x_2-y_2|$:
		$$ \left| K(x,y) - K\big( x, (y_1, y_2') \big) - K\big( x, (y_1', y_2) \big) + K(x, y') \right| 
			\leq C \frac{|y_1-y_1'|^{\delta}}{|x_1-y_1|^{n_1+\delta}}  \frac{|y_2-y_2'|^{\delta}}{|x_2-y_2|^{n_2+\delta}}.$$
		\item If $|x_1 - x_1'| \leq \frac{1}{2}|x_1 - y_1|$ and $|x_2 - x_2'| \leq \frac{1}{2}|x_2-y_2|$:
		$$ \left| K(x,y) - K\big( (x_1, x_2'), y \big) - K\big( (x_1', x_2), y \big) + K(x', y) \right| 
			\leq C \frac{|x_1-x_1'|^{\delta}}{|x_1-y_1|^{n_1+\delta}}  \frac{|x_2-x_2'|^{\delta}}{|x_2-y_2|^{n_2+\delta}}.$$
		\item If $|y_1 - y_1'| \leq \frac{1}{2}|x_1 - y_1|$ and $|x_2 - x_2'| \leq \frac{1}{2}|x_2-y_2|$:
		$$ \left| K(x,y) - K\big( (x_1, x_2'), y \big) - K\big( x, (y_1', y_2) \big) + K\big( (x_1, x_2'), (y_1', y_2) \big) \right| 
			\leq C \frac{|y_1-y_1'|^{\delta}}{|x_1-y_1|^{n_1+\delta}}  \frac{|x_2-x_2'|^{\delta}}{|x_2-y_2|^{n_2+\delta}}.$$
		\item If $|x_1 - x_1'| \leq \frac{1}{2}|x_1 - y_1|$ and $|y_2 - y_2'| \leq \frac{1}{2}|x_2-y_2|$:
		$$ \left| K(x,y) - K\big( x, (y_1, y_2') \big) - K\big( (x_1',x_2), y \big) + K\big( (x_1', x_2), (y_1, y_2') \big) \right| 
			\leq C \frac{|x_1-x_1'|^{\delta}}{|x_1-y_1|^{n_1+\delta}}  \frac{|y_2-y_2'|^{\delta}}{|x_2-y_2|^{n_2+\delta}}.$$
		\end{enumerate}
		
	\item Mixed size and H\"{o}lder conditions:
	\begin{align}
	& \text{3a. If } |x_1 - x_1'| \leq \frac{1}{2}|x_1 - y_1| \text{, then } 
			\left| K(x,y) - K\big( (x_1', x_2), y \big) \right| \leq C \frac{|x_1 - x_1'|^{\delta}}{|x_1 - y_1|^{n_1+\delta}} \frac{1}{|x_2 - y_2|^{n_2}}.\\
	& \text{3b. If } |y_1 - y_1'| \leq \frac{1}{2}|x_1 - y_1| \text{, then }
			\left| K(x,y) - K\big( x, (y_1', y_2) \big) \right| \leq C \frac{|y_1 - y_1'|^{\delta}}{|x_1 - y_1|^{n_1+\delta}} \frac{1}{|x_2 - y_2|^{n_2}}.\\
	& \text{3c. If } |x_2 - x_2'| \leq \frac{1}{2} |x_2 - y_2| \text{, then }
			\left| K(x, y) - K\big( (x_1, x_2'), y \big) \right| \leq C \frac{1}{|x_1 - y_1|^{n_1}} \frac{|x_2 - x_2'|^{\delta}}{|x_2 - y_2|^{n_2+\delta}}.\\
	& \text{3d. If } |y_2 - y_2'| \leq \frac{1}{2} |x_2 - y_2| \text{, then }
			\left| K(x, y) - K\big( x, (y_1, y_2') \big) \right| \leq C \frac{1}{|x_1 - y_1|^{n_1}} \frac{|y_2 - y_2'|^{\delta}}{|x_2 - y_2|^{n_2+\delta}}.
	\end{align}
	
	\item Calder\'{o}n-Zygmund structure in $\mbr^{n_1}$ and $\mbr^{n_2}$ separately: If $f = f_1 \otimes f_2$ and $g = g_1 \otimes g_2$ with 
		$spt(f_1) \cap spt(g_1) = \emptyset$, we assume the kernel representation:
			$$ \La Tf, g\Ra = \int_{\mbr^{n_1}} \int_{\mbr^{n_1}} K_{f_2, g_2}(x_1, y_1) f_1(y_1) g_1(x_1) \,dx_1\,dy_1, $$
		where the kernel $K_{f_2, g_2} : \mbr^{n_1} \times \mbr^{n_1} \setminus \{ (x_1, y_1) \in \mbr^{n_1}\times\mbr^{n_1} : x_1=y_1 \}$
		satisfies the following size condition:
			$$ |K_{f_2, g_2} (x_1, y_1)| \leq C(f_2, g_2) \frac{1}{|x_1 - y_1|^{n_1}}, $$
		and H\"{o}lder conditions:
			$$\text{If } |x_1 - x_1'| \leq \frac{1}{2}|x_1-y_1| \text{, then }
				 \left| K_{f_2, g_2} (x_1, y_1) - K_{f_2, g_2}(x_1', y_1) \right| \leq C(f_2, g_2) \frac{|x_1 - x_1'|^{\delta}}{|x_1-y_1|^{n_1+\delta}},$$
			$$\text{If } |y_1 - y_1'| \leq \frac{1}{2}|x_1-y_1| \text{, then }
				 \left| K_{f_2, g_2} (x_1, y_1) - K_{f_2, g_2}(x_1, y_1') \right| \leq C(f_2, g_2) \frac{|y_1 - y_1'|^{\delta}}{|x_1-y_1|^{n_1+\delta}}.$$
		We only assume the above representation and a certain control over $C(f_2, g_2)$ in the diagonal, that is:
			$$ C(\unit_{Q_2}, \unit_{Q_2}) + C(\unit_{Q_2}, u_{Q_2}) + C(u_{Q_2}, \unit_{Q_2}) \leq C |Q_2|,$$
		for all cubes $Q_2\subset\mbr^{n_2}$ and all ``$Q_2$-adapted zero-mean'' functions $u_{Q_2}$ -- that is, $spt(u_{Q_2}) \subset Q_2$,
		$|u_{Q_2}|\leq 1$, and $\int u_{Q_2} = 0$. 
		We assume the symmetrical representation with kernel $K_{f_1, g_1}$ in the case $spt(f_2)\cap spt(g_2) = \emptyset$.
	\end{enumerate}

\vspace{0.05in}
\noindent II. \underline{Boundedness and Cancellation Assumptions:}
	\begin{enumerate}[1.]
	\item Assume $T1, T^*1, T_1(1)$ and $T_1^*(1)$ are in product $BMO(\mbr^{\vn})$, where $T_1$ is the partial adjoint of $T$, defined by
		$\La T_1(f_1\otimes f_2), g_1\otimes g_2\Ra = \La T(g_1\otimes f_2), f_1\otimes g_2\Ra$.
	\item Assume $|\La T(\unit_{Q_1} \otimes \unit_{Q_2}), \unit_{Q_1}\otimes\unit_{Q_2} \Ra | \leq C |Q_1| |Q_2|$, for all cubes $Q_i \subset \mbr^{n_i}$ (weak boundedness).
	\item Diagonal BMO conditions: for all cubes $Q_i\subset\mbr^{n_i}$ and all zero-mean functions $a_{Q_1}$ and $b_{Q_2}$
	that are $Q_1-$ and $Q_2-$ adapted, respectively, assume:
		\begin{align*}
		& |\La T(a_{Q_1} \otimes \unit_{Q_2}), \unit_{Q_1}\otimes\unit_{Q_2} \Ra| \leq C|Q_1||Q_2|,
		& |\La T(\unit_{Q_1} \otimes \unit_{Q_2}), a_{Q_1}\otimes\unit_{Q_2} \Ra| \leq C|Q_1||Q_2|,\\
		& |\La T(\unit_{Q_1} \otimes b_{Q_2}), \unit_{Q_1}\otimes\unit_{Q_2} \Ra| \leq C|Q_1||Q_2|,
		& |\La T(\unit_{Q_1} \otimes \unit_{Q_2}), \unit_{Q_1}\otimes b_{Q_2} \Ra| \leq C|Q_1||Q_2|.
		\end{align*}
	\end{enumerate}


\subsection{Biparameter Dyadic Shifts and Martikainen's Representation Theorem}
\label{Ss:RepThm}

Given dyadic rectangles $\dr = \mcd_1 \times \mcd_2$ and pairs of non-negative integers $\vec{i} = (i_1, i_2)$ and $\vec{j} = (j_1, j_2)$, 
a (cancellative) biparameter dyadic shift is an operator of the form:
	\begin{equation} \label{E:CShiftDef} 
	\mbs_{\dr}^{\vec{i}, \vec{j}} f := \sum_{\substack{R_1\in\mcd_1\\R_2\in\mcd_2}} 
		\sum_{\substack{P_1\in (R_1)_{i_1} \\ P_2\in (R_2)_{i_2}}} 
		\sum_{\substack{Q_1\in(R_1)_{j_1} \\ Q_2 \in (R_2)_{j_2} }}
		a_{P_1Q_1R_1P_2Q_2R_2} \: \widehat{f}(P_1^{\ep_1}\times P_2^{\ep_2}) \: h_{Q_1}^{\delta_1} \otimes h_{Q_2}^{\delta_2},
	\end{equation}
where
	$$ |a_{P_1Q_1R_1P_2Q_2R_2}| \leq \frac{\sqrt{|P_1||Q_1|}}{|R_1|} \frac{\sqrt{|P_2||Q_2|}}{|R_2|} = 2^{\frac{-n_1}{2}(i_1+j_1)} 
	2^{\frac{-n_2}{2}(i_2+j_2)}. $$
We suppress for now the signatures of the Haar functions, and assume summation over them is understood. We use the simplified notation
	$$ \mbs_{\dr}^{\vec{i}, \vec{j}} f := \sum_{\bR, \bP, \bQ}^{\vec{i}, \vec{j}} a_{\bP\bQ\bR} \: \widehat{f}(P_1\times P_2) \: h_{Q_1}\otimes h_{Q_2}$$
for the summation above. 

First note that
	\begin{align}
	S_{\mcd}^2 (\mbs_{\dr}^{\vec{i},\vec{j}}f) &= \sum_{R_1\times R_2} \sum_{\substack{Q_1\in(R_1)_{j_1}\\ Q_2\in(R_2)_{j_2}}}
		\bigg( \sum_{\substack{P_1\in (R_1)_{i_1} \\ P_2\in(R_2)_{i_2}}} 
		a_{P_1Q_1R_1P_2Q_2R_2} \: \widehat{f}(P_1\times P_2)
		\bigg)^2 \frac{\unit_{Q_1}}{|Q_1|} \otimes \frac{\unit_{Q_2}}{|Q_2|} \\
	& \lesssim 2^{-n_1(i_1+j_1)} 2^{-n_2(i_2+j_2)} \bigg( S_{\dr}^{\vec{i},\vec{j}} f \bigg)^2,
	\end{align}
where $S_{\dr}^{\vec{i}, \vec{j}}$ is the shifted biparameter square function in \eqref{E:ShiftedDSF2pDef}. Then, by \eqref{E:ShiftedDSF2p}:
	\begin{equation} \label{E:2pDShift1wt} 
	\|\mbs_{\dr}^{\vec{i}, \vec{j}} f \|_{L^p(w)} \lesssim 2^{\frac{-n_1}{2}(i_1+j_1)} 
	2^{\frac{-n_2}{2}(i_2+j_2)} \|S_{\dr}^{\vec{i},\vec{j}} f \|_{L^p(w)} \lesssim \|f\|_{L^p(w)},
	\end{equation}
for all $w \in A_p(\mbr^{\vn})$.

Next, we state Martikainen's Representation Theorem \cite{MRep}:

\begin{thm}[Martikainen]\label{T:MRep}
For a biparameter singular integral operator $T$ as defined in Section \ref{Ss:JourneDef}, 
there holds for some biparameter shifts $\mbs_{\dr}^{\vec{i},\vec{j}}$ that:
	$$ \La Tf, g\Ra = C_T \mathbb{E}_{\omega_1} \mathbb{E}_{\omega_2} 
	\sum_{\vec{i}, \vec{j} \in \mbz^2_{+}} 2^{-\text{max}(i_1, j_1) \delta/2} 2^{-\text{max}(i_2, j_2) \delta/2}  
	\La \mbs_{\dr}^{\vec{i},\vec{j}} f, g\Ra,$$
where non-cancellative shifts may only appear if $(i_1, j_1) = (0,0)$ or $(i_2, j_2)=(0,0)$.
\end{thm}

In light of this theorem, in order to prove Theorem \ref{T:UB}, it suffices to prove the two-weight bound for commutators
$[b, \mbs_{\dr}]$ with the dyadic shifts, with the requirements that the bounds be \textit{independent} of the choice of $\dr$
and that they depend on $\vec{i}$ and $\vec{j}$ at most \textit{polynomially}. We first look at the case of cancellative shifts,
and then treat the non-cancellative case in Section \ref{Ss:NCShifts}.

\subsection{Cancellative Case} \label{Ss:CShifts}

\begin{thm} \label{T:CShifts}
Let $\dr = \mcd_1\times\mcd_2$ be dyadic rectangles in $\mbr^{\vn} = \mbr^{n_1}\otimes\mbr^{n_2}$ and
$\mbs_{\dr}^{\vec{i}, \vec{j}}$ be a cancellative dyadic shift as defined in \eqref{E:CShiftDef}. If $\mu, \lb \in A_p(\mbr^{\vn})$, $1 < p < \infty$,
and $\nu = \mu^{1/p}\lb^{-1/p}$, then
	\begin{equation}
	\left\| [b, \mbs_{\mcd}^{\vec{i}, \vec{j}}] : L^p(\mu) \rightarrow L^p(\lb)\right\| \lesssim 
	\bigg( (1+\max(i_1, j_1))(1+\max(i_2, j_2)) \bigg) \|b\|_{bmo_{\dr}(\nu)},
	\end{equation}
where $\|b\|_{bmo_{\dr}(\nu)}$ denotes the norm of $b$ in the dyadic weighted little $bmo(\nu)$ space on $\mbr^{\vn}$.
\end{thm}

\begin{proof}
We may express the product of two functions $b$ and $f$ on $\mbr^{\vn}$ as
	\begin{equation} \label{E:2pParaprodDecomp}
	 bf = \sum \Pp f + \sum \pp f + \Pi_f b, 
	 \end{equation}
where $\Pp$ runs through the nine paraproducts associated with $BMO_{\dr}(\nu)$ in Section \ref{Ss:ProdPara}, and $\pp$ runs through the six paraproducts associated with $bmo_{\dr}(\nu)$ in Section \ref{Ss:bmoPara}. Then
	$$ [b, \mbs_{\dr}^{\vec{i}, \vec{j}}] f = \sum \: [\Pp, \mbs_{\dr}^{\vec{i}, \vec{j}}] f + \sum \: [\pp, \mbs_{\dr}^{\vec{i}, \vec{j}}] f + \mathcal{R}_{\vec{i},\vec{j}}f,$$
where
	$$ \mathcal{R}_{\vec{i},\vec{j}} f := \Pi_{\mbs_{\dr}^{\vec{i}, \vec{j}}f} b - \mbs_{\dr}^{\vec{i}, \vec{j}} \Pi_f b. $$
From the two-weight inequalities for the paraproducts in Propositions \ref{P:BMOparaprod} and \ref{P:bmoparaprod}, 
and the one-weight inequality for the shifts in \eqref{E:2pDShift1wt},
	$$ \left\| \sum \: [\Pp, \mbs_{\dr}^{\vec{i}, \vec{j}}] + \sum \: [\pp, \mbs_{\dr}^{\vec{i}, \vec{j}}] : L^p(\mu) \rightarrow L^p(\lb) \right\| \lesssim \|b\|_{bmo_{\dr}(\nu)},$$
so we are left with bounding the remainder term $\mathcal{R}_{\vec{i}, \vec{j}}$. We claim that:
	\begin{equation} \label{E:CSRem}
	\left\| \mathcal{R}_{\vec{i},\vec{j}} : L^p(\mu) \rightarrow L^p(\lb)\right\| \lesssim
	\bigg( (1+\max(i_1, j_1))(1+\max(i_2, j_2)) \bigg)  \|b\|_{bmo_{\dr}(\nu)},
	\end{equation}
from which the result follows.

A straightforward calculation shows that
	$$ \mathcal{R}_{\vec{i},\vec{j}} f = \sum_{\bR,\bP,\bQ}^{\vec{i}, \vec{j}} a_{\bP\bQ\bR} \widehat{f}(P_1 \times P_2)
		\bigg( \La b\Ra_{Q_1\times Q_2} - \La b\Ra_{P_1\times P_2} \bigg) h_{Q_1} \otimes h_{Q_2}.$$
We write this as a sum $\mathcal{R}_{\vec{i},\vec{j}} f = \mathcal{R}^1_{\vec{i},\vec{j}} f + \mathcal{R}^2_{\vec{i},\vec{j}} f$ by
splitting the term in parentheses as:
	$$ \La b\Ra_{Q_1\times Q_2} - \La b\Ra_{P_1\times P_2} = \bigg( \La b\Ra_{Q_1\times Q_2} - \La b\Ra_{R_1\times R_2} \bigg) 
	+ \bigg( \La b\Ra_{R_1\times R_2} - \La b\Ra_{P_1\times P_2} \bigg).$$
For the first term, we may apply the biparameter version of \eqref{E:1pavgdiff}, where we keep in mind that $R_1 = Q_1^{(j_1)}$ and $R_2 = Q_2^{(j_2)}$:
	\begin{align}
	\La b\Ra_{Q_1\times Q_2} - \La b\Ra_{R_1\times R_2} = & 
		\sum_{\substack{1\leq k_1\leq j_1 \\ 1\leq k_2\leq j_2}} \widehat{b}(Q_1^{(k_1)} \times Q_2^{(k_2)}) h_{Q_1^{(k_1)}}(Q_1) h_{Q_2^{(k_2)}}(Q_2)\\
	& + \sum_{1\leq k_1\leq j_1} \La b, h_{Q_1^{(k_1)}} \otimes \frac{\unit_{R_2}}{|R_2|} \Ra h_{Q_1^{(k_1)}}(Q_1)
	+ \sum_{1\leq k_2\leq j_2} \La b, \frac{\unit_{R_1}}{|R_1|}\otimes h_{Q_2^{(k_2)}} \Ra h_{Q_2^{(k_2)}}(Q_2).
	\end{align}
Then, we may write the operator $\mathcal{R}^1_{\vec{i},\vec{j}} $ as
	\begin{equation} \label{temp1} 
	\mathcal{R}^1_{\vec{i},\vec{j}} f = \sum_{\substack{1\leq k_1\leq j_1\\ 1\leq k_2\leq j_2}} A_{k_1,k_2}f + \sum_{1\leq k_1\leq j_1} B_{k_1}^{(0,1)}f + \sum_{1\leq k_2\leq j_2} B_{k_2}^{(1,0)}f,
	\end{equation}
where	
	$$ A_{k_1, k_2}f := \sum_{\bR,\bP,\bQ}^{\vec{i},\vec{j}} a_{\bP\bQ\bR} \widehat{f}(P_1\times P_2) \widehat{b}(Q_1^{(k_1)}\times Q_2^{(k_2)}) 
	h_{Q_1^{(k_1)}}(Q_1) h_{Q_2^{(k_2)}}(Q_2) h_{Q_1}\otimes h_{Q_2},$$
	$$ B_{k_1}^{(0,1)}f := \sum_{\bR,\bP,\bQ}^{\vec{i},\vec{j}} a_{\bP\bQ\bR} \widehat{f}(P_1\times P_2) 
	\La b, h_{Q_1^{(k_1)}} \otimes \frac{\unit_{R_2}}{|R_2|} \Ra h_{Q_1^{(k_1)}}(Q_1) h_{Q_1}\otimes h_{Q_2}, $$
and
	$$ B_{k_2}^{(1,0)}f := \sum_{\bR,\bP,\bQ}^{\vec{i},\vec{j}} a_{\bP\bQ\bR} \widehat{f}(P_1\times P_2) 
	\La b, \frac{\unit_{R_1}}{|R_1|} \otimes h_{Q_2^{(k_2)}} \Ra h_{Q_2^{(k_2)}}(Q_2) h_{Q_1}\otimes h_{Q_2}. $$
We show that these operators satisfy:
	$$ \left\| A_{k_1,k_2} : L^p(\mu) \rightarrow L^p(\lb) \right\| \lesssim \|b\|_{BMO_{\mcd}(\nu)}\text{, for all } k_1, k_2, $$
	$$ \left\| B_{k_1}^{(0,1)} : L^p(\mu) \rightarrow L^p(\lb) \right\| \lesssim \|b\|_{bmo_{\mcd}(\nu)} \text{, for all } k_1,
	\text{ and } \left\| B_{k_2}^{(1,0)} : L^p(\mu) \rightarrow L^p(\lb) \right\| \lesssim \|b\|_{bmo_{\mcd}(\nu)} \text{, for all } k_2.$$
Going back to the decomposition in \eqref{temp1}, these inequalities will give that
	$$ \left\| \mathcal{R}^1_{\vec{i},\vec{j}} : L^p(\mu) \rightarrow L^p(\lb) \right\| \lesssim (j_1j_2 + j_1 + j_2)\|b\|_{bmo_{\dr}(\nu)}.$$
A symmetrical proof for the term $\mathcal{R}^2_{\vec{i},\vec{j}}$ coming from $(\La b\Ra_{R_1\times R_2} - \La b\Ra_{P_1\times P_2})$ will show that
	$$ \left\| \mathcal{R}^2_{\vec{i},\vec{j}} : L^p(\mu) \rightarrow L^p(\lb) \right\| \lesssim (i_1i_2 + i_1 + i_2)\|b\|_{bmo_{\dr}(\nu)}.$$
Putting these estimates together, we obtain the desired result
	$$ \left\| \mathcal{R}_{\vec{i},\vec{j}} : L^p(\mu) \rightarrow L^p(\lb)\right\| \lesssim
	(i_1 + i_2 + i_1i_2 + j_1 + j_2 + j_1j_2)  \|b\|_{bmo_{\dr}(\nu)}  \lesssim
	(1+\max(i_1, j_1))(1+\max(i_2, j_2))  \|b\|_{bmo_{\dr}(\nu)}. $$
Remark that we are allowed to have one of the situations $(i_1, i_2) = (0,0)$ or $(j_1, j_2) = (0,0)$ -- but not both -- and then
either the term $\mathcal{R}^2_{\vec{i},\vec{j}} f$ or $\mathcal{R}^1_{\vec{i},\vec{j}} f$, respectively, will vanish.

Let us now look at the estimate for $A_{k_1, k_2}$. Taking again $f \in L^p(\mu)$ and $g \in L^{p'}(\lb')$, we write 
$\La A_{k_1, k_2}f, g\Ra = \La b, \phi\Ra$, where
	\begin{align}
	 \phi &= \sum_{\bR,\bP,\bQ}^{\vec{i},\vec{j}} a_{\bP\bQ\bR} \widehat{f}(P_1\times P_2) h_{Q_1^{(k_1)}}(Q_1) h_{Q_2^{(k_2)}}(Q_2) 
	\widehat{g}(Q_1\times Q_2) h_{Q_1^{(k_1)}}\otimes h_{Q_2^{(k_2)}}\\
	&= \sum_{R_1\times R_2} \sum_{\substack{P_1\in (R_1)_{i_1} \\ P_2\in(R_2)_{i_2}}}
	\sum_{\substack{N_1\in (R_1)_{j_1-k_1}\\ N_2\in(R_2)_{j_2-k_2}}} \widehat{f}(P_1\times P_2) 
	\bigg( \sum_{\substack{Q_1\in (N_1)_{k_1} \\ Q_2\in (N_2)_{k_2}}} a_{\bP\bQ\bR} \widehat{g}(Q_1\times Q_2) h_{N_1}(Q_1) h_{N_2}(Q_2) \bigg)
	h_{N_1}\otimes h_{N_2}.
	\end{align}
Then
	\begin{align*}
	S_{\dr}^2\phi &\lesssim \sum_{N_1\times N_2} \bigg(
	\sum_{\substack{P_1\in (N_1^{(j_1-k_1)})_{i_1} \\ P_2\in (N_2^{(j_2-k_2)})_{i_2} }} |\widehat{f}(P_1\times P_2)|
	\sum_{\substack{ Q_1\in (N_1)_{k_1} \\ Q_2\in (N_2)_{k_2} }} |a_{\bP\bQ\bR}| |\widehat{g}(Q_1\times Q_2)| \frac{1}{\sqrt{|N_1|}} \frac{1}{\sqrt{|N_2|}}
	\bigg)^2 \frac{\unit_{N_1}\otimes \unit_{N_2}}{|N_1| |N_2|} \\
	&\lesssim 2^{-n_1(i_1+j_1)} 2^{-n_2(i_2+j_2)} \sum_{N_1\times N_2} 
	\bigg(  \sum_{\substack{P_1\in (N_1^{(j_1-k_1)})_{i_1} \\ P_2\in (N_2^{(j_2-k_2)})_{i_2} }} |\widehat{f}(P_1\times P_2)| 2^{n_1k_1/2}2^{n_2k_2/2}
	\La |g|\Ra_{N_1\times N_2}
	\bigg)^2 \frac{\unit_{N_1}\otimes \unit_{N_2}}{|N_1| |N_2|}\\
	&\lesssim 2^{-n_1(i_1+j_1-k_1)} 2^{-n_2(i_2+j_2-k_2)} (M_S g)^2 
	\sum_{R_1\times R_2} \bigg( \sum_{\substack{ P_1\in (R_1)_{i_1}\\ P_2\in (R_2)_{i_2} }} |\widehat{f}(P_1\times P_2)| \bigg)^2
	\sum_{\substack{ N_1\in (R_1)_{j_1-k_1} \\ N_2\in (R_2)_{j_2-k_2} }} \frac{\unit_{N_1}\otimes\unit_{N_2}}{|N_1||N_2|}\\
	&=  2^{-n_1(i_1+j_1-k_1)} 2^{-n_2(i_2+j_2-k_2)} (M_S g)^2 \bigg( S_{\dr}^{(i_1, i_2), (j_1-k_1, j_2-k_2)} f \bigg)^2,
	\end{align*}
where the last operator is the shifted square function in \eqref{E:ShiftedDSF2pDef}. Then, from \eqref{E:ShiftedDSF2p}:
	\begin{align}
	\left\| A_{k_1, k_2} : L^p(\mu) \rightarrow L^p(\lb) \right\| & \lesssim \|b\|_{BMO_{\dr}(\nu)} \|S_{\dr}\phi\|_{L^1(\nu)}\\
	&\lesssim \|b\|_{BMO_{\dr}(\nu)} 2^{\frac{-n_1}{2}(i_1+j_1-k_1)} 2^{\frac{-n_2}{2}(i_2+j_2-k_2)} \|M_Sg\|_{L^{p'}(\lb')} \|S_{\dr}^{(i_1, i_2), (j_1-k_1, j_2-k_2)} f\|_{L^p(\mu)}\\
	&\lesssim \|b\|_{BMO_{\dr}(\nu)} \|g\|_{L^{p'}(\lb')} \|f\|_{L^p(\mu)}.
	\end{align}

Finally, we look at $B_{k_1}^{(0,1)}$, with the proof for $B_{k_2}^{(1,0)}$ being symmetrical. We write again $\La B_{k_1}^{(0,1)}f, g\Ra = \La b, \phi\Ra$, where
$$ \phi = \sum_{\bR, \bP, \bQ}^{\vec{i}, \vec{j}} a_{\bP\bQ\bR} \widehat{f}(P_1\times P_2) h_{Q_1^{(k_1)}}(Q_1) \widehat{g}(Q_1\times Q_2) h_{Q_1^{(k_1)}}\otimes
\frac{\unit_{R_2}}{|R_2|}.$$
Then
	\begin{equation}
	S_{\mcd_1}^2 f  \lesssim 2^{-n_1(i_1+j_1)} 2^{-n_2(i_2+j_2)} 
	\sum_{\substack{R_1\in\mcd_1 \\ N_1\in(R_1)_{j_1-k_1}}} \frac{\unit_{N_1}}{|N_1|}
	\bigg( \sum_{R_2\in\mcd_2} \sum_{\substack{P_1\in (R_1)_{i_1} \\ P_2\in (R_2)_{i_2}}} |\widehat{f}(P_1\times P_2)| 
	\sum_{Q_2\in (R_2)_{j_2}} \La |H_{Q_2}g|\Ra_{N_1} 2^{n_1k_1/2} \frac{\unit_{R_2}}{|R_2|}	\bigg)^2,
	\end{equation}
and the summation above is bounded  by:
	 $$\bigg(\sum_{\substack{R_1\in\mcd_1 \\ N_1\in(R_1)_{j_1-k_1}}} \frac{\unit_{N_1}}{|N_1|}
	\sum_{R_2\in\mcd_2} \big( 
	\sum_{\substack{P_1\in (R_1)_{i_1} \\ P_2\in (R_2)_{i_2} }}|\widehat{f}(P_1\times P_2)| \big)^2 \frac{\unit_{R_2}}{|R_2|}\bigg)\\
	\bigg( \sum_{R_2\in\mcd_2} \big(\sum_{Q_2\in(R_2)_{j_2}} M_{\mcd_1}(H_{Q_2}g)\big)^2 \frac{\unit_{R_2}}{|R_2|} \bigg),$$
which is exactly
	$$ \bigg( S_{\dr}^{(i_1, i_2), (j_1-k_1,0)} f\bigg)^2 \big([MS]^{j_2, 0}g\big)^2.$$	
From \eqref{E:ShiftedDSF2p} and  \eqref{E:Mixed2pSFShift}, we obtain exactly $\|S_{\mcd_1}\phi\|_{L^1(\nu)} \lesssim \|f\|_{L^p(\mu)} \|g\|_{L^{p'}(\lb')}$, and the proof is complete.
\end{proof}

\subsection{The Non-Cancellative Case} \label{Ss:NCShifts}

Following Martikainen's proof in \cite{MRep}, we are left with three types of terms to consider -- all of paraproduct type:
	\begin{itemize}
	\item The full standard paraproduct: $\Pi_a$ and $\Pi^*_a$,
	\item The full mixed paraproducts: $\Pi_{a;(0,1)}$ and $\Pi_{a;(1,0)}$,
	\end{itemize}
where, in each case, $a$ is some fixed function in \textit{unweighted} product $BMO(\mbr^{\vn})$, with $\|a\|_{BMO(\mbr^{\vn})} \leq 1$,
and
	\begin{itemize}
	\item The \textit{partial} paraproducts, defined for every $i_1, j_1 \geq 0$ as:
		$$ \mbs_{\dr}^{i_1, j_1} f := \sum_{\substack{R_1\in\mcd_1\\R_2\in\mcd_2}} 
		\sum_{\substack{P_1\in(R_1)_{i_1} \\ Q_1\in (R_1)_{j_1}}} \widehat{a}_{P_1Q_1R_1}(R_2^{\delta_2}) 
		\widehat{f}(P_1^{\ep_1} \times R_2^{\ep_2}) h_{Q_1}^{\delta_1} \times \frac{\unit_{R_2}}{|R_2|},$$
	where, for every fixed $P_1$, $Q_1$, $R_1$, $a_{P_1Q_1R_1}(x_2)$ is a $BMO(\mbr^{n_2})$ function with
		$$ \|a_{P_1Q_1R_1}\|_{BMO(\mbr^{n_2})} \leq \frac{\sqrt{|P_1|}\sqrt{|Q_1|}}{|R_1|} = 2^{\frac{-n_1}{2} (i_1+j_1)}, $$
	and 
	$$\widehat{a}_{P_1Q_1R_1}(R_2^{\delta_2}) := \La a_{P_1Q_1R_1}, h_{R_2}^{\delta_2}\Ra_{\mbr^{n_2}}
		:= \int_{\mbr^{n_2}} a_{P_1Q_1R_1}(x_2) h_{R_2}^{\delta_2}(x_2)\,dx_2.$$
	The symmetrical partial paraproduct $\mbs_{\dr}^{i_2, j_2}$ is defined analogously.
	\end{itemize}
We treat each case separately.


\subsubsection{The full standard paraproduct.}
In this case, we are looking at the commutator $[b, \Pi_a]$ where
	$$ \Pi_a f := \sum_{R\in\dr} \widehat{a}(R) \La f\Ra_{R} h_{R}, $$
and $a\in BMO_{\dr}(\mbr^{\vn})$ with $\|a\|_{BMO_{\dr}(\mbr^{\vn})} \leq 1$. We prove that

\begin{thm} \label{T:NCS-1}
Let $\mu, \lb \in A_p(\mbr^{\vn})$, $1<p<\infty$ and $\nu := \mu^{1/p}\lb^{-1/p}$. Then
	$$ \left\| [b, \Pi_a] : L^p(\mu) \rightarrow L^p(\lb) \right\| \lesssim \|a\|_{BMO_{\dr}(\mbr^{\vn})} \|b\|_{bmo_{\dr}(\nu)}.$$
\end{thm}

\begin{proof}
Remark first that
	$$ \Pi_a(bf) = \sum_{R\in\dr} \widehat{a}(R) \La bf\Ra_R h_R \text{ and }
		 \Pi_{\Pi_af}b = \sum_{R\in\dr} \widehat{a}(R) \La b\Ra_R \La f\Ra_R h_R, $$
so
	\begin{align}
	\Pi_a(bf) - \Pi_{\Pi_a f}b &= \sum_{R\in\dr} \widehat{a}(R) \big( \La bf\Ra_R - \La b\Ra_R \La f\Ra_R \big) h_R\\
		& = \Pi_a \big( \sum \Pp f + \sum \pp f + \Pi_f b \big) - \Pi_{\Pi_a f}b,
	\end{align}
where the last equality was obtained by simply expanding $bf$ into paraproducts.
Then
	$$ \Pi_{\Pi_a f}b - \Pi_a \Pi_f b = \sum \Pi_a \Pp f + \sum \Pi_a \pp f - 
	\sum_{R\in\dr} \widehat{a}(R) \big( \La bf\Ra_R - \La b\Ra_R \La f\Ra_R \big) h_R.$$
Noting that
	$$ [b, \Pi_a]f = \sum \Pp \Pi_a f + \sum \pp \Pi_a f - \sum \Pi_a\Pp f - \sum \Pi_a \pp f + \Pi_{\Pi_a f}b - \Pi_a \Pi_f b, $$
we obtain
	$$ [b, \Pi_a] f = \sum \Pp \Pi_a f + \sum \pp \Pi_a f - \sum_{R\in\dr} \widehat{a}(R) \big( \La bf\Ra_R - \La b\Ra_R \La f\Ra_R\big) h_R. $$
The first terms are easily handled:
	$$ \|\Pp \Pi_a f\|_{L^p(\lb)} \lesssim \|b\|_{BMO_{\dr}(\nu)} \|\Pi_af\|_{L^p(\mu)} \lesssim \|b\|_{BMO_{\dr}(\nu)} \|a\|_{BMO_{\dr}(\mbr^{\vn})} \|f\|_{L^p(\mu)}, $$
	$$ \|\pp \Pi_a f\|_{L^p(\lb)} \lesssim \|b\|_{bmo_{\dr}(\nu)} \|\Pi_af\|_{L^p(\mu)} \lesssim \|b\|_{bmo_{\dr}(\nu)} \|a\|_{BMO_{\dr}(\mbr^{\vn})} \|f\|_{L^p(\mu)}. $$
So we are left with the third term.

Now, for any dyadic rectangle $R$:
	$$ \La bf\Ra_R - \La b\Ra_R\La f\Ra_R = \frac{1}{|R|} \int_R f(x) \unit_R(x) (b(x) - \La b\Ra_R)\,dx. $$
Expressing $\unit_R(b - \La b\Ra_R)$ as in \eqref{E:2punitmo}, we obtain
	\begin{align}
	 \La bf\Ra_R - \La b\Ra_R\La f\Ra_R =& \frac{1}{|R|} \sum_{\substack{P_1\subset Q_1 \\ P_2\subset Q_2}} \widehat{b}(P_1\times P_2) \widehat{f}(P_1\times P_2)\\
	 	&+  \frac{1}{|R|} \sum_{P_1\subset Q_1} \La b, h_{P_1}\otimes \frac{\unit_{Q_2}}{|Q_2|}\Ra \La f, h_{P_1}\otimes \unit_{Q_2}\Ra
		+ \frac{1}{|R|} \sum_{P_2\subset Q_2} \La b, \frac{\unit_{Q_1}}{|Q_1|}\otimes h_{P_2}\Ra \La f, \unit_{Q_1}\otimes h_{P_2}\Ra.
	 \end{align}
Therefore
	$$ \sum_{R\in\dr} \widehat{a}(R) \big( \La bf\Ra_R - \La b\Ra_R \La f\Ra_R\big) h_R =
		\Lambda_{a, b}f + \lambda_{a,b}^{(0,1)}f + \lambda_{a,b}^{(1,0)}f, $$
where:		
	$$ \Lambda_{a,b}f := \sum_{Q_1\times Q_2} \widehat{a}(Q_1\times Q_2) \frac{1}{|Q_1||Q_2|} 
		\bigg( \sum_{\substack{P_1\subset Q_1 \\ P_2\subset Q_2}} \widehat{b}(P_1\times P_2) \widehat{f}(P_1\times P_2) \bigg)  h_{Q_1}\otimes h_{Q_2},$$	
	$$ \lambda_{a,b}^{(0,1)} f := \sum_{Q_1\times Q_2} \widehat{a}(Q_1\times Q_2) \frac{1}{|Q_1||Q_2|} 
		\bigg( \sum_{P_1\subset Q_1} \La b, h_{P_1}\otimes \frac{\unit_{Q_2}}{|Q_2|}\Ra 
		\La f, h_{P_1}\otimes \unit_{Q_2}\Ra \bigg) h_{Q_1}\otimes h_{Q_2}, $$
	$$ \lambda_{a,b}^{(1,0)} f := \sum_{Q_1\times Q_2} \widehat{a}(Q_1\times Q_2) \frac{1}{|Q_1||Q_2|} 
		\bigg( \sum_{P_2\subset Q_2}  \La b, \frac{\unit_{Q_1}}{|Q_1|}\otimes h_{P_2}\Ra \La f, \unit_{Q_1}\otimes h_{P_2}\Ra \bigg) h_{Q_1}\otimes h_{Q_2}. $$

To analyze the term $\Lambda_{a,b}$, we write $\La \Lambda_{a,b}f, g\Ra = \La b, \phi\Ra$, where
	\begin{align}
	 \phi &= \sum_{P_1\times P_2} \widehat{f}(P_1\times P_2) 
	\bigg( \sum_{\substack{Q_1\supset P_1\\ Q_2\supset P_2}} \widehat{a}(Q_1\times Q_2) \widehat{g}(Q_1\times Q_2) \frac{1}{|Q_1||Q_2|} \bigg) h_{P_1}\otimes h_{P_2}\\
	&= \sum_{R\in\dr} \widehat{f}(R) \bigg( \sum_{T\in\dr: T\supset R} \widehat{a}(T)\widehat{g}(T)\frac{1}{|T|} \bigg) h_R.
	\end{align}
So $|\La\Lambda_{a,b}f, g\Ra| \lesssim \|b\|_{BMO_{\dr}(\nu)}\|S_{\dr}\phi\|_{L^1(\nu)}$, and
	$$ S_{\dr}^2\phi = \sum_{R\in\dr} |\widehat{f}(R)|^2 \bigg( \sum_{T\in\dr: T\supset R} \widehat{a}(T) \widehat{g}(T) \frac{1}{|T|} \bigg)^2 \frac{\unit_R}{|R|}
	\leq \sum_{R\in\dr} |\widehat{f}(R)|^2 \bigg(  \sum_{T\in\dr: T\supset R} \widehat{a}_{\tau}(T) \widehat{g}_{\tau}(T)\frac{1}{|T|} \bigg)^2 \frac{\unit_R}{|R|}, $$	
where $a_{\tau}: = \sum_{R\in\dr} |\widehat{a}(R)|h_R$ and $g_{\tau} := \sum_{R\in\dr} |\widehat{g}(R)|h_R$ are martingale transforms which do not increase
either the $BMO$ norm of $a$, or the $L^{p'}(\lambda')$ norm of $g$.
Now note that
	$$ \La \Pi_{a_{\tau}}^*g_{\tau}\Ra_R = \sum_{T \subsetneq R} \widehat{a}_{\tau}(T) \widehat{g}_{\tau}(T) \frac{1}{|R|} + \sum_{T \supset R} 
	\widehat{a}_{\tau}(T) \widehat{g}_{\tau}(T) \frac{1}{|T|},$$
and since all the Haar coefficients of $a_{\tau}$ and $g_{\tau}$ are non-negative, we may write
	$$ \sum_{T\supset R} \widehat{a}_{\tau}(T) \widehat{g}_{\tau}(T)\frac{1}{|T|} \leq \La \Pi_{a_{\tau}}^*g_{\tau}\Ra_R.$$
Then
	$$
	S_{\dr}^2\phi \leq \sum_{R\in\dr} |\widehat{f}(R)|^2 \La \Pi_{a_{\tau}}^* g_{\tau}\Ra_R^2 \frac{\unit_R}{|R|} 
	\leq \big( M_S \Pi_{a_{\tau}}^* g_{\tau} \big)^2 S_{\dr}^2f,
	$$
and
	\begin{align}
	 \|S_{\dr}\phi\|_{L^1(\nu)} &\leq \|M_S \Pi_{a_{\tau}}^* g_{\tau}\|_{L^{p'}(\lambda')} \|S_{\dr}f\|_{L^p(\mu)} \\
	 & \lesssim \|\Pi_{a_{\tau}}^* g_{\tau} \|_{L^{p'}(\lambda')} \|f\|_{L^p(\mu)}\\
	 & \lesssim \|a_{\tau}\|_{BMO_{\dr}(\mbr^{\vn})} \|g_{\tau}\|_{L^{p'}(\lb')} \|f\|_{L^p(\mu)},
	 \end{align}
which gives us the desired estimate
	$$ \left\| \Lambda_{a,b} : L^p(\mu) \rightarrow L^p(\lb) \right\| \lesssim \|a\|_{BMO_{\dr}(\mbr^{\vn})} \|b\|_{BMO_{\dr}(\nu)}.$$

Finally, we analyze the term $\lambda_{a,b}^{(0,1)}$, with the last term being symmetrical. We have
$\La \lambda_{a,b}^{(0,1)}f, g\Ra = \La b, \phi\Ra$ with
	$$ \phi = \sum_{P_1} \bigg(
	\sum_{P_2} \La f, h_{P_1}\otimes \unit_{P_2}\Ra \frac{1}{|P_2|} \sum_{Q_1\supset P_1} \widehat{a}(Q_1\times P_2)
	\widehat{g}(Q_1\times P_2) \frac{1}{|Q_1|} \frac{\unit_{P_2}}{|P_2|}
	\bigg) h_{P_1}, $$
and $| \La \lambda_{a,b}^{(0,1)} f, g\Ra | \lesssim \|b\|_{bmo_{\dr}(\nu)} \|S_{\mcd_1}\phi\|_{L^1(\nu)}$.
Now
	$$ S_{\mcd_1}^2\phi \leq \sum_{P_1} \bigg(
	\sum_{P_2} \La |H_{P_1}f|\Ra_{P_2} \big( 
	\sum_{Q_1\supset P_1} \widehat{a}_{\tau}(Q_1\times P_2) \widehat{g}_{\tau}(Q_2\times P_2) \frac{1}{|Q_1|} \big) \frac{\unit_{P_2}}{|P_2|}
	\bigg)^2 \frac{\unit_{P_1}}{|P_1|}, $$
where we are using the same martingale transforms as above. Note that
	$$ \La \Pi_{a_{\tau}}^* g_{\tau}, \frac{\unit_{P_1}}{|P_1|}\Ra_{\mbr^{n_1}}(x_2) = \sum_{P_2} \frac{\unit_{P_2}(x_2)}{|P_2|}
		\sum_{Q_1} \widehat{a}_{\tau}(Q_1\times P_2) \widehat{g}_{\tau}(Q_1\times P_2) \frac{|Q_1\cap P_1|}{|Q_1||P_1|}, $$
and again since all terms are non-negative:
	\begin{align}
	S_{\mcd_1}^2\phi & \leq \sum_{P_1} M_{\mcd_2}^2(H_{P_1}f)(x_2) 
		\bigg( \sum_{Q_1\supset P_1} \sum_{P_2} \widehat{a}_{\tau}(Q_1\times P_2) \widehat{g}_{\tau}(Q_1\times P_2) \frac{1}{|Q_1|} \frac{\unit_{P_2}(x_2)}{|P_2|}
		\bigg)^2 \frac{\unit_{P_1}(x_1)}{|P_1|}\\
	&\leq \sum_{P_1} M_{\mcd_2}^2(H_{P_1}f)(x_2)  \bigg( \La \Pi_{a_{\tau}}^* g_{\tau}, \frac{\unit_{P_1}}{|P_1|}\Ra_{\mbr^{n_1}}(x_2) \bigg)^2 \frac{\unit_{P_1}(x_1)}{|P_1|}\\
	&\leq \bigg( M_{\mcd_1}(\Pi_{a_{\tau}}^* g_{\tau}) (x_1, x_2) \bigg)^2 \sum_{P_1} M_{\mcd_2}^2(H_{P_1}f)(x_2) \frac{\unit_{P_1}(x_1)}{|P_1|}
	= \bigg( M_{\mcd_1}(\Pi_{a_{\tau}}^* g_{\tau}) (x_1, x_2) \bigg)^2 \bigg( [SM]f(x_1, x_2) \bigg)^2.
	\end{align}
Then
	$$ \|S_{\mcd_1}\phi \|_{L^1(\nu)} \lesssim \|\Pi_{a_{\tau}}^* g_{\tau}\|_{L^{p'}(\lb')} \|[SM]f\|_{L^p(\mu)} \lesssim \|a\|_{BMO_{\dr}(\mbr^{\vn})} \|g\|_{L^{p'}(\lb')}\|f\|_{L^p(\mu)},$$
and so
	$$ \left\|\lambda_{a,b}^{(0,1)} : L^p(\mu) \rightarrow L^p(\lb)  \right\| \lesssim \|a\|_{BMO_{\dr}(\nu)} \|b\|_{bmo_{\dr}(\nu)},$$
and the proof is complete.
\end{proof}

\subsubsection{The full mixed paraproduct}
We are now dealing with $[b, \Pi_{a;(0,1)}]$, where
	$$ \Pi_{a;(0,1)}f := \sum_{P_1\times P_2} \widehat{a}(P_1\times P_2) \La f, h_{P_1} \otimes \frac{\unit_{P_2}}{|P_2|}\Ra \frac{\unit_{P_1}}{|P_1|} \otimes h_{P_2}. $$
	
\begin{thm} \label{T:NCS-2}
Let $\mu, \lb \in A_p(\mbr^{\vn})$, $1<p<\infty$ and $\nu := \mu^{1/p}\lb^{-1/p}$. Then
	$$ \left\| [b, \Pi_{a; (0,1)}] : L^p(\mu) \rightarrow L^p(\lb) \right\| \lesssim \|a\|_{BMO_{\dr}(\mbr^{\vn})} \|b\|_{bmo_{\dr}(\nu)}.$$
\end{thm}
Note that the case $[b, \Pi_{a;(1,0)}]$ follows symmetrically.

\begin{proof}
By the standard considerations, we only need to bound the remainder term
	$$ \mathcal{R}^{(0,1)}_{a,b} f := \Pi_{\Pi_{a;(0,1)}f}b - \Pi_{a;(0,1)} \Pi_f b. $$
Explicitly, these terms are:
\begin{align}
& \Pi_{\Pi_{a;(0,1)}f}b = \sum_{P_1\times P_2} \widehat{a}(P_1^{\ep_1}\times P_2^{\ep_2}) \La f, h_{P_1}^{\ep_1} \otimes \frac{\unit_{P_2}}{|P_2|}\Ra 
	\bigg( \sum_{Q_1\supsetneq P_1} \La b\Ra_{Q_1\times P_2} h_{Q_1}^{\delta_1}(P_1) h_{Q_1}^{\delta_1}(x_1)  \bigg) h_{P_2}^{\ep_2}(x_2), \\
& \Pi_{a;(0,1)} \Pi_f b = \sum_{P_1\times P_2} \widehat{a}(P_1^{\ep_1}\times P_2^{\ep_2}) 
\bigg( \sum_{Q_2\supsetneq P_2} \widehat{f}(P_1^{\ep_1}\times Q_2^{\delta_2}) \La b\Ra_{P_1\times Q_2} h_{Q_2}^{\delta_2}(P_2) \bigg)
\frac{\unit_{P_1}(x_1)}{|P_1|} \otimes h_{P_2}^{\ep_2}(x_2).
\end{align}

Consider now a third term
	$$ T := \sum_{P_1\times P_2} \widehat{a}(P_1^{\ep_1}\times P_2^{\ep_2}) \La b\Ra_{P_1\times P_2} 
	\La f, h_{P_1}^{\ep_1}\otimes \frac{\unit_{P_2}}{|P_2|}\Ra \frac{\unit_{P_1}}{|P_1|}\otimes h_{P_2}^{\ep_2}. $$
Using the one-parameter formula:
	$$ \frac{\unit_{P_1}(x_1)}{|P_1|} = \sum_{Q_1\supsetneq P_1} h_{Q_1}^{\delta_1} (P_1) h_{Q_1}^{\delta_1}(x_1), $$
we write $T$ as
	$$ T =  \sum_{P_1\times P_2} \widehat{a}(P_1^{\ep_1}\times P_2^{\ep_2}) \La f, h_{P_1}^{\ep_1} \otimes \frac{\unit_{P_2}}{|P_2|}\Ra 
		\bigg( \sum_{Q_1\supsetneq P_1} \La b\Ra_{P_1\times P_2} h_{Q_1}^{\delta_1}(P_1) h_{Q_1}^{\delta_1}(x_1) \bigg) h_{P_2}^{\ep_2}(x_2),$$
allowing us to combine this term with $\Pi_{\Pi_{a;(0,1)}f}b$:
	$$ \Pi_{\Pi_{a;(0,1)}f}b - T =  
	\sum_{P_1\times P_2} \widehat{a}(P_1^{\ep_1}\times P_2^{\ep_2}) \La f, h_{P_1}^{\ep_1} \otimes \frac{\unit_{P_2}}{|P_2|}\Ra 
		\bigg( \sum_{Q_1\supsetneq P_1} \big(\La b\Ra_{Q_1\times P_2} - \La b\Ra_{P_1\times P_2}\big) h_{Q_1}^{\delta_1}(P_1) h_{Q_1}^{\delta_1}(x_1) \bigg) h_{P_2}^{\ep_2}(x_2).$$
Using \eqref{E:1pavgdiff}:
	$$ \La b\Ra_{Q_1\times P_2} - \La b\Ra_{P_1\times P_2} = - \sum_{R_1: P_1\subsetneq R_1\subset Q_1} 
	\La b, h_{R_1}^{\tau_1} \otimes \frac{\unit_{P_2}}{|P_2|}\Ra h_{R_1}^{\tau_1}(P_1). $$		
Then the term in parentheses above becomes
	\begin{equation}\label{temp2} 
	-\sum_{Q_1\supsetneq P_1} \bigg( \sum_{R_1: P_1\subsetneq R_1\subset Q_1} \La b, h_{R_1}^{\tau_1} \otimes \frac{\unit_{P_2}}{|P_2|}\Ra 
	h_{R_1}^{\tau_1}(P_1) \bigg) h_{Q_1}^{\delta_1}(P_1) h_{Q_1}^{\delta_1}(x_1).
	\end{equation}	
Next, we analyze this term depending on the relationship between $R_1$ and $Q_1$:

\vspace{0.05in}
\noindent \underline{Case 1: $R_1\subsetneq Q_1$:} Then we may rewrite the sum as
	\begin{align}
	& \sum_{R_1\supsetneq P_1} \La b, h_{R_1}^{\tau_1} \otimes \frac{\unit_{P_2}}{|P_2|}\Ra h_{R_1}^{\tau_1}(P_1) 
	\sum_{Q_1\supsetneq R_1} \underbrace{h_{Q_1}^{\delta_1}(P_1)}_{= h_{Q_1}^{\delta_1}(R_1)} h_{Q_1}^{\delta_1}(x_1)
	=  \sum_{R_1\supsetneq P_1} \La b, h_{R_1}^{\tau_1} \otimes \frac{\unit_{P_2}}{|P_2|}\Ra h_{R_1}^{\tau_1}(P_1) 
		\frac{\unit_{R_1}(x_1)}{|R_1|}.
	\end{align}
This then leads to
	\begin{align}
	& \sum_{P_1\times P_2} \widehat{a}(P_1^{\ep_1}\times P_2^{\ep_2}) \La f, h_{P_1}^{\ep_1} \otimes \frac{\unit_{P_2}}{|P_2|}\Ra
	\bigg( \sum_{R_1\supsetneq P_1} \La b, h_{R_1}^{\tau_1} \otimes \frac{\unit_{P_2}}{|P_2|}\Ra h_{R_1}^{\tau_1}(P_1) 
		\frac{\unit_{R_1}(x_1)}{|R_1|} \bigg) h_{P_2}^{\ep_2}(x_2) \\
	=& \sum_{R_1\times P_2} \La b, h_{R_1}^{\tau_1} \otimes \frac{\unit_{P_2}}{|P_2|}\Ra 
	\bigg( \sum_{P_1\subsetneq R_1} \widehat{a}(P_1^{\ep_1}\times P_2^{\ep_2}) \La f, h_{P_1}^{\ep_1}\otimes\frac{\unit_{P_2}}{|P_2|}\Ra h_{R_1}^{\tau_1}(P_1)
	 \bigg) \frac{\unit_{R_1}(x_1)}{|R_1|}\otimes h_{P_2}^{\ep_2}(x_2)\\
	 =& \sum_{R_1\times P_2} \La b, h_{R_1}^{\tau_1} \otimes \frac{\unit_{P_2}}{|P_2|} \Ra
	 \La \Pi_{a;(0,1)} f, h_{R_1}^{\tau_1}\otimes h_{P_2}^{\ep_2}\Ra \frac{\unit_{R_1}(x_1)}{|R_1|}\otimes h_{P_2}^{\ep_2}(x_2)\\
	 =& \pi_{b;(0,1)}^* \Pi_{a;(0,1)}f.
	\end{align}

\vspace{0.05in}
\noindent \underline{Case 2(a): $R_1 = Q_1$ and $\tau_1\neq \delta_1$:} Then \eqref{temp2} becomes:
	$$ -\sum_{Q_1\supsetneq P_1}  \La b, h_{Q_1}^{\tau_1} \otimes \frac{\unit_{P_2}}{|P_2|}\Ra  \frac{1}{\sqrt{|Q_1|}} h_{Q_1}^{\tau_1+\delta_1}(P_1) h_{Q_1}^{\delta_1}(x_1),$$
which leads to
	\begin{align}
	& \sum_{Q_1\times P_2} \La b, h_{Q_1}^{\tau_1} \otimes \frac{\unit_{P_2}}{|P_2|}\Ra \frac{1}{\sqrt{|Q_1|}} h_{Q_1}^{\delta_1}(x_1) h_{P_2}^{\ep_2}(x_2)
	\sum_{P_1\subsetneq Q_1} \widehat{a}(P_1^{\ep_1}\times P_2^{\ep_2}) \La f, h_{P_1}^{\ep_1}\otimes \frac{\unit_{P_2}}{|P_2|}\Ra h_{Q_1}^{\tau_1+\delta_1}(P_1)\\
	=& \sum_{Q_1\times P_2} \La b, h_{Q_1}^{\tau_1} \otimes \frac{\unit_{P_2}}{|P_2|}\Ra
	\La \Pi_{a;(0,1)}f, h_{Q_1}^{\tau_1+\delta_1}\otimes h_{P_2}^{\ep_2}\Ra \frac{1}{\sqrt{|Q_1|}} h_{Q_1}^{\delta_1}(x_1) \otimes h_{P_2}^{\ep_2}(x_2)\\
	=& \gamma_{b; (0,1)} \Pi_{a;(0,1)}f.
	\end{align}
	
\vspace{0.05in}
\noindent \underline{Case 2(b): $R_1 = Q_1$ and $\tau_1 = \delta_1$:} Then \eqref{temp2} becomes:
	$$ \sum_{Q_1\supsetneq P_1}  \La b, h_{Q_1}^{\delta_1} \otimes \frac{\unit_{P_2}}{|P_2|}\Ra \frac{1}{|Q_1|} h_{Q_1}^{\delta_1},$$
which gives rise to the term
	$$ T_{a,b}^{(0,1)}f :=  \sum_{Q_1\times P_2} \La b, h_{Q_1}^{\delta_1} \otimes \frac{\unit_{P_2}}{|P_2|}\Ra
	h_{Q_1}^{\delta_1}(x_1) h_{P_2}^{\ep_2} (x_2) \frac{1}{|Q_1|} \sum_{P_1\subsetneq Q_1} \widehat{a}(P_1^{\ep_1}\times P_2^{\ep_2})
	\La f, h_{P_1}^{\ep_1}\otimes \frac{\unit_{P_2}}{|P_2|} \Ra. $$
We have proved that
	\begin{equation}\label{temp3}
	\Pi_{\Pi_{a;(0,1)}f}b - T = - \pi_{b; (0,1)}^* \Pi_{a;(0,1)}f - \gamma_{b; (0,1)} \Pi_{a;(0,1)}f - T_{a,b}^{(0,1)}f.
	\end{equation}
Expressing $T$	 instead as
	$$ T = \sum_{P_1\times P_2} \widehat{a}(P_1^{\ep_1}\times P_2^{\ep_2})
	\bigg( \sum_{Q_2\supsetneq P_2} \widehat{f}(P_1^{\ep_1}\times Q_2^{\delta_2}) \La b\Ra_{P_1\times P_2} h_{Q_2}^{\delta_2}(P_2) \bigg)
	\frac{\unit_{P_1}}{|P_1|}\otimes h_{P_2}^{\ep_2}, $$
we are able to pair it with $\Pi_{a;(0,1)}\Pi_f b$. Then, a similar analysis yields
	\begin{equation}\label{temp4}
	T - \Pi_{a;(0,1)} \Pi_f b = \Pi_{a;(0,1)} \pi_{b;(1,0)}f + \Pi_{a;(0,1)} \gamma_{b;(1,0)}f + T_{a,b}^{(1,0)}f,
	\end{equation}
where
	$$ T_{a,b}^{(1,0)}f := \sum_{P_1\times P_2} \widehat{a}(P_1^{\ep_1}\times P_2^{\ep_2}) \frac{\unit_{P_1}(x_1)}{|P_1|}\otimes h_{P_2}^{\ep_2}(x_2)
	\bigg( \sum_{Q_2\supsetneq P_2} \La b, \frac{\unit_{P_1}}{|P_1|}\otimes h_{Q_2}^{\delta_2}\Ra \widehat{f}(P_1^{\ep_1}\times Q_2^{\delta_2}) \frac{1}{|Q_2|} \bigg). $$
Then
	$$ \mathcal{R}_{a,b}^{(0,1)} f = \Pi_{a;(0,1)} \pi_{b;(1,0)}f + \Pi_{a;(0,1)} \gamma_{b;(1,0)}f  
	- \pi_{b; (0,1)}^* \Pi_{a;(0,1)}f - \gamma_{b; (0,1)} \Pi_{a;(0,1)}f  + T_{a,b}^{(1,0)}f - T_{a,b}^{(0,1)}f.$$	
It is now obvious that the first four terms are bounded as desired, and it remains to bound the terms $T_{a,b}$.

We look at $T_{a,b}^{(0,1)}$, for which we can write $\La T_{a,b}^{(0,1)}f, g\Ra = \La b, \phi\Ra$, where
	$$ \phi = \sum_{Q_1\times P_2} \widehat{g}(Q_1^{\delta_1}\times P_2^{\ep_2}) \frac{1}{|Q_1|}
	\bigg( \sum_{P_1\subsetneq Q_1} \widehat{a}(P_1^{\ep_1}\times P_2^{\ep_2}) 
	\La f, h_{P_1}^{\ep_1}\otimes \frac{\unit_{P_2}}{|P_2|}\Ra  \bigg) h_{Q_1}^{\delta_1}\otimes \frac{\unit_{P_2}}{|P_2|}.$$	
Then $|\La T_{a,b}^{(0,1)}f, g \Ra| \lesssim \|b\|_{bmo_{\dr}(\nu)} \|S_{\mcd_1}\phi\|_{L^1(\nu)}$, and
	$$ S_{\mcd_1}^2\phi = \sum_{Q_1} \bigg(  
	\sum_{P_2} \widehat{g}(Q_1^{\delta_1}\times P_2^{\ep_2}) \bigg( 
	\frac{1}{|Q_1|} \sum_{P_1\subsetneq Q_1} \widehat{a}(P_1^{\ep_1}\times P_2^{\ep_2})
	\La f, h_{P_1}^{\ep_2}\otimes \frac{\unit_{P_2}}{|P_2|}\Ra
	\bigg) \frac{\unit_{P_2}(x_2)}{|P_2|}
	\bigg)^2 \frac{\unit_{Q_1}(x_1)}{|Q_1|}. $$
Now,
	$$ \La \Pi_{a;(0,1)}f, \frac{\unit_{Q_1}}{|Q_1|} \otimes h_{P_2}^{\ep_2}\Ra = \sum_{P_1} \widehat{a}(P_1^{\ep_1}\times P_2^{\ep_2})
	\La f, h_{P_1}^{\ep_1}\otimes \frac{\unit_{P_2}}{|P_2|}\Ra \frac{|P_1\cap Q_1|}{|P_1||Q_1|}.$$
Define the martingale transform $a \mapsto a_{\tau} = \sum_{P_1\times P_2} \tau_{P_1,P_2}^{\ep_1,\ep_2} \widehat{a}(P_1^{\ep_1}\times P_2^{\ep_2})$,
where
	$$
	\tau_{P_1,P_2}^{\ep_1,\ep_2} = \left\{ \begin{array}{l}
	+1 \text{, if } \La f, h_{P_1}^{\ep_1}\otimes \frac{\unit_{P_2}}{|P_2|}\Ra \geq 0\\
	-1 \text{, otherwise.}
	\end{array}\right.
	$$
Remark that, while this transform does depend on $f$, in the end it will not matter, as this will be absorbed into the product BMO norm of $a_{\tau}$.
Then we have
	$$ \frac{1}{|Q_1|} \left| \sum_{P_1\subsetneq Q_1} \widehat{a}(P_1^{\ep_1}\times P_2^{\ep_2})
	\La f, h_{P_1}^{\ep_1}\otimes \frac{\unit_{P_2}}{|P_2|}\Ra  \right| \leq \La \Pi_{a_{\tau};(0,1)}f, \frac{\unit_{Q_1}}{|Q_1|}\otimes h_{P_2}^{\ep_2} \Ra.$$
Returning to the square function estimate, we now have
	\begin{align}
	S_{\mcd_1}^2\phi &\leq \sum_{Q_1} \bigg( \sum_{P_2} |\widehat{g}(Q_1^{\delta_1}\times P_2^{\ep_2})|^2 \frac{\unit_{P_2}(x_2)}{|P_2|} \bigg)
	\bigg( \sum_{P_2} \La |H_{P_2}^{\ep_2} \Pi_{a_{\tau};(0,1)} f| \Ra^2_{Q_1} \unit_{Q_1}(x_1) \frac{\unit_{P_2}(x_2)}{|P_2|} \bigg) \frac{\unit_{Q_1}(x_1)}{|Q_1|}\\
	&\leq S_{\dr}^2g \bigg( \sum_{P_2} M_{\mcd_1}^2 (H_{P_2}^{\ep_2} \Pi_{a_{\tau};(0,1)} f)(x_1) \frac{\unit_{P_2}(x_2)}{|P_2|} \bigg)
	 = S_{\dr}^2 g \bigg( [MS] \Pi_{a_{\tau};(0,1)}f \bigg)^2.
	\end{align}
Finally,
	\begin{align}
	\|S_{\mcd_1}\phi\|_{L^1(\nu)} &\leq \|S_{\dr}g\|_{L^{p'}(\lb')} \left\| [MS] \Pi_{a_{\tau};(0,1)}f \right\|_{L^p(\mu)}\\
	&\lesssim \|g\|_{L^{p'}(\lb')} \underbrace{ \|\Pi_{a_{\tau};(0,1)}f\|_{L^p(\mu)} }_{\lesssim \|a_{\tau}\|_{BMO_{\dr}(\mbr^{\vn})} \|f\|_{L^p(\mu)} }
	 \lesssim \|a\|_{BMO_{\dr}(\mbr^{\vn})} \|f\|_{L^p(\mu)} \|g\|_{L^{p'}(\lb')},
	\end{align}
showing that
	$$ \| T_{a,b}^{(0,1)} : L^p(\mu) \rightarrow L^p(\lb)\| \lesssim \|a\|_{BMO_{\dr}(\mbr^{\vn})} \|b\|_{bmo_{\dr}(\nu)}.$$
The estimate for $T_{a,b}^{(1,0)}$ follows similarly.
\end{proof}

\subsubsection{The partial paraproducts} We work with
	$$ \mbs_{\dr}^{i_1, j_1}f := \sum_{R_1\times R_2} \sum_{\substack{P_1\in (R_1)_{i_1} \\ Q_1\in(R_1)_{j_1}}} 
	\widehat{a}_{P_1Q_1R_1}(R_2^{\ep_2}) \widehat{f}(P_1^{\ep_1}\times R_2^{\ep_2}) h_{Q_1}^{\delta_1} \otimes \frac{\unit_{R_2}}{|R_2|},$$
where $i_1, j_1$ are non-negative integers, and for every $P_1, Q_1, R_1$:
	$$ a_{P_1Q_1R_1}(x_2) \in BMO(\mbr^{n_2}) \text{ with } \|a_{P_1Q_1R_1}\|_{BMO(\mbr^{n_2})} \leq 2^{\frac{-n_1}{2}(i_1+j_1)}.$$
	
\begin{thm} \label{T:NCS-3}
Let $\mu, \lb \in A_p(\mbr^{\vn})$, $1<p<\infty$ and $\nu := \mu^{1/p}\lb^{-1/p}$. Then
	$$ \left\| [b, \mbs_{\dr}^{i_1, j_1}] : L^p(\mu) \rightarrow L^p(\lb) \right\| \lesssim  \|b\|_{bmo_{\dr}(\nu)}.$$
\end{thm}

First we need the one-weight bound for the partial paraproducts:

\begin{prop} \label{P:NCS-3-1wt}
For any $w\in A_p(\mbr^{\vn})$, $1<p<\infty$:
	\begin{equation} \label{E:NCS-3-1wt}
	\left\| \mbs_{\dr}^{i_1, j_1} : L^p(w) \rightarrow L^p(w) \right\| \lesssim 1.
	\end{equation}
\end{prop}

\begin{proof}
Let $f \in L^p(w)$ and $g \in L^{p'}(w')$, and show that $| \La \mbs_{\dr}^{i_1,j_1}f, g \Ra | \lesssim \|f\|_{L^p(w)} \|g\|_{L^{p'}(w')}$.
\begin{align}
\left\| \La \mbs_{\dr}^{i_1,j_1}f, g \Ra \right\| &\leq \sum_{R_1} \sum_{\substack{P_1\in(R_1)_{i_1} \\ Q_1\in(R_1)_{j_2}}}
	\left| \La a_{P_1Q_1R_1}, \phi_{P_1Q_1R_1} \Ra_{\mbr^{n_2}} \right| \\
	& \leq  \sum_{R_1} \sum_{\substack{P_1\in(R_1)_{i_1} \\ Q_1\in(R_1)_{j_2}}}
		\|a_{P_1Q_1R_1}\|_{BMO(\mbr^{n_2})} \|S_{\mcd_2}\phi_{P_1Q_1R_1}\|_{L^1(\mbr^{n_2})}\\
	&\leq 2^{\frac{-n_1}{2}(i_1+j_1)} \sum_{R_1} \sum_{\substack{P_1\in(R_1)_{i_1} \\ Q_1\in(R_1)_{j_2}}} \|S_{\mcd_2}\phi_{P_1Q_1R_1}\|_{L^1(\mbr^{n_2})},
\end{align}
where for every $P_1, Q_1, R_1$:
	$$ \phi_{P_1Q_1R_1}(x_2) := \sum_{R_2} \widehat{f}(P_1\times R_2) \La g, h_{Q_1}\otimes \frac{\unit_{R_2}}{|R_2|}\Ra h_{R_2}(x_2).$$
Now,
	\begin{align}
	S^2_{\mcd_2}\phi_{P_1Q_1R_1}  &= \sum_{R_2} |\widehat{H_{P_1}f}(R_2)|^2 \La |H_{Q_1}g|\Ra_{R_2}^2 \frac{\unit_{R_2}(x_2)}{|R_2|}\\
	&\leq (M_{\mcd_2} H_{Q_1}g)^2(x_2) (S_{\mcd_2} H_{P_1}f)^2(x_2),
	\end{align}
so
	\begin{align}
	&\sum_{R_1} \sum_{\substack{P_1\in(R_1)_{i_1} \\ Q_1\in(R_1)_{j_2}}} \|S_{\mcd_2}\phi_{P_1Q_1R_1}\|_{L^1(\mbr^{n_2})} 
	\leq \sum_{R_1} \sum_{\substack{P_1\in(R_1)_{i_1} \\ Q_1\in(R_1)_{j_2}}} \int_{\mbr^{n_2}} (M_{\mcd_2} H_{Q_1}g)(x_2) (S_{\mcd_2} H_{P_1}f)(x_2)\,dx_2\\
	=& \int_{\mbr^{n_2}} \int_{\mbr^{n_1}} \sum_{R_1} \sum_{\substack{P_1\in(R_1)_{i_1} \\ Q_1\in(R_1)_{j_2}}} 
		(M_{\mcd_2} H_{Q_1}g)(x_2) (S_{\mcd_2} H_{P_1}f)(x_2) \frac{\unit_{R_1}(x_1)}{|R_1|}\,dx_1\,dx_2\\
	\leq& \int_{\mbr^{\vn}} \bigg(\sum_{R_1} \bigg( \sum_{P_1\in(R_1)_{i_1}} S_{\mcd_2} H_{P_1}f (x_2)\bigg)^2 \frac{\unit_{R_1}(x_1)}{|R_1|} \bigg)^{1/2}
		\bigg( \sum_{R_1} \bigg( \sum_{Q_1\in(R_1)_{j_1}} M_{\mcd_2}H_{Q_1}g(x_2) \bigg)^2 \frac{\unit_{R_1}(x_1)}{|R_1|} \bigg)^{1/2}\,dx\\
	=& \int_{\mbr^{\vn}} [S S_{\mcd_2}]^{i_1, 0} f \cdot [SM_{\mcd_2}]^{j_1, 0}g w^{1/p} w^{-1/p}\,dx.
	\end{align}
Then, from the estimates in \eqref{E:Mixed2pSFShift}:
	\begin{align}
	\left\| \La \mbs_{\dr}^{i_1,j_1}f, g \Ra \right\| &\leq 2^{\frac{-n_1}{2}(i_1+j_1)} 
		\left\| [SS_{\mcd_2}]^{i_1, 0} f\right\|_{L^p(w)} \left\| [SM_{\mcd_2}]^{j_1, 0}g \right\|_{L^{p'}(w')}\\
	&\lesssim 2^{\frac{-n_1}{2}(i_1+j_1)}  2^{\frac{n_1i_1}{2}}\|f\|_{L^p(w)} 2^{\frac{n_1j_1}{2}} \|g\|_{L^{p'}(w')},
	\end{align}
and the result follows.
\end{proof}

\begin{proof}[Proof of Theorem \ref{T:NCS-3}]
In light of \eqref{E:NCS-3-1wt}, we only need to bound the remainder term
	$$ \mathcal{R}^{i_1, j_1}f := \Pi_{\mbs_{\dr}^{i_1, j_1}f}b - \mbs_{\dr}^{i_1, j_1} \Pi_f b. $$
The proof is somewhat similar to that of the full mixed paraproducts, in that we combine each of these terms:
	\begin{align}
	& \Pi_{\mbs_{\dr}^{i_1, j_1}f}b = \sum_{R_1\times R_2} \sum_{ \substack{ P_1\in(R_1)_{i_1} \\ Q_1 \in (R_1)_{j_1} } }
	\widehat{a}_{P_1Q_1R_1}(R_2^{\ep_2}) \widehat{f}(P_1^{\ep_1}\times R_2^{\ep_2}) 
	\bigg( \sum_{Q_2\supsetneq R_2} \La b\Ra_{Q_1\times Q_2} h_{Q_2}^{\delta_2}(R_2) h_{Q_2}^{\delta_2}(x_2) \bigg) h_{Q_1}^{\delta_1}(x_1),\\
	& \mbs_{\dr}^{i_1, j_1} \Pi_f b =  \sum_{R_1\times R_2} \sum_{\substack{P_1\in(R_1)_{i_1} \\ Q_1\in(R_1)_{j_2}}}
	\widehat{a}_{P_1Q_1R_1}(R_2^{\ep_2}) \widehat{f}(P_1^{\ep_1}\times R_2^{\ep_2}) \La b\Ra_{P_1\times R_2} h_{Q_1}^{\delta_1}(x_1) \otimes \frac{\unit_{R_2}(x_2)}{|R_2|},
	\end{align}
with a third term:
	$$T := \sum_{R_1\times R_2} \sum_{\substack{P_1\in(R_1)_{i_1} \\ Q_1\in(R_1)_{j_2}}}
		\widehat{a}_{P_1Q_1R_1}(R_2^{\ep_2}) \widehat{f}(P_1^{\ep_1}\times R_2^{\ep_2}) \La b\Ra_{Q_1\times R_2} h_{Q_1}^{\delta_1}\otimes \frac{\unit_{R_2}}{|R_2|}.$$	
As before, expanding the indicator function in $T$ into its Haar series, we may combine $T$ with $\Pi_{\mbs_{\dr}^{i_1, j_1}f}b$:
	$$ \Pi_{\mbs_{\dr}^{i_1, j_1}f}b - T = \sum_{R_1\times R_2} \sum_{\substack{P_1\in(R_1)_{i_1} \\ Q_1\in(R_1)_{j_2}}}
		\widehat{a}_{P_1Q_1R_1}(R_2^{\ep_2}) \widehat{f}(P_1^{\ep_1}\times R_2^{\ep_2}) T_b(x_2) h_{Q_1}^{\delta_1}(x_1), $$	
where
	\begin{align}
	 T_b(x_2) &= \sum_{Q_2\supsetneq R_2} \bigg( \La b\Ra_{Q_1\times Q_2} - \La b\Ra_{Q_1\times P_2} \bigg) h_{Q_2}^{\delta_2}(R_2) h_{Q_2}^{\delta_2}(x_2) \\
	 &= \sum_{Q_2\supsetneq R_2} \bigg( \sum_{P_2: R_2\subsetneq P_2 \subset Q_2} \La b, \frac{\unit_{Q_1}}{|Q_1|}\otimes h_{P_2}^{\tau_2}\Ra
	 h_{P_2}^{\tau_2}(R_2) \bigg)h_{Q_2}^{\delta_2}(R_2) h_{Q_2}^{\delta_2}(x_2).
	\end{align}
We analyze this term depending on the relationship of $P_2$ with $Q_2$.

\vspace{0.05in}
\noindent \underline{Case 1: $P_2\subsetneq Q_2$:} Then
	$$ T_b(x_2) = \sum_{P_2\supsetneq R_2}  \La b, \frac{\unit_{Q_1}}{|Q_1|}\otimes h_{P_2}^{\tau_2}\Ra h_{P_2}^{\tau_2}(R_2) \frac{\unit_{P_2}(x_2)}{|P_2|},$$
which gives the operator
	\begin{align}
	& \sum_{Q_1 \times P_2} \La b, \frac{\unit_{Q_1}}{|Q_1|}\otimes h_{P_2}^{\tau_2}\Ra h_{Q_1}^{\tau_1}(x_1) \frac{\unit_{P_2}(x_2)}{|P_2|}
	\bigg( \sum_{P_1\in (Q_1^{(j_1)})_{i_1}} \sum_{R_2\subsetneq P_2} \widehat{a}_{P_1Q_1R_1} (R_2^{\ep_2}) \widehat{H_{P_1}^{\ep_1}f}(R_2^{\ep_2})
	h_{P_2}^{\tau_2}(R_2) \bigg)\\
	=& \sum_{Q_1 \times P_2} \La b, \frac{\unit_{Q_1}}{|Q_1|}\otimes h_{P_2}^{\tau_2}\Ra h_{Q_1}^{\tau_1}(x_1) \frac{\unit_{P_2}(x_2)}{|P_2|}
	\bigg(  \sum_{P_1\in (Q_1^{(j_1)})_{i_1}} \La \Pi_{a_{P_1Q_1R_1}}^*(H_{P_1}^{\ep_1}f), h_{P_2}^{\tau_2} \Ra_{\mbr^{n_2}} \bigg) \\
	=& \pi_{b;(1,0)}^*F,
	\end{align}
where
	$$ F := \sum_{Q_1} \bigg(  \sum_{P_1\in (Q_1^{(j_1)})_{i_1}} \Pi_{a_{P_1Q_1R_1}}^*(H_{P_1}^{\ep_1}f)(x_2) \bigg) h_{Q_1}^{\delta_1}(x_1). $$
Now
	$$ \| \pi_{b;(1,0)}^*F\|_{L^p(\lambda)} \lesssim \|b\|_{bmo_{\dr}(\nu)} \|F\|_{L^p(\mu)},$$
so we are done if we can show that
	\begin{equation}\label{temp5}
	\|F\|_{L^p(\mu)} \lesssim \|f\|_{L^p(\mu)}.
	\end{equation}
Take $g \in L^{p'}(\mu')$. Then
	\begin{align}
	|\La F, g\Ra| &\leq \sum_{Q_1} \sum_{P_1\in (Q_1^{(j_1)})_{i_1}} \left| \La \Pi^*_{a_{P_1Q_1R_1}}(H_{P_1}^{\ep_1}f), H_{Q_1}^{\delta_1}g \Ra_{\mbr^{n_2}} \right|.
	\end{align}
Notice that we may write 
	$$ \La \Pi^*_{a_{P_1Q_1R_1}}(H_{P_1}^{\ep_1}f), H_{Q_1}^{\delta_1}g \Ra_{\mbr^{n_2}} = \La a_{P_1Q_1R_1}, \phi_{P_1Q_1R_1}\Ra_{\mbr^{n_2}}, $$
where 
	$$ \phi_{P_1Q_1R_1}(x_2) = \sum_{R_2} \widehat{H_{P_1}^{\ep_1}f}(R_2^{\delta_2}) \La H_{Q_1}^{\delta_1}g\Ra_{R_2} h_{R_2}^{\delta_2}(x_2). $$
Then
	\begin{align}
	|\La F, g\Ra| &\leq \sum_{Q_1} \sum_{P_1\in (Q_1^{(j_1)})_{i_1}} \|a_{P_1Q_1R_1}\|_{BMO(\mbr^{n_2})} \|S_{\mcd_2}\phi_{P_1Q_1R_1}\|_{L^1(\mbr^{n_2})}\\
	&\leq 2^{\frac{-n_1}{2}(i_1+j_1) } \sum_{R_1} \sum_{\substack{P_1\in(R_1)_{i_1} \\ Q_1\in(R_1)_{j_2}}} \int_{\mbr^{n_2}} \bigg(
	\sum_{R_2} |\widehat{H_{P_1}^{\ep_1}f}(R_2^{\delta_2})|^2 \La |H_{Q_1}^{\delta_1}g|\Ra^2_{R_2} \frac{\unit_{R_2}(x_2)}{|R_2|}
	\bigg)^{1/2}\,dx_2\\
	&\leq 2^{\frac{-n_1}{2}(i_1+j_1) } \int_{\mbr^{\vn}} \sum_{R_1} \sum_{\substack{P_1\in(R_1)_{i_1} \\ Q_1\in(R_1)_{j_2}}}
		(M_{\mcd_2}H_{Q_1}^{\delta_1}g)(x_2) (S_{\mcd_2} H_{P_1}^{\ep_1}f)(x_2) \frac{\unit_{R_1}(x_1)}{|R_1|}\,dx.
	\end{align}
The integral above is bounded by 
\begin{align}
& \int_{\mbr^{\vn}} 
		\bigg( \sum_{R_1} \bigg( \sum_{P_1\in(R_1)_{i_1}} (S_{\mcd_2}H_{P_1}^{\ep_1}f)(x_2) \bigg)^2 \frac{\unit_{R_1}(x_1)}{|R_1|} \bigg)^{1/2}
		\bigg( \sum_{R_1} \bigg( \sum_{P_1\in(R_1)_{i_1}} (S_{\mcd_2}H_{P_1}^{\ep_1}f)(x_2) \bigg)^2 \frac{\unit_{R_1}(x_1)}{|R_1|} \bigg)^{1/2}\,dx \\
&= 	\int_{\mbr^{\vn}} \bigg( [SS_{\mcd_2}]^{i_1, 0}f \bigg)	\bigg( [SM_{\mcd_2}]^{j_1, 0}g \bigg)\,dx
\leq \left\| [SS_{\mcd_2}]^{i_1, 0}f \right\|_{L^p(\mu)}  \left\| [SM_{\mcd_2}]^{j_1, 0}g \right\|_{L^{p'}(\mu')} \\
&\lesssim  2^{\frac{n_1}{2}(i_1+j_1)} \|f\|_{L^p(\mu)} \|g\|_{L^{p'}(\mu')} \text{, by \eqref{E:Mixed2pSFShift}.}
\end{align}
The desired estimate in \eqref{temp5} is now proved.

\vspace{0.05in}
\noindent \underline{Case 2(a): $P_2 = Q_2$ and $\tau_2 \neq \delta_2$:} Then
	$$T_b(x_2) = \sum_{Q_2\supsetneq R_2} \La b, \frac{\unit_{Q_1}}{|Q_1|}\otimes h_{Q_2}^{\tau_2}\Ra 
	\frac{1}{\sqrt{|Q_2|}} h_{Q_2}^{\tau_2+\delta_2}(R_2) h_{Q_2}^{\delta_2}(x_2),$$
giving rise to the operator
	\begin{align}
	& \sum_{Q_1\times Q_2} \La b, \frac{\unit_{Q_1}}{|Q_1|}\otimes h_{Q_2}^{\tau_2}\Ra \bigg(
	\sum_{P_1\in (Q_1^{(j_1)})_{i_1}} \La \Pi^*_{a_{P_1Q_1R_1}}(H_{P_1}^{\ep_1}f), h_{Q_2}^{\tau_2+\delta_2} \Ra_{\mbr^{n_2}}
	\bigg) \frac{1}{\sqrt{|Q_2|}} h_{Q_1}^{\delta_1}\otimes h_{Q_2}^{\delta_2}
	= \gamma_{b;(1,0)}F,
	\end{align}
which is handled as in the previous case.

\vspace{0.05in}
\noindent \underline{Case 2(b): $P_2 = Q_2$ and $\tau_2 = \delta_2$:} In this case, $T_b(x_2)$ gives rise to the operator
	$$ T' := \sum_{Q_1\times Q_2} \La b, \frac{\unit_{Q_1}}{|Q_1|}\otimes h_{Q_2}^{\delta_2}\Ra h_{Q_1}^{\delta_1}\otimes h_{Q_2}^{\delta_2}
	 \sum_{P_1\in (Q_1^{(j_1)})_{i_1}} \frac{1}{|Q_2|} \sum_{R_2\subsetneq Q_2} \widehat{a}_{P_1Q_1R_1} (R_2^{\ep_2}) \widehat{H_{P_1}^{\ep_1}f}(R_2^{\ep_2}).$$
Now define 
	$$ F_{\tau} := \sum_{Q_1} \bigg(  \sum_{P_1\in (Q_1^{(j_1)})_{i_1}} \Pi_{a^{\tau}_{P_1Q_1R_1}}^*(H_{P_1}^{\ep_1}f)(x_2) \bigg) h_{Q_1}^{\delta_1}(x_1), $$
just as we defined $F$ before, except now to every function $a_{P_1Q_1R_1}$ we apply the martingale transform
	$$ a_{P_1Q_1R_1} \mapsto a^{\tau}_{P_1Q_1R_1} = \sum_{R_2} \tau_{R_2}^{\ep_2} \widehat{a}_{P_1Q_1R_1}(R_2^{\ep_2}) h_{R_2}^{\ep_2} 
	\text{, where } \tau_{R_2}^{\ep_2} := \left\{ \begin{array}{l}
		+1 \text{, if } \widehat{H_{P_1}^{\ep_1}f}(R_2^{\ep_2}) \geq 0,\\
		-1 \text{, otherwise.}
	\end{array}\right.$$
Since this does not increase the $BMO(\mbr^{n_2})$ norms of the $a_{P_1Q_1R_1}$ functions,
the estimate \eqref{temp5} still holds: $\|F_{\tau\|_{L^p(\mu)}} \lesssim \|f\|_{L^p(\mu)}$.

Moreover, note that
	$$ \La \Pi^*_{a^{\tau}_{P_1Q_1R_1}}(H_{P_1}^{\ep_1}f)\Ra_{Q_2} = \sum_{R_2} 
	\underbrace{\widehat{a^{\tau}}_{P_1Q_1R_1}(R_2^{\ep_2}) \widehat{H_{P_1}^{\ep_1}f}(R_2^{\ep_2})  }_{\geq 0} \frac{|R_2\cap Q_2|}{|R_2||Q_2|}$$
and that
	$$ \pi_{b;(1,0)}F_{\tau} = \sum_{Q_1\times Q_2}  \La b, \frac{\unit_{Q_1}}{|Q_1|}\otimes h_{Q_2}^{\delta_2}\Ra
	  \sum_{P_1\in (Q_1^{(j_1)})_{i_1}} \La \Pi^*_{a^{\tau}_{P_1Q_1R_1}}(H_{P_1}^{\ep_1}f)\Ra_{Q_2} h_{Q_1}^{\delta_1}\otimes h_{Q_2}^{\delta_2}. $$
Then
	\begin{align}
	S_{\dr}^2 T' & \leq \sum_{Q_1\times Q_2} \left| \La b, \frac{\unit_{Q_1}}{|Q_1|}\otimes h_{Q_2}^{\delta_2}\Ra \right|^2
	\bigg(  \sum_{P_1\in (Q_1^{(j_1)})_{i_1}}  \frac{1}{|Q_2|} \sum_{R_2\subsetneq Q_2} | \widehat{a}_{P_1Q_1R_1} (R_2^{\ep_2}) \widehat{H_{P_1}^{\ep_1}f}(R_2^{\ep_2})| \bigg)^2
		\frac{\unit_{Q_1}}{|Q_1|}\otimes\frac{\unit_{Q_2}}{|Q_2|}\\
	& \leq \sum_{Q_1\times Q_2} \left| \La b, \frac{\unit_{Q_1}}{|Q_1|}\otimes h_{Q_2}^{\delta_2}\Ra \right|^2
	\bigg(  \sum_{P_1\in (Q_1^{(j_1)})_{i_1}}  \La \Pi^*_{a^{\tau}_{P_1Q_1R_1}}(H_{P_1}^{\ep_1}f)\Ra_{Q_2} \bigg)^2 \frac{\unit_{Q_1}}{|Q_1|}\otimes\frac{\unit_{Q_2}}{|Q_2|}\\
	&= S_{\dr}^2(\pi_{b;(1,0)} F_{\tau}).
	\end{align}
Finally, this gives us that
	\begin{align}
	 \|T'\|_{L^p(\lb)} & \simeq \|S_{\dr}T'\|_{L^p(\lb)} \leq \|S_{\dr} \pi_{b;(1,0)}F_{\tau}\|_{L^p(\lb)} 
	\simeq \|\pi_{b;(1,0)}F_{\tau}\|_{L^p(\lb)} \lesssim \|b\|_{bmo_{\dr}(\nu)}\|F_{\tau}\|_{L^p(\mu)}\\
	&\lesssim  \|b\|_{bmo_{\dr}(\nu)}\|f\|_{L^p(\mu)}.
	\end{align}
This proves that $\Pi_{\mbs_{\dr}^{i_1, j_1}f}b - T$ obeys the desired bound, and the case $T - \mbs_{\dr}^{i_1,j_1}\Pi_f b$
is handled similarly.

\end{proof}


\subsection{Proof of Theorem \ref{T:Journe}} \label{Ss:FeffProof}

Having now proved all the one-weight inequalities for dyadic shifts, we may conclude that
	$$\| \mbs_{\dr}^{\vec{i}, \vec{j}} : L^p(w) \rightarrow L^p(w) \| \lesssim 1,$$
for all $w\in A_p(\mbr^{\vn})$. 
For the cancellative shifts, this was proved in \eqref{E:2pDShift1wt}. For the non-cancellative shifts,
the first two types are simply paraproducts with symbol $\|a\|_{BMO_{\dr}(\mbr^{\vn})} \leq 1$,
while the third type, a partial paraproduct, was proved to be bounded on $L^p(w)$ in
Proposition \ref{P:NCS-3-1wt}. 

Theorem \ref{T:Journe} now follows trivially from Martikainen's representation Theorem \ref{T:MRep}:
take $f \in L^p(w)$ and $g \in L^{p'}(w')$. Then
\begin{align}
|\La Tf, g\Ra| & \leq C_T \mathbb{E}_{\omega_1} \mathbb{E}_{\omega_2} 
			\sum_{\vec{i}, \vec{j} \in \mbz^2_{+}} 2^{-\text{max}(i_1, j_1) \delta/2} 2^{-\text{max}(i_2, j_2) \delta/2}  
			\left|\La \mbs_{\dr}^{\vec{i},\vec{j}} f, g\Ra\right| \\
		& \lesssim \|f\|_{L^p(w)} \|g\|_{L^{p'}(w')} \sum_{\vec{i}, \vec{j} \in \mbz^2_{+}} 2^{-\text{max}(i_1, j_1) \delta/2} 2^{-			\text{max}(i_2, j_2) \delta/2}  \\
		& \simeq  \|f\|_{L^p(w)} \|g\|_{L^{p'}(w')}.
\end{align}

\qed


\section{The unweighted case of higher order Journ\'e commutators}\label{S:mix}

Here is the definition of the BMO spaces which are in between little BMO and product BMO. 
\begin{def}\label{definitionlpbmo}
Let $b : \mathbb{R}^{\vec{d}} \to \mathbb{C}$ with $\vec{d}=(d_1,\cdots,d_t)$. Take a partition $\mathcal{I}=\{I_s:1\le s\le l\}$ of $\{1,2,...,t\}$ so that $\dot{\cup}_{1\le s \le l} I_s=\{1,2,...,t\}$.  We say that $b \in \text{BMO}_{\mathcal{I}}(\mathbb{R}^{\vec{d}})$ if for any choices ${\boldsymbol{v}}=(v_s), v_s\in I_{s}$, $b$ is uniformly in product BMO in the variables indexed by ${v_s}$.
We call a $\text{BMO}$  space of this type a `little product BMO'. If for any $\vec{x}=(x_1,...,x_t) \in \mathbb{R}^{\vec{d}}$, we define $\vec{x}_{\hat{\boldsymbol{v}}}$ by removing those variables indexed by ${v_s}$, the little product BMO norm becomes $$\|b\|_{\text{BMO}_{\mathcal{I}}}=\max_{{\boldsymbol{v}}} \{\sup_{\vec{x}_{\hat{\boldsymbol{v}}}}\|b(\vec{x}_{\hat{\boldsymbol{v}}})\|_{\text{BMO}}\}$$ where the BMO norm is product BMO in the variables indexed by ${v_s}$. 
\end{def}

In \cite{OPS} it was proved that commutators involving tensor products of Riesz transforms in $L^p$ are a testing class for these BMO spaces:

\begin{thm}[Ou-Petermichl-Strouse]\label{corollary_riesz}
Let $\vec{j}=(j_1,\ldots,j_t)$ with $1\le j_k\le d_k$ and let for each $1\le s\le l$, $\vec{j}^{(s)}=(j_k)_{k\in I_s}$  be associated a tensor product of Riesz transforms $\vec{R}_{s,\vec{j}^{(s)}}=\bigotimes_{k\in I_s}R_{k,j_k}$; here $R_{k,j_{k}}$ are $j_k^{\text{th}}$ Riesz transforms acting on functions defined on the $k^{\text{th}}$ variable. 
We have the two-sided estimate $$\|b\|_{BMO_{\mathcal{I}}(\mathbb{R}^{\vec{d}})} \lesssim \sup_{\vec{j}}\|[\vec{R}_{1,\vec{j}^{(1)}},\ldots,[\vec{R}_{t,\vec{j}^{(t)}},b]\ldots]\|_{L^p(\mathbb{R}^{\vec{d}})\to L^p(\mathbb{R}^{\vec{d}})}\lesssim \|b\|_{BMO_{\mathcal{I}}(\mathbb{R}^{\vec{d}})}.$$
\end{thm}

It was also proved that the estimate self improves to paraproduct-free Journ\'e commutators in $L^2$, in the sense $T$ is paraproduct free $T(1\otimes \cdot)=T(\cdot\otimes 1)=T^*(1\otimes \cdot)=T^*(\cdot\otimes 1)=0$.

\begin{thm}[Ou-Petermichl-Strouse]\label{upperbd_journe}
Let us consider $\mathbb{R}^{\vec{d}}$, $\vec{d}=(d_1,\ldots ,d_t)$ with a partition $\mathcal{I}=(I_s)_{1\le s \le l}$ of $\{1,\ldots ,t\}$ as discussed before. Let $b\in {\text{BMO}}_{\mathcal{I}}(\mathbb{R}^{\vec{d}})$ and let $T_s$ denote a multi-parameter paraproduct free Journ\'e operator acting on function defined on $\bigotimes_{k\in I_s}\mathbb{R}^{d_k}$. Then we have the estimate below
\[
\|[T_1,\ldots[T_l,b]\ldots]\|_{L^2(\mathbb{R}^{\vec{d}})\to L^2(\mathbb{R}^{\vec{d}})}\lesssim \|b\|_{{\text{BMO}}_{\mathcal{I}}(\mathbb{R}^{\vec{d}})}.
\]
\end{thm}

This estimate was generalised somewhat in \cite{OP} in that the paraproduct free condition was slightly weakened, the considerations in this present text in combination with arguments from \cite{DO} and \cite{OPS} to pass to the iterated case, readily give us the following full result, for all Journ\'e operators and all $p$:

\begin{thm}
Let us consider $\mathbb{R}^{\vec{d}}$, $\vec{d}=(d_1,\ldots ,d_t)$ with a partition $\mathcal{I}=(I_s)_{1\le s \le l}$ of $\{1,\ldots ,t\}$ as discussed before. Let $b\in {\text{BMO}}_{\mathcal{I}}(\mathbb{R}^{\vec{d}})$ and let $T_s$ denote a multi-parameter Journ\'e operator acting on function defined on $\bigotimes_{k\in I_s}\mathbb{R}^{d_k}$. Then we have the estimate below
\[
\|[T_1,\ldots[T_l,b]\ldots]\|_{L^p(\mathbb{R}^{\vec{d}})\to L^p(\mathbb{R}^{\vec{d}})}\lesssim \|b\|_{{\text{BMO}}_{\mathcal{I}}(\mathbb{R}^{\vec{d}})}.
\]
\end{thm}

\noindent \textbf{Acknowledgement: } 
We are grateful to Jill Pipher and Yumeng Ou for helpful suggestions, especially in pointing out to us the connection with \cite{RF} and \cite{RF2}.
The first and second author would like to thank MSRI for their support during January 2017, as well as Michigan State University's Department of Mathematics for their hospitality in May 2017.
Finally, we are grateful to the referee for the careful reading of our paper and the very valuable suggestions to improve the presentation.


\end{document}